\documentclass{article}

\usepackage[utf8x]{inputenc}
\usepackage[english]{babel}
\usepackage{graphicx}
\usepackage{amscd}
\usepackage{amsfonts}
\usepackage{latexsym}
\usepackage{amssymb,amsmath}
\usepackage[usenames]{color}

\usepackage{amscd}
\usepackage{psfrag}

\usepackage{amsthm}

\usepackage{tikz-cd}
\usepackage{faktor}

\theoremstyle{plain}
\newtheorem{teor}{Theorem}[section]

\newtheorem{conj}{Conjecture}[section]

\newtheorem{cor}{Corollary}[section]
\newtheorem{prop}[teor]{Proposition}
\newtheorem{lemma}{Lemma}[section]

\newtheorem{remark}{Remark}[section]

\theoremstyle{definition}
\newtheorem{defi}[teor]{Definition}

\def\C{\mathbb{C}}
\def\R{\mathbb{R}}

\def\Z{\mathbb{Z}}

\def\D{\mathbb{D}}

\def\m{M_\Gamma}
\def\F{\mathcal F}
\def\K{\mathcal K}
\def\S{\mathcal S}
\newcommand{\HH}{\mathbb{H}}

\usepackage{hyperref}
\hypersetup{
    colorlinks=true,
    linkcolor=blue,
    citecolor=blue,
    urlcolor=cyan,
    pdfpagemode=UseOutlines,
    bookmarksopen=true
}

\begin{document}

\title{Tessellating the discreteness locus for the modular mating family of correspondences}

\author{Shaun Bullett, Luna Lomonaco, Arcelino Lobato do Nascimento,\\
Pedro Ivan Suarez Navarro, Miguel Ratis Laude}

\date{August, 2026}

\maketitle
\begin{abstract}

The modular Mandelbrot set $\m$, the connectedness locus of the modular mating family of $2:2$ holomorphic correspondences $\F_a$ on the Riemann sphere, is homeomorphic to the classical Mandelbrot set $M$. The Klein combination locus $\K$ (the `discreteness locus' of the family $\F_a$) is a pinched neighborhood of $\m$ in the $a$-plane, pinched at the root point. We construct a canonical map $\Psi$ from $\K\setminus\m$ into the hyperbolic plane $\HH$, inspired by the construction of Douady and Hubbard for their celebrated conformal bijection $\Phi:\widehat \C\setminus M\to \widehat\C\setminus\overline{\D}$, and we prove that $\Psi$ is analytic. This map $\Psi$ induces a tessellation of $\K\setminus\m$ by pulling back a tessellation of $\HH$ invariant under the modular group. We develop a series of conjectures concerning the structure of $\K$, its boundary, and $\Psi(\K)\subset \HH$.

\end{abstract}

\section{Introduction}\label{intro}

The family $\F_a: z \to w$ of $(2:2)$ holomorphic correspondences on the Riemann sphere $\widehat\C$, defined by the relation $$\left(\frac{az+1}{z+1}\right)^2+\left(\frac{az+1}{z+1}\right)\left(\frac{aw-1}{w-1}\right)+\left(\frac{aw-1}{w-1}\right)^2=3,$$ was introduced by the first author with Christopher Penrose in \cite{BP1}, where it was shown to contain examples of {\it matings} between quadratic maps and the modular group $PSL(2,\Z)$. The {\it modular Mandelbrot set} $\m$ is the connectedness locus of this family in the plane of the parameter $a$. In \cite{BL3} the first two authors of the current article proved the conjecture in \cite{BP1} that $\m$ is homeomorphic to the classical Mandelbrot set, $M$. 

Below, we show that there exists a (pinched) neighbourhood $\mathcal D$ of $\m$ in parameter space (the $a$-plane), and a canonical dynamically defined analytic map $\Psi$ from $\mathcal D\setminus\m$ into the hyperbolic plane $\HH$. Any tessellation of $\HH$ which is invariant under the standard action of $PSL(2,\Z)$ pulls back via $\Psi$ to a tessellation of $\mathcal D\setminus\m$ which one can think of as a framework encoding the combinatorial structure of $\m$, analogously to the way that Douady and Hubbard \cite{DH1,DH2} encoded the combinatorial structure of the classical Mandelbrot set $M$ via the grid of {\it external rays} and {\it equipotentials} defined by their conformal bijection $\Phi: \widehat\C \setminus M \to \widehat\C \setminus \overline \D$.

Initially, in Section \ref{Psi_def}, we define $\Psi$ as a continuous map from $\mathcal D\setminus\m$ into $\HH$, where $\mathcal D:=\{a:|a-4|<3\}$. For $a\in \mathcal{D}\setminus\m$ the correspondence $\F_a$ has a {\it standard} fundamental domain $\Delta_a^{st}$ for its action on the regular set $\Omega(\F_a)$ (see Section \ref{prelim}). There is a conformal bijection $\varphi_a$ from $\Delta_a^{st}$ to a fundamental domain (with specified vertices) for the action of $PSL(2,\Z)$ on $\HH$. We show that when $a\in \mathcal D\setminus \m$, the critical value $v_a$ of $\F_a$ lies in a copy $\F_a^{-n}(\Delta_a^{st})$ of $\Delta_a^{st}$, and that $\varphi_a$ extends equivariantly to this copy. We define $\Psi(a)$ to be $\varphi_a(v_a)$.

There is an obvious parallel with Douady and Hubbard's construction \cite{DH1} of a canonical conformal bijection $\Phi:\widehat\C\setminus M \to \widehat\C\setminus \D: c \to \varphi_c(c)$, which was the key step in their proof that the classical Mandelbrot set $M$ is connected. For any quadratic polynomial $q_c:z \to z^2+c$, the point at infinity is a super-attracting fixed point, and as a consequence there is a neighborhood of $\infty$ on which $q_c$ is conjugate to $z\to z^2$ via the B\"ottcher conjugacy $\varphi_c$ (one defines $\varphi_c(z)$ to be $\lim_{n\to \infty} (q_c^{\circ n}(z))^{1/2^n}$ for a suitable branch of the $2^n$th root function, which is easy to define in practice: see \cite{DH1,DH2} for details). The B\"ottcher conjugacy is defined on a large neighborhood of $\infty$, indeed on the whole of the complement $\widehat{\C}\setminus K(q_c)$ of the filled Julia set $K(q_c)$ if the latter is connected. The stroke of genius of Douady and Hubbard was to observe that when $K(q_c)$ is not connected, the domain of definition of the B\"ottcher conjugacy $\varphi_c$ contains the critical value $c$ of $q_c$ and that the assignment $\Phi:c \to \varphi_c(c)$ defines a conformal bijection between the complement $\C\setminus M$ of the Mandelbrot set and the complement $\C\setminus \overline{\D}$ of the closed unit disc.

Just as in Douady and Hubbard's construction, our `B\"ottcher conjugacy' $\varphi_a$ (between the correspondence $\F_a$ on its regular set $\Omega(\F_a)$, and  the group $PSL(2,\Z)$ on $\HH$) is defined on the whole of $\widehat{\C}\setminus \Lambda(\F_a)$ when the limit set $\Lambda(\F_a)$ of the correspondence $\F_a$ is connected, in other words when $a \in \m$, but for $a\notin \m$ there is a significant difference between their situation and ours: we can only define $\varphi_a$ for values of $a$ such that $\F_a$ is `discrete' in an appropriate sense, and has a `fundamental domain' $\Delta_a$. This restricts us to a particular pinched neighborhood $\K$ of $\m$, known as the {\it Klein combination locus} (see Section \ref{prelim} below for a definition). For most values of $a$ outside $\K$, the dynamical behaviour of the correspondence $\F_a$ appears to be `chaotic'; there does not exist a partition into a regular set $\Omega(\F_a)$ and limit set $\Lambda(\F_a)$,
so there is no `fundamental domain' $\Delta_a$ and no `B\"ottcher map' $\varphi_a$. The Klein combination locus $\K$ has properties akin to those of the {\it discreteness locus} in the moduli space of representations in $PSL(2,\C)$ of the free product of two finite cyclic groups. 

It is straightforward to extend the domain of definition of $\Psi$ from $\mathcal D\setminus\m$ to the whole of $\K\setminus\m$, but $\K$ is not simply-connected (it has a puncture point at $a=1$), with the consequence that the extension of $\Psi$ may be multivalued. Indeed in a detailed study in Section \ref{extension} we prove that extending $\Psi$ along paths $\ell^\mathcal{U}$ and $\ell^\mathcal{L}$ from $a=1+\delta$ to $a=-1$, above and below the puncture point $a=1$, leads to different answers for $\Psi(-1)$ (Proposition \ref{Psi-1}). We conclude Section \ref{extension} with a discussion of the set of all fundamental domains for the action of the modular group $PSL(2,\Z)$ on $\HH$ which have boundary a Jordan curve through a standard set of vertices: these are the candidates for the image $\varphi_a(\Delta_a)$ of a fundamental domain $\Delta_a$ for $\F_a$.

In Section \ref{anal} we prove that (every branch of) $\Psi$ is (locally) analytic on $\K$ (Theorem \ref{psi_analytic}). While this statement is analogous to Douady and Hubbard's, the proof is more difficult in our case as we do not have a single point to take the role of `organizing center' played by infinity for their map $\Phi$. What we have instead is a conformal isomorphism $\varphi_a$ between the orbit space $\Omega(\F_a)/\langle \F_a \rangle$ and the modular surface $\HH/PSL(Z)$ for each $a\in \K$, but it turns out that conformal rigidity of the modular surface is just sufficient to deliver the analyticity of $\Psi$.

In the alternative coordinate $Z=(az+1)/(z+1)$ the correspondence $\F_a$ can be expressed as $J_a\circ Cov_0^Q$ where $J_a$ is the conformal involution of the sphere which has fixed points $Z=1$ and $Z=a$, and $Cov_0^Q$ is the (deleted) covering correspondence of the polynomial $Q(Z)=Z^3-3Z$ (see Section \ref{prelim}). The vertices of the tessellation of $\K\setminus\m$ correspond to values of $a\in \K$ where $\F_a$ satisfies a {\it critical relation}: the grand orbit of the critical point of $\F_a$ contains either the fixed point $Z=a$ of the involution $J_a$ or the fixed point $Z=\infty$ of the covering correspondence $Cov_0^Q$. In Section \ref{crit-rel} we examine the role of these critical relations both inside and outside $\K$. In particular we exhibit a computer plot, Figure \ref{kleinC}, of critical relation parameters outside $\K$: this suggests that such parameters may be dense outside $\K$, and also provides the first indication of the shape of $\K$, and of the possible critical relation points on its boundary $\partial\K$.

In Section \ref{structure_K} we propose a series of conjectures concerning the structures of $\K$ and $\partial\K$. The Klein combination locus $\K$ is an analogue of the (better understood) `discreteness locus' in the moduli space of representations in $PSL(2,\C)$ of the free product of a cyclic group of order $2$ and a cyclic group of order $3$, but $\K$ exhibits intriguing differences from the latter: see Section \ref{global}. In Section \ref{scenario}, we outline a possible scenario for $\Psi(\K)\subset \HH$, compatible with what we know, and what we conjecture, about $\K$ and $\partial\K$. 

\begin{figure}
 \begin{center}
 \includegraphics[width=3.8cm]{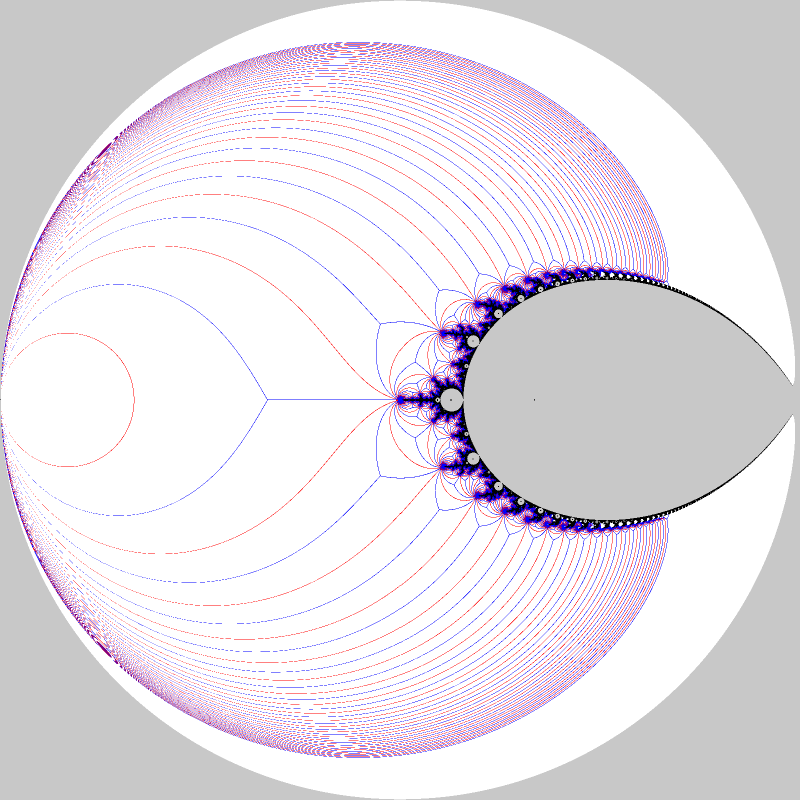}\hspace{0.1cm}
 \includegraphics[width=3.8cm]{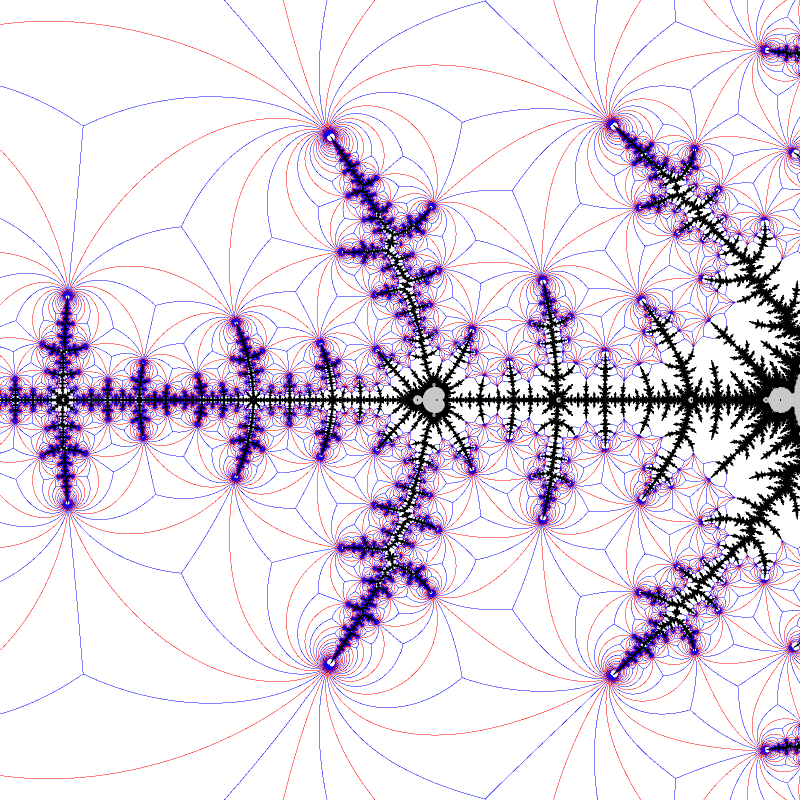} 
 \includegraphics[width=3.8cm]{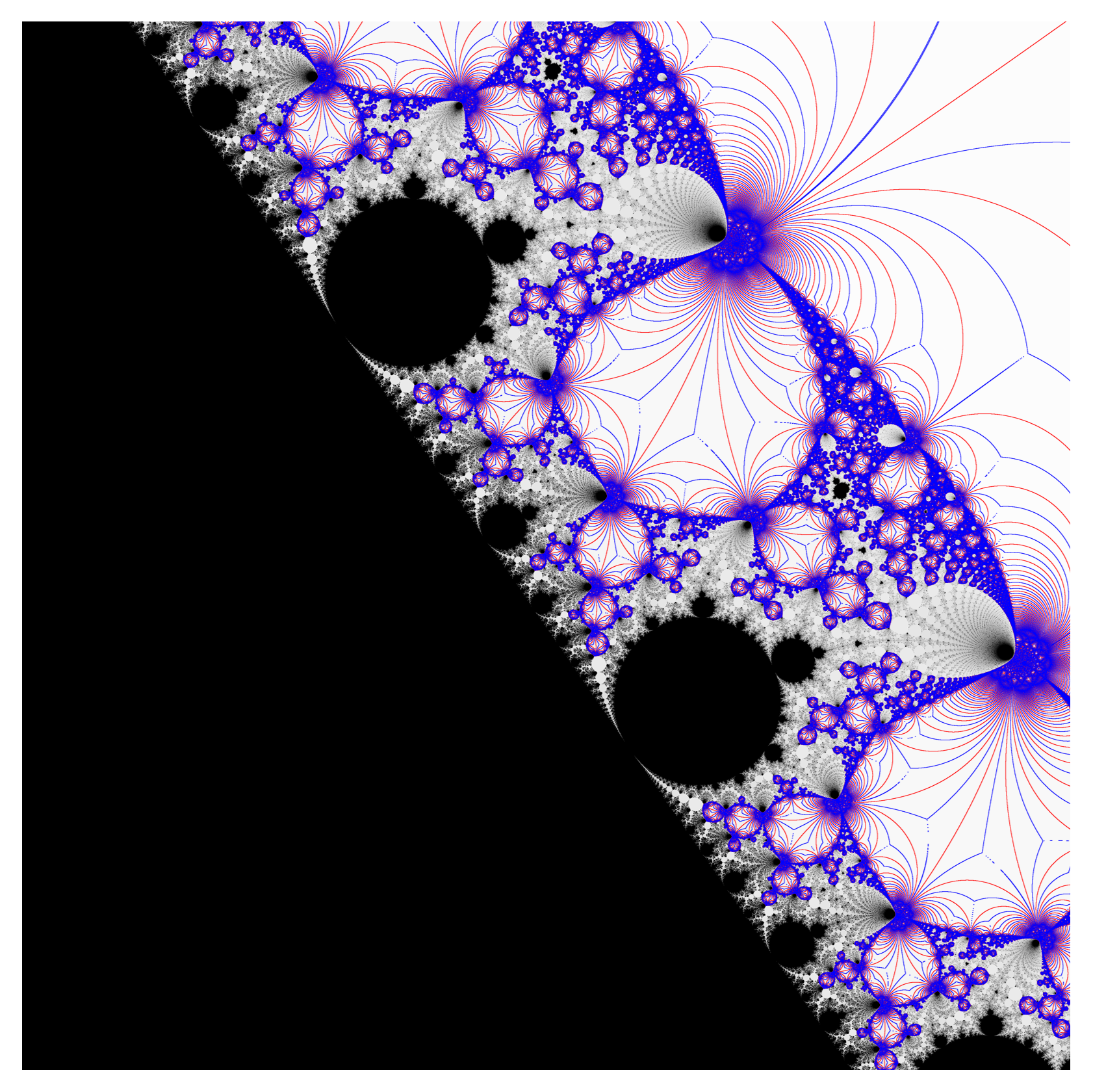} \hspace{0.1cm}
 
 \includegraphics[width=3.8cm]{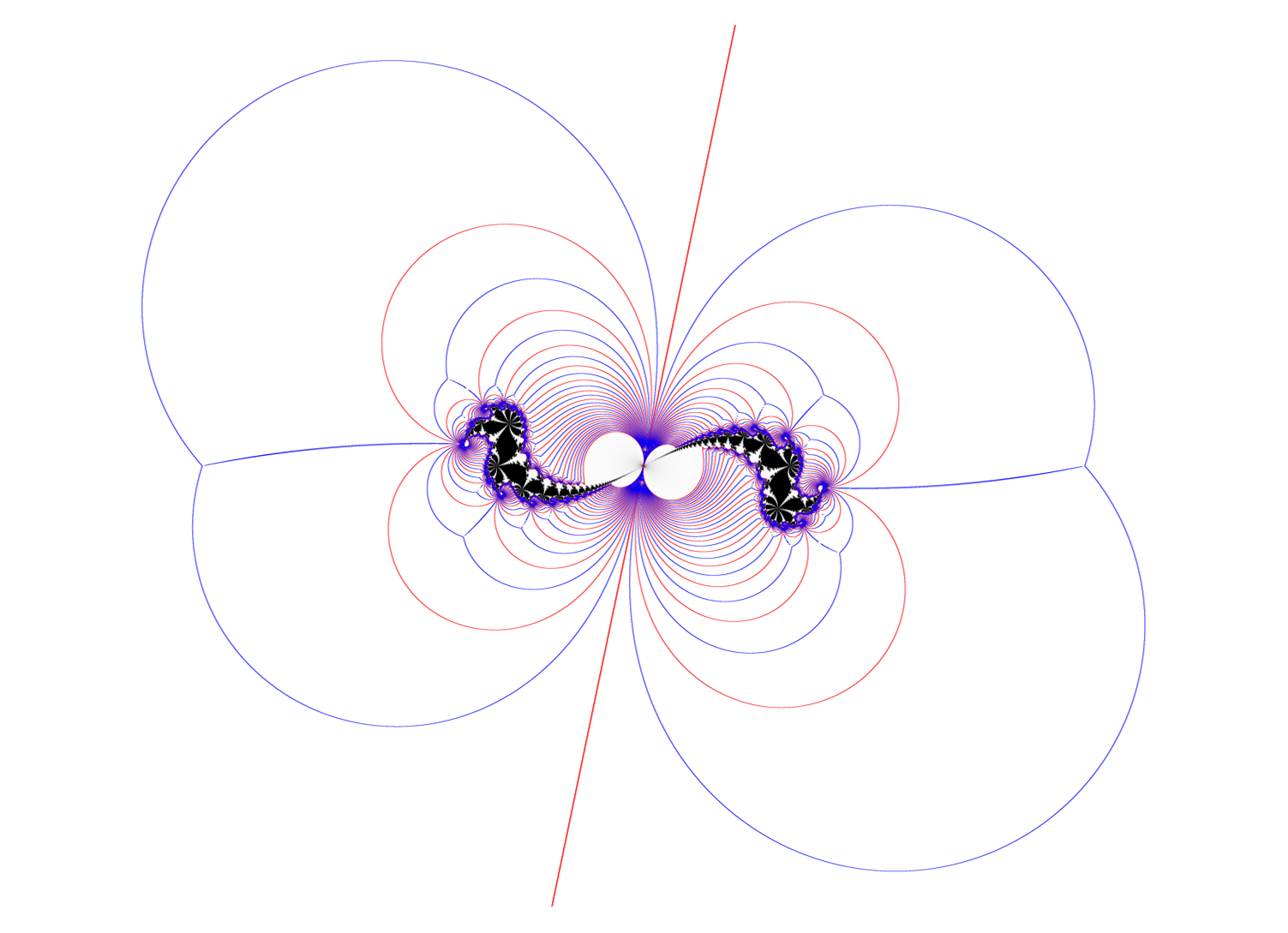}
 \includegraphics[width=3.8cm]{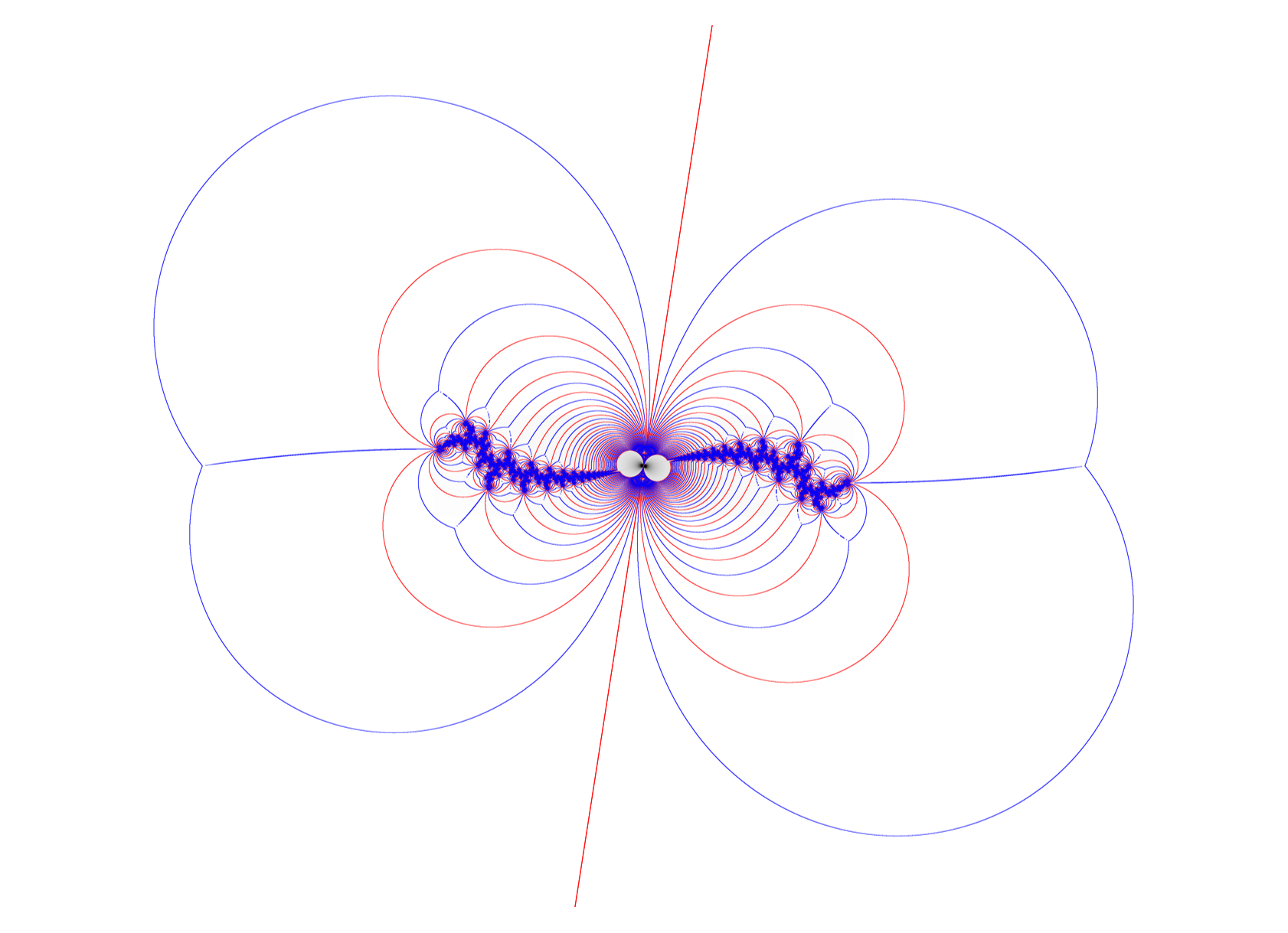} \hspace{0.1cm}
 \includegraphics[width=3.8cm]{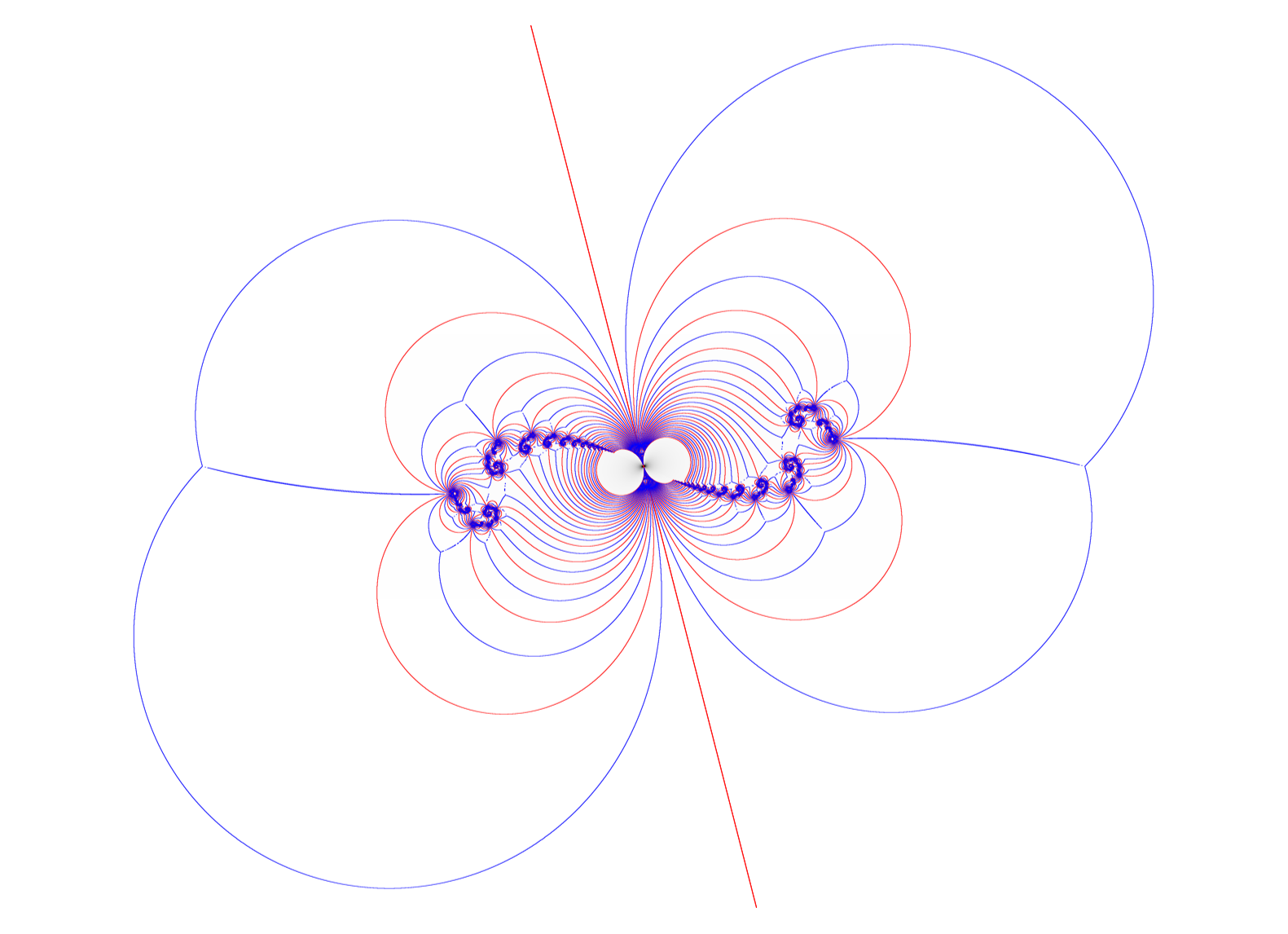}
 \end{center} 
\caption{Top line: Plots of the modular Mandelbrot set $M_\Gamma$ and the surrounding modular tiling pulled back from $\HH$ by $\Psi$. In the left-hand plot the round white disc has centre $a=4$ and radius $3$; the zoom in the middle has at its centre the small copy of $M_\Gamma$ on the real axis which corresponds to period $3$; the right-hand plot is a zoom in `elephant valley' near the root $a=7$ of $M_\Gamma$ (despite appearances there are narrow channels connecting the `sea' to each pair of circular `lakes'). Bottom line: three examples of plots of the dynamical plane of correspondences $\F_a$. In the first two the `B\"ottcher map' $\varphi_a$ conjugates the action of $\F_a$ on the complement of the limit set $\Lambda(\F_a)$ to the action of the modular group $PSL(2,\Z)$ on the upper half plane $\HH$. In the third example there is only a partial `B\"ottcher map' as $\Lambda(\F_a)$ is disconnected: at critical points (and pre-images) there are tiles which are ramified double covers of tiles in the modular tessellation of $\HH$. }\label{mandelcorr}

\end{figure}

Figure \ref{mandelcorr} contains a plot of an approximation to the pull-back to $ \mathcal D \setminus \m$ of a tessellation of $\HH$ invariant under $\Gamma=PSL(2,\Z)$. While the positions of the vertices of tiles are accurately plotted by the algorithm, this is not the case for the edges between them. See Appendix I (Section \ref{plots}) for a discussion of the algorithms used in plots, and the difficulty in plotting the `true' tessellation.

In Appendix II (Section \ref{triangles}) we exhibit a sequence of isolated discrete members of the family $\F_a$, which lie outside $\K$ on the real axis in the $a$-plane between $-\infty$ and $-1$, accumulating at $-1$ (a `parabolic' boundary point of $\K$). These echo a similar phenomenon in the moduli space of representations of $C_2*C_3$ in $PSL(2,\C)$, where each `parabolic' point on the boundary of the discreteness locus is the accumulation point of a `harmonic' sequence of isolated non-faithful discrete representations outside the locus. 


\section{Preliminaries and notation}\label{prelim}

Further details of the material in this section can be found in \cite{BL1}.

\subsection{Covering correspondences for rational maps}
For a degree $d$ rational map $Q:\widehat\C \to \widehat\C$ we define the {\it covering correspondence} of $Q$ to be the $(d:d)$ holomorphic correspondence
$$Cov^Q:\ z \to w \ \Leftrightarrow \ Q(w)=Q(z),$$
and the {\it deleted covering correspondence} to be the $(d-1:d-1)$ holomorphic correspondence
$$Cov_0^Q:\ z \to w \ \Leftrightarrow \ \frac{Q(w)-Q(z)}{w-z}=0.$$

\subsection{Fundamental domains for covering correspondences}
A {\it fundamental domain} for $Cov^Q$ is a maximal open set $\Delta_Q\subset \C$ such that $Q$ is injective on $\Delta_Q$. Thus the images of $\Delta_Q$ under
$Cov^Q$ are disjoint, and when their boundaries are also included these images cover $\widehat \C$.  
We shall require all our fundamental domains to be topological discs, bounded by piecewise-smooth Jordan curves which are smooth where they pass through simple critical points, and have 
tangents meeting at angles $2\pi/n$ at critical points of higher order.

\subsection{The family $\F_a$ and the Klein combination condition}\label{Klein_condition}

Via the change of coordinates 
$$Z=\frac{az+1}{z+1}, \ W=\frac{aw-1}{w-1}$$ 
the family of correspondences $\F_a$ is M\"obius conjugate to the family:
$$Z \to W\  \Leftrightarrow\ J{_a}(W) \in Cov_0^Q(Z)$$
where $J_a$ is the (M\"obius) involution which has fixed points $Z=a$ and $Z=1$, that is to say 
$$J_a(Z)=\frac{(a+1)Z-2a}{2Z-(a+1)},$$
and $Q$ is the cubic polynomial $Q(Z)=Z^3-3Z$. Thus in the coordinate $Z$ we can write $\F_a$ as the composition 
$$\F_a=J_a\circ Cov_0^Q.$$
In this article we will mostly use the coordinate $Z$, though many of our computer plots will be displayed in the $z$-plane, where the symmetry $J_a$ between forward and backward iteration of $\F_a$ is particularly apparent since $J_a(z)=-z$.

The critical points of $Q$ are $Z=1,-1$ and $\infty$. The correspondence $Cov_0^Q$ maps $Z=1$ to $\{1,-2\}$, it maps $-1$ to $\{-1,2\}$ and it maps $\infty$ to $\infty$. So the point $Z=1$ is a fixed point of both $J_a$ and $Cov_0^Q$, hence of the branch of $\F_a$ which fixes it. This fixed point is parabolic, since the branch of $\F_a$ which fixes it has derivative $1$ there. The {\it Klein combination condition} for a correspondence in the family $\F_a$ is that there exist a pair of fundamental domains, $\Delta_{J_a}$ for $J_a$ and $\Delta_Q$ for $Cov_Q$, such that $\Delta_{J_a}\cup \Delta_Q=\widehat\C \setminus\{1\}$ (equivalently $(\widehat\C\setminus\Delta_{J_a})\cap(\widehat\C\setminus\Delta_Q)=\{1\}$).

\subsection{The standard Klein combination pair for $a \in \D(4,3)$}\label{standard_domains}

The inverse image $Q^{-1}(L)$ of the real line segment $L=(-\infty,-2]\subset \C$ consists of $L$ together with a curve $C'$ which crosses the real axis at $Z=+1$ and runs off towards infinity asymptotically to the directions of argument $\pm \pi/3$. The component of $\widehat \C \setminus C'$ to the right of $C'$ is the {\it standard fundamental domain} $\Delta_Q^{st}$ for $Cov^Q$. When $a\in \mathcal D:= \D(4,3)$, the round circle through $Z=1$ and $Z=a$, centred on the real axis, meets $C'$ only at $Z=1$ (see \cite{BL1} Proposition 3.4). For such a value of $a$ we define the set of all points of $\widehat{\C}$ outside this round circle to constitute the {\it standard fundamental domain} $\Delta^{st}_{J_a}$ for $J_a$, and we call $(\Delta^{st}_{J_a},\Delta^{st}_Q)$ the {\it standard} Klein combination pair.

\subsection{Transversality at the parabolic fixed point} 

The boundaries of a Klein combination pair $(\Delta_{J_a},\Delta_Q)$ will necessarily be tangent at the parabolic fixed point $Z=1$. We shall require them to satisfy the additional condition that the direction of the common tangent be transverse to the attracting-repelling axis of $\F_a$ at $Z=1$. The {\it standard} Klein combination pair for $\F_a$ satisfy this condition when $|a-4|<3$, but not for values of $a$ such that $|a-4|=3$ (see \cite{BL1}, Proposition 3.9).

Let $\Delta_a$ denote the intersection of a Klein combination pair $(\Delta_{J_a},\Delta_Q)$ for $\F_a$. Whether or not we have the transversality condition, the Riemann sphere is partitioned dynamically into two completely invariant subsets, the {\it regular set} (which is open),
$$\Omega(\F_a,\Delta_a):=\bigcup_{n=-\infty}^{+\infty} \F_a^n(\overline \Delta_a\setminus\{1\}),$$
and the {\it limit set} (which is closed) $\Lambda(\F_a,\Delta_a)=\Lambda_+(\F_a,\Delta_a)\cup \Lambda_-(\F_a,\Delta_a)$ where
$$ \Lambda_+(\F_a,\Delta_a)=\bigcap_{n=0}^\infty \F_a^n(\overline\Delta_Q),
\ \mbox{\rm and }\ \Lambda_-(\F_a,\Delta_a)=J_a(\Lambda_+(\F_a,\Delta_a)).$$
Note that the Klein combination condition implies that
$$\Lambda_+(\F_a,\Delta_a)\cap\Lambda_-(\F_a,\Delta_a)=\{1\}.$$

The transversality condition on the pair $(\Delta_{J_a},\Delta_Q)$ ensures that the partition of $\widehat{\C}$ into $\Omega(\F_a)$ and $\Lambda(\F_a)$ is unique, and independent of the choice of such $(\Delta_{J_a},\Delta_Q)$ (see \cite{BL1}, Remark 3.1, for what can go wrong without this condition). It is proved in \cite{BL1}, Proposition 3.9, that given any Klein combination pair one can make an arbitrarily small perturbation of their boundaries, arbitrarily close to the parabolic fixed point, and ensure that the perturbed pair satisfy the condition.

\subsection{The Klein combination locus $\K$ and the modular Mandelbrot set $\m$}

The {\it Klein combination locus} $\K$ is defined to be the set of values of the parameter $a\in \C$ such that there exists a Klein combination pair $(\Delta_{J_a},\Delta_Q)$ for $\F_a$ satisfying the conditions we have required above, including the transversality condition. This locus has properties analogous to properties of the locus in the moduli space of representations in $PSL(2,\C)$ of a free product $C_p*C_q$ of finite cyclic groups for which the representation is both discrete and faithful: indeed it is this discreteness locus for the family of $(2:2)$ correspondences which have graphs the union of the graph of an element of $PSL(2,\C)$ of order $2$ with the graph of an element of order $3$. In particular (though we have not proved this)  it appears that for a dense set of values of the parameter $a$ in $\C\setminus \overline\K$ the space $\widehat\C/\langle\F_a\rangle$ of grand orbits does not have any Hausdorff component. We think of $\C\setminus \overline\K$ as the set of parameter values for which the dynamical behaviour of $\F_a$ is ``generically chaotic'', though there are isolated values of $a$ in this set for which $\widehat\C/\langle\F_a\rangle$ {\it does} have a Hausdorff component. These are analogous to the non-faithful discrete representations in the group case: see Appendix II (Section \ref{triangles}).

The {\it modular Mandelbrot set} $\m$ is the set of values of the parameter $a$ such that $\Lambda_-(\F_a)$ is connected (or equivalently $\Lambda_+(\F_a)$ is connected). When $a\in\m$ the correspondence $\F_a$ is a {\it mating} between a parabolic quadratic rational map and the modular group (\cite{BL1}). On a pinched neighbourhood of $\Lambda_-(\F_a)$ the restriction of $\F_a$ is a $2$-to-$1$ `pinched quadratic-like' map with critical point $Z=-1$ (denoted by $c_a$). 
The corresponding {\it critical value} is $v_a=J_a(2)$.

\subsection{The orbifold $\Omega(\F_a)/\langle \F_a \rangle$}

For every $a\in \K$, the grand orbit space $\Omega_a/\langle \F_a \rangle$ is naturally identified with the quotient space
$$\S:=(\overline\Delta_a\setminus\{P\})/{\rm (side-pairing)},$$ 
where $P$ is the parabolic fixed point ($Z=1$), and the side pairings are defined by $J$ and $Cov$ on the appropriate boundaries of $\Delta_a$. As an orbifold, $\S$ has the conformal structure of a  punctured sphere with two marked cone points $Q$ and $R$, of angles $\pi$ and $2\pi/3$ respectively: the puncture point corresponds to $P$, and the cone points to fixed points of $J$ and $Cov$ respectively. 

At the heart of our construction in this note is the observation that $\S$, being a sphere with these three marked points, is uniquely conformally isomorphic to the quotient of the upper half-plane $\HH$, by the modular group $\Gamma=PSL(2,\Z)$, generated by the pair of elliptic automorphisms of $\HH$:
$$\rho:z \to -1/(z+1)$$ 
of order $3$, which has fixed point $Q'=(-1+i\sqrt{3})/2$, and
$$\sigma: z \to -1/z$$ 
of order $2$, which has fixed point $R'=i$. Recall that as an abstract group, $PSL(2,\Z)$ is the free product of the cyclic groups generated by $\rho$ and $\sigma$. 

{\bf Note.}{\it In this article we adopt the convention of referring to points in the hyperbolic plane by their coordinates in the upper half-plane model $\HH$, and to orientation-preserving automorphisms of the hyperbolic plane by elements of the group $PSL(2,\R)$, but our illustrations will mostly be presented in the more convenient Poincar\'e disc model.}


\section{A canonical continuous map $\Psi:\mathcal D \setminus \m \to \HH$}\label{Psi_def}

When $a\in \m$ the correspondence $\F_a$ is a {\it mating} between a parabolic rational map $f$ and $\Gamma$ in the sense that the dynamical space $\widehat \C$ of $\F_a$ is partitioned (invariantly under $\F_a$) into a copy $\Omega(\F_a)$ of $\HH$ carrying the standard action of  $\Gamma$ via a map we call the {\it B\"ottcher map} $\varphi_a:\Omega(\F_a) \to \HH$, and a closed set $\Lambda(\F_a)$ consisting of the one-point union $\Lambda_+(\F_a)\cup \Lambda_-(\F_a)$ of two copies (forward and backward) of the connected filled Julia set of $f$. Although $\Lambda(\F_a)$ becomes disconnected when $a$ moves outside $\m$, we will show that for $a$ close enough to $\m$ we can still partially define a {\it B\"ottcher map} $\varphi_a: U \to \HH$ on an open subset $U$ of $\Omega(\F_a):=\widehat\C \setminus \Lambda(\F_a)$ containing a fundamental domain for $\F_a$. In particular, when $a\in \mathcal D\setminus \m$, the partial B\"ottcher map $\varphi_a$ is well-defined on the {\it critical value} $v_a$ of $\F_a$. In this section our main result will be that the assignment $a \to \varphi_a(v_a)$ defines a canonical continuous map from $\mathcal D \setminus \m$ into $\HH$ (Proposition \ref{DHmap}).
 
\begin{defi}\label{canonical}
Let $N$ be a (pinched) neighbourhood of $\m$ in $\K$. We shall say that a continuous map $\psi:N\setminus\m\to \HH$ is {\it canonical} if for each $a\in N\setminus\m$ the image $[v_a]$ of the critical value $v_a\in \Omega_a$ in the orbifold $\Omega_a/\langle \F_a\rangle$ is mapped to the image $[\psi(v_a)]$ of $\psi(v_a)$ in $\HH/PSL(2,\Z)$ by the (unique) conformal isomorphism $$\Theta_a: \Omega_a/\langle \F_a\rangle\to \HH/PSL(2,\Z)$$ which sends the marked points in $\Omega_a/\langle \F_a\rangle$ (the order $2$ and $3$ cone points, and the puncture point) to the corresponding marked points in $\HH/PSL(2,\Z)$.
\end{defi}

Thus if $\psi$ is {\it canonical}, then for any given $a\in N\setminus \m$ we only have a discrete set of possible values for $\psi(a)$, the points of the orbit of $PSL(2,\Z)$ which sits above $\Theta([v_a])\in \HH/PSL(2,\Z)$. The question of defining a canonical map $\Psi$ becomes the question of continuously choosing an inverse branch of the surjection $\HH \to \HH/PSL(2,\Z)$ at the point $\Theta([v_a])\in\HH/PSL(2,\Z)$) as $a$ varies: particular care is required at branch points, that is to say points of $\HH$ stabilised by non-trivial subgroups of $PSL(2,\Z)$.

\subsection{The position of the critical value $v_a$ of $\F_a$}

We start with a lemma which holds for {\it every} Klein combination pair $(\Delta_J,\Delta_Q)$, not just for the standard pair. \\

\noindent
{\bf Notation.} In the lemma below, and elsewhere, $\Delta_a$ will denote the intersection $\Delta_J\cap\Delta_Q$, and $\Delta'_a$ will denote $\overline \Delta_a\setminus \{P\}$
(where $P$ is the parabolic fixed point).\\

\begin{lemma}\label{fund_lemma}
For every $a\in \K\setminus\m$, and every $\Delta_a$ the intersection of a Klein combination pair $(\Delta_J,\Delta_Q)$ for $\F_a$, the forwards critical point $c_a$ of $\F_a$ lies in $\bigcup_{n\ge 1} \F_a^{-n}(\Delta'_a\cup J_a\Delta'_a)$, and the critical value $v_a\in \bigcup_{n\ge 0} \F_a^{-n}(\Delta'_a\cup J_a\Delta'_a)$.
\end{lemma}

\begin{proof}
The Klein combination theorem tells us that, modulo the boundary $\partial\Delta_a$ of $\Delta_a$ and its forwards and backwards images, $\Omega_a$ is the union of the sets $W(\Delta_a)$ obtained when $W$ runs through all finite words in $Cov_0$ and $J_a$. To be precise:
$$\widehat{\C}\setminus \Lambda_a = \bigcup_{n\in \Z} \F_a^n(\Delta'_a\cup J_a\Delta'_a).$$
Note also that
$$\overline\Delta_Q=\Lambda_{+,a}\cup\bigcup_{n\ge 0} \F_a^n(\Delta'_a\cup J_a\Delta'_a)$$
and (though we do not need it in this proof)
$$\overline\Delta_J=\Lambda_{a,-}\cup\bigcup_{n\ge 0} \F_a^{-n}(\Delta'_a\cup Cov_0(\Delta'_a)).$$

Write $c_a$ for the forwards critical point of $\F_a$. The point $c_a$ has two images under $Cov_0$, namely $c_a$ itself and another point, $d_a$ (in the $Z$-coordinate $c_a=-1$ and $d_a=2$). The correspondence $\F_a$ maps $c_a$  to the pair $\{J_a(c_a),J_a(d_a)\}$. The critical value $v_a$ is $J_a(d_a)$ (the value of $\F_a(c_a)$ which has unique pre-image $c_a$).

The point $c_a$ cannot lie in $\Delta_Q$, since $c_a$ is a critical point of $Q$, and a transversal for $Q$ cannot contain a critical point in its interior. Thus
$$c_a\in \widehat\C\setminus\Delta_Q = \Lambda_{a,-}\cup\bigcup_{n \ge 1}\F_a^{-n}(\Delta'_a\cup J_a\Delta'_a).$$
Now consider $d_a$. Necessarily $d_a\in \overline\Delta_Q$, since $\Delta_Q$ is the interior of a transversal to $Q$ that does not contain $c_a$, and $\{c_a,d_a\}$ is an orbit of $Cov$. Observing that
$$\overline\Delta_Q=(\Delta'_a\cup J_a\Delta'_a)\cup  J_a(\widehat\C\setminus\Delta_Q),$$
we see that either $d_a\in (\Delta'_a\cup J_a\Delta'_a)$, and hence so does $v_a=J_a(d_a)$, or else $d_a\in J_a(\widehat\C\setminus\Delta_Q)$, in which case
$$v_a \in \widehat\C\setminus\Delta_Q=\Lambda_{a,-}\cup\bigcup_{n\ge 1} \F_a^{-n}(\Delta'_a\cup J_a\Delta'_a).$$
Thus in every case we have
$$v_a \in \Lambda_{a,-}\cup\bigcup_{n\ge 0} \F_a^{-n}(\Delta'_a\cup J_a\Delta'_a).$$
If $a\in\K\setminus\m$ we know that $c_a\notin \Lambda_{a,-}$ and $v_a\notin \Lambda_{a,-}$. So in this case
$$c_a \in \bigcup_{n\ge 1} \F_a^{-n}(\Delta'_a\cup J_a\Delta'_a)\  \   {\rm  and} \ \  v_a\in \bigcup_{n\ge 0} \F_a^{-n}(\Delta'_a\cup J_a\Delta'_a).$$
\end{proof}

We note that it follows from the Lemma that when $a\in \K\setminus\m$ the forwards critical point $c_a$ of $\F_a$ cannot be in the image $\F_a^n(J(c_a))$ of the backwards critical point $J(c_a)$ for any $n>0$. We will see later that this is no longer necessarily the case when $a$ lies on the {\it boundary} of $\K$ and  the Klein combination domains become pinched.\\

Specialising to the case that $a\in \mathcal D =\{a:|a-4|<3\}$ and the standard Klein combination pair, we next note a simple necessary and sufficient condition for $v_a$ to lie in the `tile' $J_a(\overline\Delta^{st}_a)$, rather than the other tiles listed in the Lemma above:

\begin{lemma}\label{small_disc}
Let $a\in \mathcal D \setminus \m$. Then $v_a\in J(\overline\Delta^{st}_a)\ \ \Leftrightarrow\ \ a\in \overline{\D(3/2,1/2)}.$
\end{lemma}

\begin{proof}
We work in the $Z$-coordinate in the dynamical plane, so $v_a$ is the point $J_a(2)$. By definition $\Delta^{st}_{J_a}=\widehat\C\setminus \overline D_a$, where $D_a$ is the open round disc which has centre on the real axis and boundary passing through $1$ and $a$. The critical value $J_a(2)$ is in $J_a(\overline\Delta^{st}_a)$ if and only if $2\in\widehat\C\setminus D_a$, in other words if and only if $a\in \overline\D(3/2,1/2)$. 
\end{proof}

\subsection{The canonical map $\Psi:\mathcal D \setminus\m \to \HH$}\label{canon} 

\begin{figure}
\begin{center}
\scalebox{.40}{\includegraphics{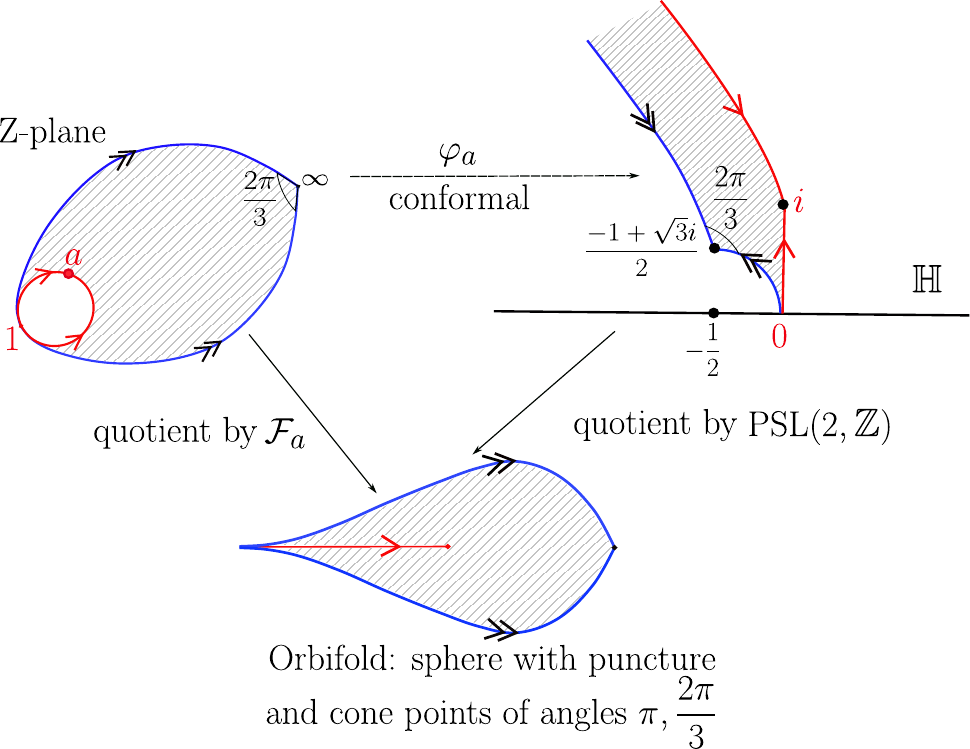}}
\caption{The partial B\"ottcher map $\varphi_a$}\label{Boett}
\end{center}
\end{figure}

Recall that the grand orbit space $\Omega_a/\langle \F_a\rangle$ is uniquely conformally isomorphic to $\HH/PSL(2,\Z)$ (Section \ref{prelim}). For $a\in \mathcal D$ we define a partial B\"ottcher map 
$$\varphi_a:\overline\Delta^{st}_a\cup J_a(\overline\Delta^{st}_a)\to \HH$$ 
as follows (Figure \ref{Boett}). For each $Z \in \overline\Delta_a^{st}\setminus\{P\}$ choose a path $\ell$ in $\overline\Delta_a^{st}$ from the fixed point $a$ of $J_a$ to $Z$. Project to the orbit space $\Omega_a/\langle \F_a\rangle$ and lift to a path $\ell'$ in $\HH$ starting at the fixed point $i$ of $\sigma$ (there are two possible lifts, in opposite directions: we make the `obvious' choice, 
the one which for $\ell$ a path from $a$ to $\infty$ lifts to $\ell'$ a path from $i$ to $(-1+i\sqrt{3})/2$). Define $\varphi_a(Z)$ to be the end point of $\ell'$. This recipe gives us a well-defined 
conformal lift $\varphi_a$ of $\Delta^{st}_a\setminus\{P\}$ to $\HH$. We can extend $\varphi_a$ to $J_a(\Delta_a^{st}\setminus\{P\})$ by equivariance with respect to $J_a$ on $\Omega(\F_a)$ and $\sigma$ on $\HH$. Moreover we can iteratively extend $\varphi_a$ to $\F_a^{-n}(\Delta^{st} \cup  J(\Delta^{st})\setminus\{P\})$ up to and including $v_a$ (we can even reach $c_a$ from one side) using equivariance with respect to $\F_a$ on $\Omega(\F_a)$ and $(\sigma\rho, \sigma\rho^{-1})$ on $\HH$. \\

\begin{defi}\label{Psi}
For each $a\in\mathcal D\setminus \m$ define $\Psi(a):=\varphi_a(v_a)\in \HH$.
\end{defi}

\begin{prop}\label{DHmap}
The assignment $a \to \varphi_a(v_a)$ defines a canonical continuous map $$\Psi: \D(4,3) \setminus \m \to \HH.$$
\end{prop} 

\begin{proof}
The map $\Psi$ is continuous, since $z \to \varphi_a(z)$ is jointly continuous in $a$ and $z$ and $a \to v_a$ is continuous. The homeomorphism $\varphi_a$ was extended dynamically to include $v_a$ in its domain, so $\Psi$ is {\it canonical} (i.e. satisfies Definition \ref{canonical}).
\end{proof}

\begin{remark}

(i) 
The conformal map $\varphi_a$ sends $\Delta_a^{st}$ onto a Jordan fundamental domain $\Delta_{a,\Gamma}$ for $\Gamma:=PSL(2,\Z)$ on $\HH$ (such that $\Delta_{a,\Gamma}$ has standard vertices).

(ii)
We shall prove later (Theorem \ref{psi_analytic}) that $\Psi$ is analytic. The projection $\HH\to \HH/PSL(2,\Z)$ has an explicit analytic expression as the $j$-function. It is an intriguing question as to whether there might be an explicit expression for $\Psi$ involving this ubiquitous classical function.

\end{remark}

\begin{prop}\label{reala}
The image of the real interval $\{a:1 < a < 4\}$ under $\Psi$ is the subset of the unit semicircle in $\HH$ from $z=(1+i\sqrt{3})/2\in \HH$ to $z=-1\in \partial\HH$. In particular $\lim_{a\to 4}\Psi(a)=-1\in \partial\HH$, $\Psi(3)=(-1+i\sqrt{3})/2$, $\Psi(2)=i$ and $\lim_{a\to 1}\Psi(a)=(1+i\sqrt{3})/2$.
\end{prop}

\begin{proof}
For each $a\in (1,4)$, the set $\Delta^{st}_a$ is symmetric under complex conjugation. Hence its image under $\varphi_a$, in the half-plane model of $\HH$, is our standard fundamental domain for $PSL(2,\Z)$,  invariant under reflection in the unit semicircle. As $v_a$ lies on the real $Z$-axis, the point $\Psi(a)=\varphi_a(v_a)$ lies on this semicircle. It is easily verified that  the values of $\Psi(2)$ and $\Psi(3)$ are as stated, since at $a=2$ we have $v_a=J_a(2)=2$, so $\varphi_a(v_a)$ is the fixed point of $\sigma$, and at $a=3$ the critical value $v_a$ is fixed by $Cov_0^Q$, so $\varphi_a(v_a)$ is the fixed point of $\rho$. That $\lim_{a\to 4}\Psi(a)=-1\in \partial\HH$ follows from the direct calculations that $\lim_{a\to 4}v_a=-2$ and $\F_4(-2)=1$, so that $\Psi$ extends continuously (on $\R$) to the Misiurewicz point $a=4 \in \m$. 

To prove that $\lim_{a\to 1}\Psi(a)=(1+i\sqrt{3})/2 \in \HH$, we will first change the dynamical space coordinate $Z$ to a coordinate in which we can more easily evaluate $\varphi_a$ as $a$ tends to $1$. We define $\zeta:=1/z$, where $Z=(az+1)/(z+1)$. In the coordinate $\zeta$ on the Riemann sphere, $J_a$ is $\zeta \leftrightarrow -\zeta$ with fixed points $\zeta=0$ and $\zeta=\infty$ (just as in the $z$-coordinate), the double fixed point $Z=\infty$ of $Cov^Q$  becomes $\zeta=-1$, the single fixed point of $Cov^Q$ becomes $\zeta=\infty$, and the critical value $Z=J_a(2)$ of $\F_a$ becomes $\zeta=2-a$. Thus when $a$ is real and in the interval $(1,2)$, our $\Delta_a$ is an isosceles triangle in the $\zeta$-plane with long side the imaginary axis and its other two sides meeting at $\zeta=-1$ at angles $\pm\pi/3$ to the real axis, and the critical value $v_a$ is the point $\zeta=2-a \in J(\Delta_a)$. It is now a triviality to observe that in the limit as $a$ goes to $1$, the point $v_a$ goes to $\zeta=+1$, and hence that in the hyperbolic metric on the orbifold $\Omega(\F_a)/\langle \F_a\rangle$ the critical value $v_a$ converges to $\zeta=1$. It follows that under the conformal map $\varphi_a$ from $\Delta_a\cup J_a(\Delta_a)$ to our standard $\Delta\cup\sigma(\Delta)$ for $PSL(2,\Z)$ on $\HH$, the image $\varphi_a(v_a)$, (in other words $\Psi(a)$) converges to $(1+i\sqrt{3})/2$ as $a$ goes to $1$.
\end{proof}


\section{Extending $\Psi$ to a larger domain}\label{extension}

The method by which we constructed a partial B\"ottcher map $\varphi_a$ for a given $a$ in $\mathcal D \setminus\m$ works equally well for any $a\in \K\setminus \m$: we may take any Klein combination pair $(\Delta_{J_a}, \Delta_Q)$ for $\F_a$, in place of the standard pair. Different choices of Klein combination pair for the same parameter $a$ may give us different values for $\varphi_a(v_a)$ (see Proposition \ref{Psi-1}, Section \ref{second_extension}).

\begin{defi}\label{psi} 
We adopt the notation $\psi(a)$ for $\varphi_a(v_a)$ when we are using Klein combination pairs other than the standard pair $(\Delta_Q^{st},\Delta^{st}_{J_a})$ to define $\varphi_a$, and reserve the notation $\Psi$ for the map defined in the previous section, Definition \ref{Psi}, or for its extension when this is known to be unique.
\end{defi}

A priori, we can only be sure that different choices of Klein combination pairs will give $\psi(a)$ in the same $PSL(2,\Z)$ orbit on $\HH$, since $v_a$ projects to the same point as $\psi(a)$ in $\Omega(\F_a)/\langle \F_a \rangle \cong \HH/PSL(2,\Z)$. But if we have two Klein combination pairs with homotopic boundary curves, and $v_a$ remains in the same tile (i.e. copy of $\Delta_a$) throughout the homotopy, it is clear that $\varphi_a(v_a)$ remains fixed, since the fibre $PSL(2,\Z)$ is discrete. By considering the equivariant extension of $\varphi_a$ which we constructed in order to define $\Psi$ (Definition \ref{Psi}), this argument generalises to prove:

\begin{prop} \label{def_fixed_a}
Given $a\in \K\setminus\m$, if $(\Delta^t_{J_a}, \Delta^t_{Cov})$ is a homotopy of Klein combination pairs for $\F_a$ such that $v_a \in \bigcup _{n\ge 0}\F_a^{-n}((\Delta_a^t)'\cup J_a(\Delta_a^t)')$ for all $t\in[0,1]$, then $\psi(a):=\varphi_a(v_a)$ remains fixed throughout the homotopy. $\qed$
\end{prop}

We can allow the parameter $a$ to vary and apply the same argument to prove:
 
\begin{cor}\label{continuous_deformation}
Given any path $a_t\in\K\setminus\m$, $t\in [0,1]$, and a homotopy of Klein combination pairs $(\Delta^t_{J_{a_t}}, \Delta^t_{Cov})$ for $\F_{a_t}$ satisfying the condition that $v_{a_t} \in \bigcup _{n\ge 0}\F_{a_t}^{-n}(\Delta_{a_t}'\cup J_{a_t}\Delta_{a_t}')$ for all $t\in[0,1]$, the map $t \to \psi(a_t):=\varphi_{a_t}(v_{a_t})$ is continuous. $\qed$
\end{cor}

\subsection{A continuous extension of $\Psi$ from $\mathcal D$ to a larger pinched neighborhood of $\m$}

Let $D_1$ be the open disc in the $a$-plane having center $4+\sqrt{3}i$ and radius $2\sqrt{3}$, and let $D_2$ the open disc with center $4-\sqrt{3}i$ and radius $2\sqrt{3}$ (Figure \ref{D1andD2}, left).

\begin{figure}
\begin{center}
\includegraphics[width=9.0cm]{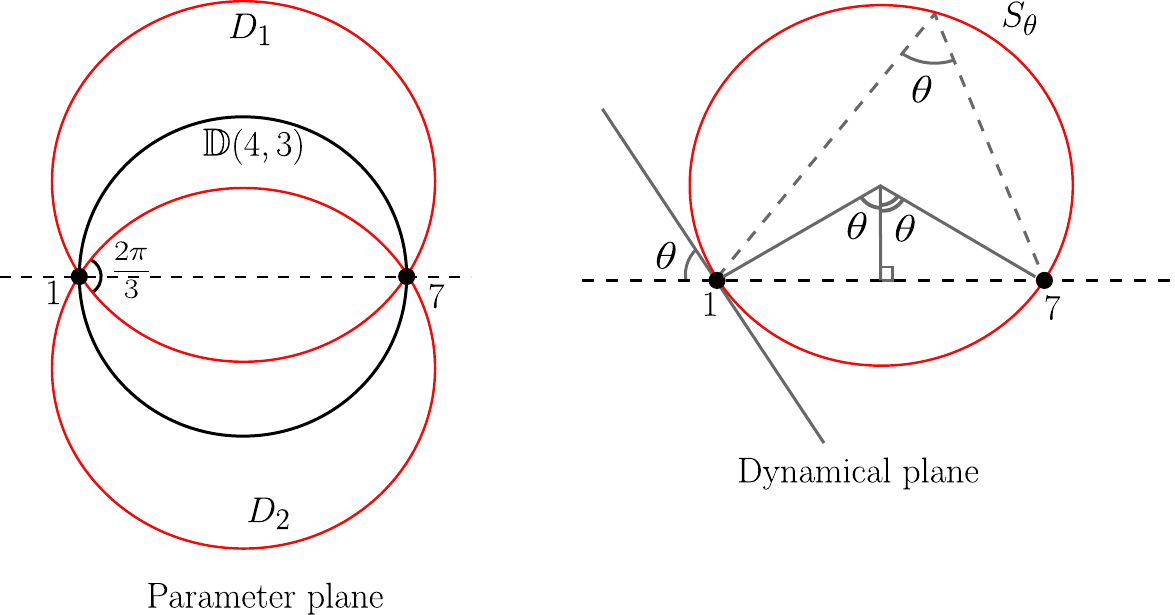}
\caption{The discs $D_1$ and $D_2$, and the circle $S_\theta$}\label{D1andD2}
\end{center}
\end{figure}

\begin{prop}\label{D1D2}
$\Psi$ extends (uniquely) to a canonical continuous map $D_1\cup D_2 \to \HH$.
\end{prop}

\begin{proof}

Let $a$ be any point of $D_1$ in the upper half-plane. Denote by $S_\theta$ the circle through $z=1$ and $z=7$ which has tangent at $z=1$ at angle $\theta$ to the horizontal, as shown on the right in
Figure \ref{D1andD2}. Then $a$ lies on $S_{\theta'}$ for a unique $\theta'\in (\pi/3,\pi/2]$, and for all $\theta\in(\pi/3,\theta')$ it lies in the open disc bounded by $S_\theta$. 

Let $\theta\in(\pi/3,\theta')$. In the Appendix to \cite{BL3} it is proved that when $\pi/3\le\theta\le\pi/2$, the circle $S_\theta$ only intersects its image $Cov_0(S_\theta)$ at $z=1$. As a consequence there is a Jordan curve $C_\theta$ which passes through $z=+1$, separating $S_\theta$ from $Cov_0(S_\theta)$ and bounding a fundamental domain $\Delta_Q^\theta$ for $Cov_0$. Let $S_{\theta_a}$ denote the circle through $z=1$ and $z=a$ which is tangent to $S_\theta$ at $z=1$. Then the unbounded component of $\C\setminus S_{\theta_a}$ is a fundamental domain $\Delta^\theta_{J_a}$ for $J_a$ which has boundary meeting that of $\Delta_Q^\theta$ only at $z=1$. To show that $(\Delta_Q^\theta,\Delta^\theta_{J_a})$ is a Klein combination pair it remains to check the transversality condition at the parabolic point. The argument of the parabolic axis at $z=1$ is $(\bar{a}-7)/(\bar{a}-1)$ (Proposition 3.6 in \cite{BL1}). But $Arg((a-7)/(a-1))<\theta$ by elementary geometry (since $a$ is in the disc bounded by $S_\theta$), so $Arg((\bar{a}-7)/(\bar{a}-1))<-\theta$, in other words the parabolic axis is transverse to the tangent to $S_\theta$.

We can choose the pair  $(\Delta_Q^\theta,\Delta^\theta_{J_a})$ to have boundaries moving continuously respect to $a$ and $\theta$. Let $\mathcal{U}$ denote the upper half-plane. Then the method of the preceding Section, together with Corollary \ref{continuous_deformation}, will give us an associated continuous canonical map $D_1 \cap \mathcal{U} \to \HH$, which, by Corollary \ref{def_fixed_a}, restricts to the function $\Psi$ of Definition \ref{Psi} at points $a\in \D(4,3)\cap \mathcal{U}$. Similarly $\Psi$ may be extended to $D_2$.
\end{proof}

\subsection{Extending $\Psi$ to $\D(0,1)$}\label{second_extension}

Recall that $\F_a$ is undefined at the parameter value $a=1$. 

\begin{prop}\label{D01}
$\D(0,1) \subset \K$, and $\Psi$ has extensions to $\D(0,1)$, both clockwise and anticlockwise around the puncture $a=1$.
\end{prop}

\begin{figure}
\begin{center}
\includegraphics[width=5cm]{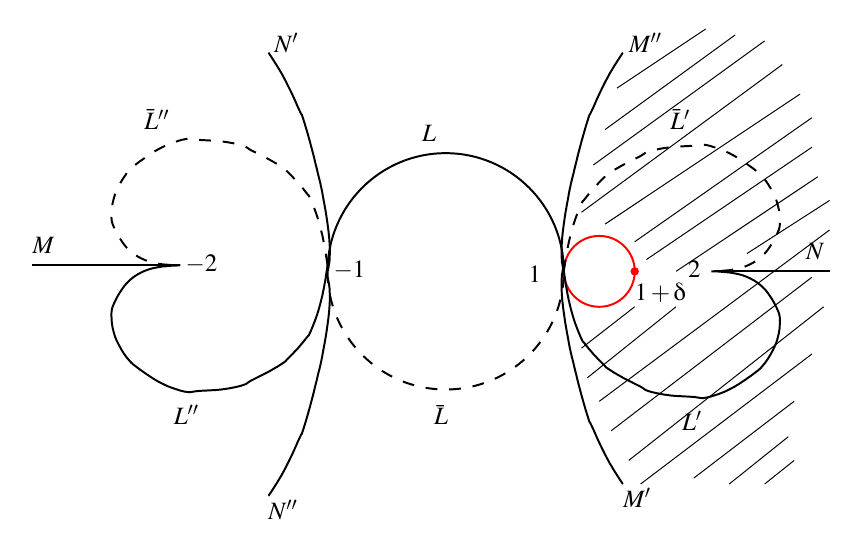}\hskip1cm\includegraphics[width=5cm]{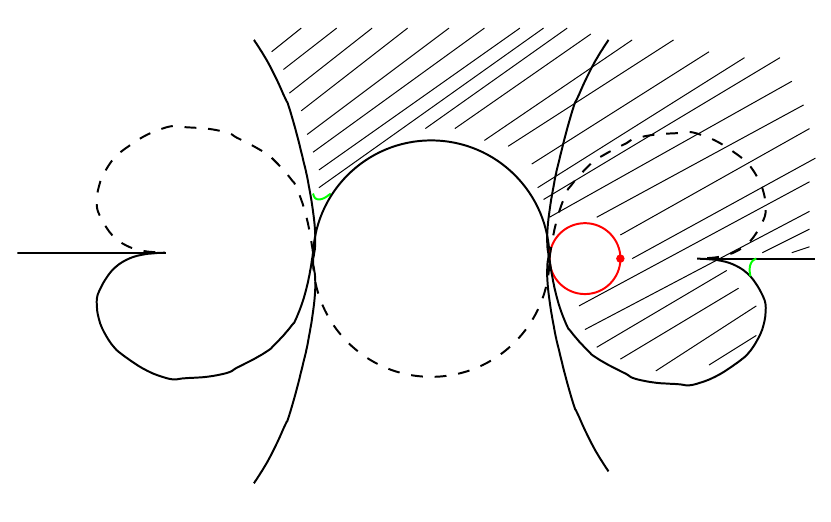}
\includegraphics[width=5cm]{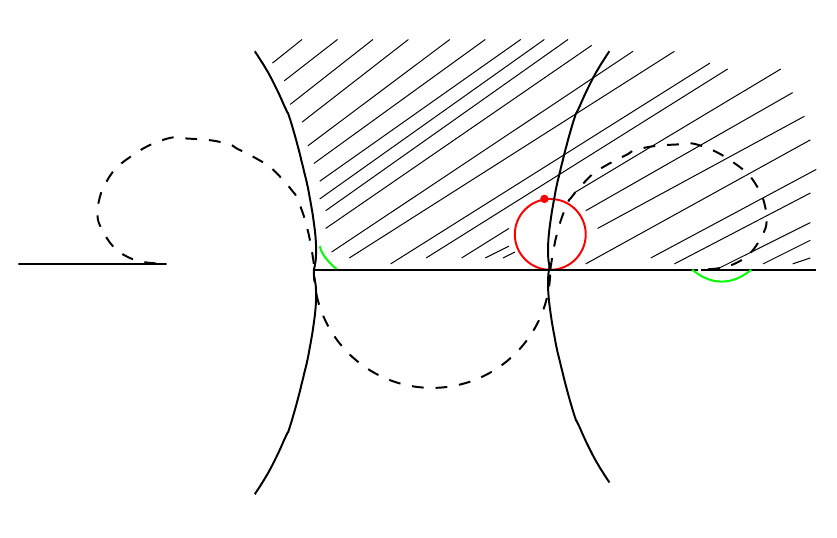}\hskip1cm\includegraphics[width=5cm]{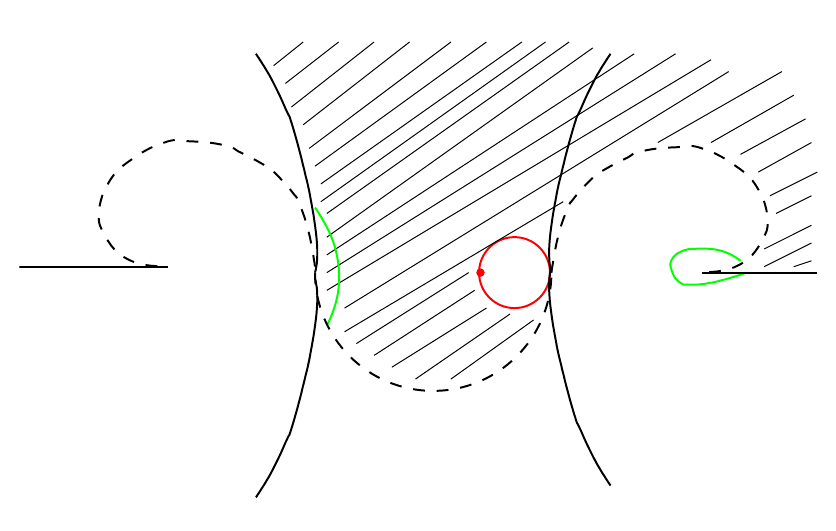}
\end{center}
\caption{\small The homotopy of domains in Proposition \ref{D01}. Top line: we move the boundary of $\Delta_Q$ continuously, keeping $a=1+\delta$ fixed (and $\Delta_{J_a}$ fixed).  Bottom line: we then move $L$ downwards, through arcs of circles, first to a straight line and then to the lower semicircle, simultaneously rolling $\Delta_{J_a}$ around $1$ until $a$ reaches $1-\delta$. (The green arcs are explained in Remark \ref{deform_rem}.)}\label{deform_pic}
\end{figure}

\begin{proof}

Let $b$ be any point in $\D(0,1)$. Starting at $a=1+\delta$ for any real $\delta$ with $0<\delta<1$ we will move $a$ continuously to $b$ along an anticlockwise path around the puncture point $a=1$, and continuously deform the standard Klein combination pair for $\F_a$ to a Klein combination pair for  $\F_b$. We proceed in three steps, all of which have obvious analogues for a clockwise path around $a=1$. As in Figure \ref{deform_pic}: $L$ and $\bar L$ denote the upper and lower halves of the unit circle; $L',L''$ are the $Cov_0^Q$-images of $L$, and $\bar L',\bar L''$ those of $\bar L$; $M=(-\infty,-2]$ and $M',M''$ are its $Cov_0^Q$ images; $N=[2,\infty)$ and $N',N''$ are its $Cov_0^Q$-images.

{\it Step 1: move $\partial \Delta_Q$, keeping $a$ and $\partial \Delta_{J_a}$ fixed.}  In the space bounded by $M'\cup L'\cup N'$, homotop  $M'$ rel. $\{0,\infty\}$ from its initial position to $L'\cup N$. Simultaneously (under the action of $Cov^Q$), the line $M''$ moves round to $L\cup N'$. The region bounded by $L'\cup N\cup L\cup N'$ is now our fundamental domain for $Cov^Q$.

{\it Step 2: move $L$ and $\partial \Delta_{J_a}$.} Homotop $L$ rel. $\{-1,1\}$, through arcs of circles, to the lower unit semi-circle $\bar{L}$, simultaneously moving $a$ anticlockwise around the circle of radius $\delta$ centred on $a=1$ so as to remain within the moving $\Delta_Q$. At the end of this step $\partial\Delta_Q$ is bounded by $N$, the arc complex conjugate to $L'$, $L$ and $N'$; $\partial\Delta_{J_a}$ is the circle which has centre $1-\delta/2$ and radius $\delta/2$, and $a$ is the point $1-\delta$.

{\it Step 3: keeping $\partial \Delta_Q$ fixed, move $a$ to $b$, and adjust $\partial\Delta_{J_a}$ to become a circle through $1$ and $b$ centred on the real axis.} This is straightforward as $Cov^Q$ now contains the unit circle.

\end{proof}

\begin{remark}\label{deform_rem}
In this proof, to keep the description of the procedure as simple as possible, our $\Delta_Q$ has the critical point $Z=-1$ on its boundary from Step 2 onwards. This causes no problem, but can be avoided: $\Delta_Q$ can be modified to put $Z=2$ into its interior, and hence move $Z=-1$ outside its closure, by adding to $\Delta_Q$ the region bounded by a small loop around $Z=2$, and keeping $\Delta_Q$ a transversal to $Q$ by subtracting from $\Delta_Q$ the $Cov_0^Q$-image of this region which contains  $Z=-1$ on its boundary (see the green arcs in Figure \ref{deform_pic}).
\end{remark}

\begin{cor}
$-1\in \partial\K$.
\end{cor}

\begin{proof}
By Proposition \ref{D01}, $\D(0,1)\subset\K$, so $-1$ is in the closure of $\K$. To see that $-1\notin \K$, observe that for any $a$, if $(\Delta_Q,\Delta_{J_a})$ is a Klein combination pair for $\F_a$, the point $a$ is in the interior of $\Delta_Q$, since $a$ is on the boundary of $\Delta_{J_a}$, and every point of this boundary, except $Z=1$, is in the interior of $\Delta_Q$. But $-1$ is a critical point of $Q$, so cannot be in the interior of a transversal to $Q$.
\end{proof}

Let us denote by $\Psi^{\mathcal{U}}$ the extension of $\Psi$ to $\D(0, 1)$ obtained by moving anticlockwise around $a = 1$, and by $\Psi^{\mathcal{L}}$ the extension obtained by moving clockwise.

\begin{lemma}\label{cut_line}
(i) As $a$ approaches the puncture point $1$ of $\K$, from any direction, $\Psi(a)$ converges to $(1+i\sqrt{3})/2\in \HH$ (the fixed point of $\sigma\rho\sigma$).
(ii) There exists $g\in PSL(2,\Z)$ such that $\Psi^\mathcal{L}(a)=g\Psi^\mathcal{U}(a)$ for all $a\in \D(0,1)$.
(iii) The element $g$ of part (ii) is one of the set $\{Identity, \sigma\rho\sigma,\sigma\rho^{-1}\sigma\}\subset PSL(2,\Z)$ .
\end{lemma}

\begin{proof}
    (i) The proof is identical to the last part of the proof of Proposition \ref{reala}: we change from the coordinate $Z$ to the coordinate $\zeta=1/z$ defined there, and we see that when $a$ converges to $1$ from any direction, the critical value $v_a$ converges to $\zeta=+1$. Since this is true on the quotient orbifold, the lift $\varphi_a(v_a)$ converges to $(1+i\sqrt{3})/2\in\HH$.
    (ii) This is immediate from the continuity of $\Psi^\mathcal{U}$ and of $\Psi^\mathcal{L}$ and the fact that both are canonical maps.
    (iii) From (i), $g\in PSL(2,\Z)$ stabilizes $(1+i\sqrt{3})/2\in \HH$, so it is one of these three elements.
\end{proof}

Which of the three candidates listed in (iii) of Lemma \ref{cut_line} is the element $g$? It would suffice to know the value of $\Psi^\mathcal{L}$ and $\Psi^\mathcal{U}$ at a single $a\in \D(0,1)$. As we do not have this information we turn to the boundary point $-1$.

\begin{figure}
\includegraphics[width=8.5cm]{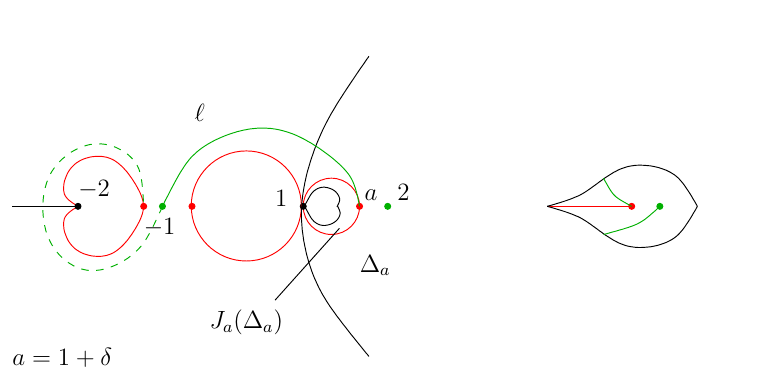}\includegraphics[width=3.5cm]{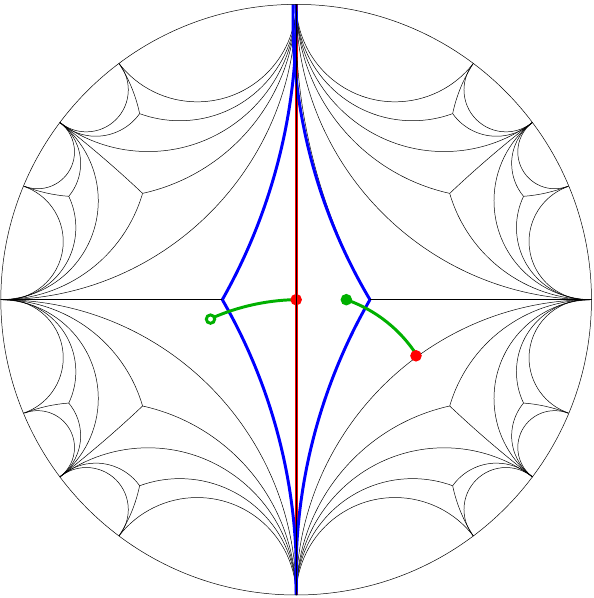} 

\includegraphics[width=8.5cm]{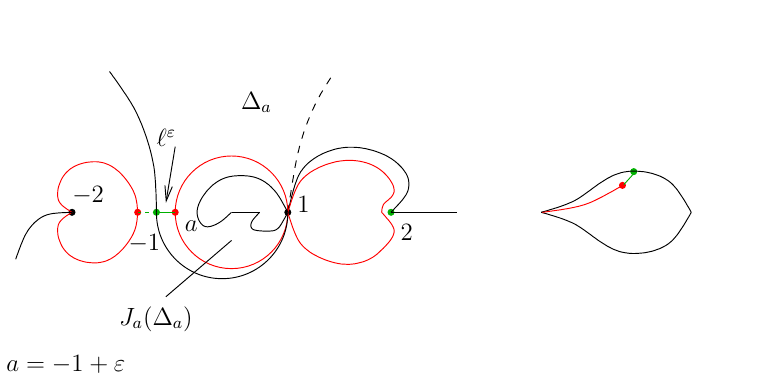}\includegraphics[width=3.5cm]{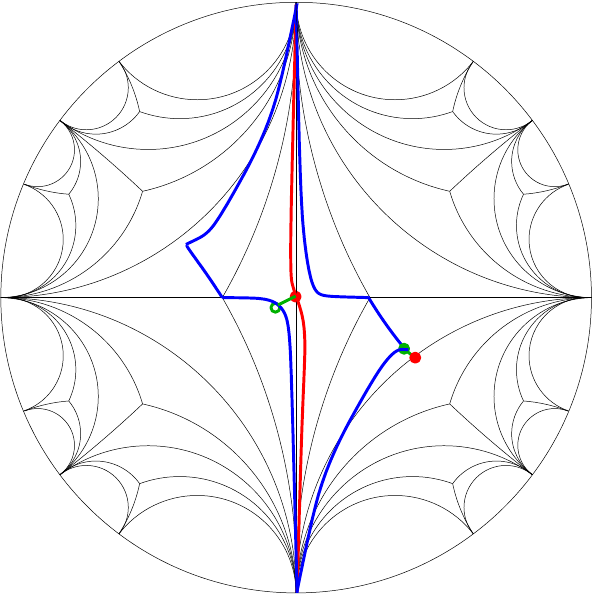}

\caption{Top row: for $a=1+\delta$ (some $0<\delta<1$) we show: the path $\ell:=\ell^\mathcal{U}$ on the $Z$-plane, from $Z=a$ to $Z=-1$; its projection $[\ell]$ to the orbifold $\Omega(\F_a)/\langle \F_a\rangle$; two lifts of $[\ell]$ to $\HH$ - the marked triangles are $\varphi_a(\Delta_a)$ on the left of the vertical axis, and $\varphi_a(J_a(\Delta_a))$ on the right. Bottom row: the same for $a=-1+\varepsilon$ (some $0<\varepsilon<1$) and a path $\ell^{\epsilon}$ from $a$ to $-1$. In each of the diagrams on the $Z$-plane, the three topological circles coloured red are the round circle $\partial\Delta_{J_a}$ and its two images under $Cov_0^Q$.}\label{newfig5}
\end{figure}

For $a\in \D(0,1)$ we shall denote by ($\Delta_Q^\mathcal{U},\Delta_{J_a})$ the Klein combination pair at the end of Step $3$ of the proof of Proposition \ref{D01} (the sets illustrated bottom right in Figure \ref{deform_pic}); we denote by $\varphi_a^\mathcal{U}$ the B\"ottcher map which sends $\Delta^\mathcal{U}_a:=\Delta_Q^\mathcal{U}\cap \Delta_{J_a}$ into $\HH$, and note that $\varphi_a^\mathcal{U}(-1)$ is well-defined.

We claim that when $a=-1+\varepsilon$ with $0<\varepsilon<1$ the image of $\Delta^\mathcal{U}_a$ under $\varphi_a$ is as illustrated in Figure \ref{newfig5}, at least as far as the vertices of the image are concerned (and the true edges are homotopic to those illustrated). To show why this is the case, we will define a map $\varphi_{-1}^{\mathcal{U}}$, determine its image, and argue using joint continuity with nearby $\varphi_a^\mathcal{U}$, $a\in \D(0,1)$. For $a=-1$, we no longer have a Klein combination pair, as the intersection between the complements of $\Delta_Q^\mathcal{U}$ and $\Delta_{J_{-1}}$ has expanded from the single point $1$ to become the lower half of the unit circle. The orbifold $\Omega(\F_{-1})/\langle \F_{-1}\rangle$ is isomorphic to the quotient of $\Delta_Q^\mathcal{U}$ by boundary pairings (induced by $Cov_0^Q$ and $J_{-1}\circ Cov_0^Q$) so it is still conformally a sphere with three marked points. Although the marked point corresponding to $Z=-1$ is a puncture in the boundary-identified $\Delta_Q^\mathcal{U}$, the orbifold universal cover with this marked point regarded as a $\pi$-cone point (and with $2\pi/3$-cone point corresponding to $Z=\infty$ and puncture point corresponding to $Z=1$) is $\HH$, as before, equipped with $PSL(2,\Z)$ as covering group. Thus we can apply the same recipe as we used earlier to define $\varphi_a^\mathcal{U}$, and define a conformal homeomorphism $\varphi_{-1}^\mathcal{U}$ of  $\Delta_Q^\mathcal{U}$ into $\HH$, namely send $-1$ to $i\in\HH$, and for each $Z\in \Delta_Q^\mathcal{U}$ send an arc from $-1$ to $Z$ to the corresponding lifted arc in $\HH$. The map $(a,Z) \to \varphi_a^\mathcal{U}(Z)$ is jointly continuous in $a$ and $Z$ at every $(a,Z)$ where $\varphi_a^\mathcal{U}(Z)$ is defined (since we have used the same recipe for both $\varphi_a^\mathcal{U}(Z)$ and $\varphi_{-1}^\mathcal{U}(Z)$). In particular this is true at $(a,Z)=(-1,-1)$. The image $\varphi_{-1}^\mathcal{U}(\Delta_Q^\mathcal{U})$ is a quadrilateral in $\HH$ with vertices $+\infty,i,(-1+i\sqrt{3})/2$ and $\rho(i)$. We deduce that for $a=-1+\varepsilon$ the configuration of $\varphi_a^\mathcal{U}(\Delta_a^\mathcal{U})$ is as illustrated, and that the image of $\ell^\varepsilon$ contracts to the single point $i\in \HH$ as $\varepsilon$ goes to zero. This reasoning applies to any $a\in \D(0,1)$ as $a$ goes to $-1$, not just for $a=-1+\varepsilon$ with $\varepsilon$ real.

The same reasoning applies when $\Delta_Q^\mathcal{U}$ is replaced by $\Delta_Q^\mathcal{L}$, its complex conjugate in the $Z$-plane. The following result is an immediate consequence of joint continuity of both $\varphi_a^\mathcal{U}(Z)$ and $\varphi_a^\mathcal{L}(Z)$ at $(a,Z)=(-1,-1)$.

\begin{lemma}\label{varphi-1}
For $a\in \D(0,1)$,
$$\lim_{a \to -1}\varphi_a^\mathcal{U}(-1)=i\in\HH \mbox{ and }\lim_{a \to -1}\varphi_a^\mathcal{L}(-1)=i\in\HH. \mbox{\qed}$$
\end{lemma}

It is natural to conjecture that this statement holds more generally for $a$ approaching $-1$ from any direction within $\K$, but we do not yet know enough about the boundary of $K$ in a neighbourhood of $-1$ to prove this.

\begin{prop}\label{Psi-1} 
As $a\in\K$ approaches $-1$ along any path in $\D(0,1)$,
$$\lim_{a \to -1}\Psi^\mathcal{U}(a) = \sigma\rho(i)=(1+i)/2\in \HH\mbox{ and } \lim_{a \to -1}\Psi^\mathcal{L}(a) = \sigma\rho^{-1}(i)=1+i \in \HH.$$
\end{prop}

\begin{proof}

The diagrams in Figure \ref{newfig5} illustrate $\varphi_a(\Delta_a)\subset \HH$ when $a$ is at the start of the path, and when $a$ is nearing $-1$. Recall that the partial B\"ottcher map $\varphi_a$ was first defined as a conformal map from $\Delta_a$ into $\HH$ with the vertices of $\Delta_a$ mapped to the standard vertices of a fundamental domain for $PSL(2,\Z)$ on $\HH$, in particular $\varphi_a(a)$ was defined to be $i\in\HH$. The map $\varphi_a$ was then extended as far as the critical point $c_a=-1$, but only `from one side'. A choice has to be made when defining this extension, namely which side of $\Delta_a$ to extend across near to the fixed point of $\rho$: we have made our choice by following an upper path around $Z=1$, and so $\varphi_{1+\delta}(-1)$ is the point in $\HH$ indicated by the green circle in the top right of the figure. By the time $a$ has reached the position in the lower diagram, close to $-1$, the images $\varphi_a(\Delta_a)$ and $\varphi_a(J_a(\Delta_a))=\sigma(\varphi_a(\Delta_a))$ have moved to (approximately) the positions indicated in the lower figure. In the limit as $a$ goes to $-1$, the point $\varphi_a(-1)$ tends to $i\in \HH$ (by Lemma \ref{varphi-1}). Since the critical value $v_a$ is the point $\F_a(-1)$, it follows at once that $\varphi_a(v_a)$ tends to $\sigma\rho(i)\in \HH$. Following the complex conjugate path $\ell^\mathcal{L}$ from $a=1+\delta$ to $a=-1$ clockwise around the puncture point $1$, we find that $\varphi(v_a)$ tends to $\sigma\rho^{-1}(i)$.
\end{proof}

One consequence is an answer to the question raised by Lemma \ref{cut_line}(iii).

\begin{cor}
For every $a\in \D(0,1)$, $\Psi^\mathcal{L}(a) =  \sigma\rho\sigma (\Psi^\mathcal{U}(a)$).
\end{cor}

\begin{proof}
Proposition  \ref{Psi-1} allows us to extend the definitions of $\Psi^\mathcal{U}$ and $\Psi^\mathcal{L}$ continuously from the open unit disc $\D(0,1)$ to the boundary point $a=-1$, and tells us that $\Psi^\mathcal{L}(-1) =  \sigma\rho\sigma (\Psi^\mathcal{U}(-1)$).
The result follows.
\end{proof}

We next examine the dynamics of $\F_{-1}$. Although $-1\notin \K$, the correspondence $\F_{-1}$ is still {\it discrete} in the sense that there are points $Z\in\widehat{\C}$ which have a neighbourhood $V$ such that no image of $V$ under forward, backward or mixed iteration of $\F_a$ (other than identity on $V$) meets $V$. Indeed we have a dynamical partition into a regular set and a limit set just as for $a\in\K$.

\begin{figure}
\centering
\includegraphics[width=8cm]{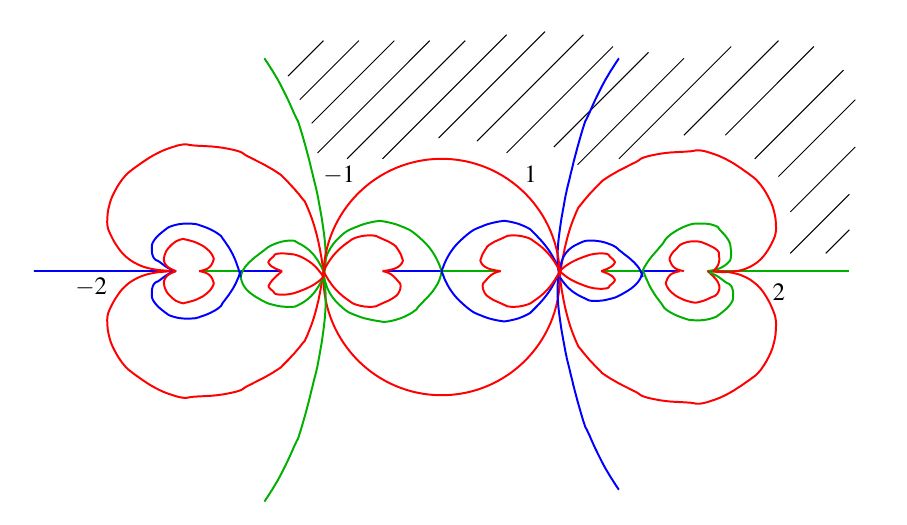}
\caption{Dynamics of $\F_{-1}$: copies of the hatched set $\Delta^\mathcal{U}_{-1}$ tile $\Omega(\F_{-1})$.}\label{-1plot}
\end{figure}

\begin{prop}\label{F_-1}
The Riemann sphere $\widehat{\C}$ is partitioned into a non-empty open set $\Omega(\F_{-1})$ and its (non-empty closed) complement $\Lambda(\F_{-1})$, each invariant under both $\F_{-1}$ and its inverse. There exist (open) fundamental domains $\Delta_{J_{-1}}$ and $\Delta_Q$, for $J_{-1}$ and $Cov^Q$ respectively, such that $\Omega(\F_{-1})$ is tiled by the images of $\Delta_Q\cap\Delta_{J_{-1}}$ under (mixed forward and backward) iteration of $\F_{-1}$.
\end{prop}

\begin{proof}
As before we take $\Delta_Q^\mathcal{U}$ to be the transversal of $Q$ at the end of Step 2 of the proof of Proposition \ref{D01}, and $\Delta_{J_{-1}}$ to be the complement of the unit disc in the $Z$-plane. The images of $\Delta_{-1}^\mathcal{U}:=\Delta_Q^\mathcal{U}\cap\Delta_{J_{-1}}$ (the hatched set in Figure \ref{-1plot}), under all finite words in $Cov_Q$ and $J_{-1}$, tile an open set, and by including appropriate boundaries of tiles we obtain an open completely invariant set $\Omega(\F_{-1})$. 
\end{proof}

Figure \ref{-1plot} illustrates the images of $M=[-\infty,-2]$, $N=[2,\infty]$ and $\partial \Delta_{J_{-1}}= L\cup \bar{L}$ (in the notation of Figure \ref{deform_pic}), under words in $Cov^Q$ and $J_{-1}$. Whether we choose $\Delta_{-1}^\mathcal{U}$ or its complex conjugate $\Delta_{-1}^\mathcal{L}$ as our initial tile, we obtain the same tiling. Indeed $\Delta^\mathcal{L}_{-1}$ is one of the two images of $\Delta^\mathcal{U}_{-1}$ under $Cov_0^Q$. The complement $\Lambda(\F_{-1})$ of the union of the closed tiles (minus their cusp vertices) is a Cantor set contained in the interval $[-2,+2]$ on the real axis, indeed contained in $[-2,-1/2]\cup [+1/2,+2]$; this Cantor set has a single family of gaps, the images of the interval $(-1/2,+1/2)$ (or equivalently the interval from $+2$, through $\infty$, to $-2$) under (mixed) iteration of $Cov_0^Q$ and $J_{-1}$.

\begin{remark}
 Topologically $\F_{-1}$ can be obtained from $\F_a$ $(a=1+\delta)$, by pinching the path $\ell^\mathcal{U}$ on $\Delta_a$ to a point, and each of the images of this path under backwards, forwards, and mixed iteration of $\F_a$ to a point. Note that these images are assembled into `long' paths in $\Omega(\F_a)$ since lifts of $[\ell]$ to $\Omega(\F_a)$ are joined at the critical point $-1$, points $\F_a^{-n}(-1)$ and also at $a$ and points $\F_a^{-n}(a)$. In contrast the lifts of $[\ell]$ to $\HH$ are `short' (they each lie in a finite union of tiles).
 \end{remark}

The dynamical picture for any $a =-1+\delta$, ($\delta$ real $>0$) is broadly similar to that for $\F_{-1}$. Once again the limit set is a Cantor set, now contained in a smaller real interval than $[-2,2]$, and now with further gaps, generated by a gap which opens up around the point $-1$ (the parabolic point splits into an attracting fixed point and a repelling fixed point, both real). Although the tiling generated by the shaded tile in Figure \ref{deform_pic}(iv) is no longer preserved by complex conjugation, we can make it so, by an appropriate adjustment to the boundary of $\Delta_Q$. We note that as $a$ passes through $-1$ in the opposite direction, towards the left (so moving out of $\K$), the parabolic fixed point splits into a complex conjugate pair of neutral fixed points. In Appendix II  (Section \ref{triangles}) we discuss the dynamics of $\F_a$ for $a\in (-\infty,-1)$, and the existence of isolated parameter values in this interval where $\F_a$ is discrete.

\subsection{Based fundamental domains for $PSL(2,\Z)$}\label{isotopy_description}

In this subsection, we shall establish some constraints on the set of points of $\HH$ which could lie in the image of a canonical map $\psi:\K\setminus\m \to \HH$. 

\begin{defi}\label{based_defi}
We say a fundamental domain for $PSL(2,\Z)$ on $\HH$ is {\it based} if it has boundary a Jordan curve which passes through the points $0,i,\infty,(-1+i\sqrt{3})/2$ (the {\it vertices}) and is smooth between vertices. Two fundamental domains are {\it based isotopic} if there is an isotopy between them through based fundamental domains. 

We define the set of {\it accessible points} $\mathcal{W}_{acc}\subset \HH$ to be the points $x+iy\in \HH$ which have $x\le 0$ or else lie in $\Delta\cup\sigma(\Delta)$ for some based fundamental domain $\Delta$ for $PSL(2,\Z)$. 
\end{defi}

For each $a\in\K$ and Klein combination pair $(\Delta_Q,\Delta_{J_a})$, our B\"ottcher map $\varphi_a$ sends $\Delta_a$ conformally to a based fundamental domain for $PSL(2,\Z)$ on $\HH$. It follows from the way we defined a canonical map $\psi$ that points $z\in \HH\setminus\mathcal{W}_{acc}$ cannot be in the image of $\psi$.  

\begin{figure}
\begin{center}

\scalebox{0.3}{\includegraphics{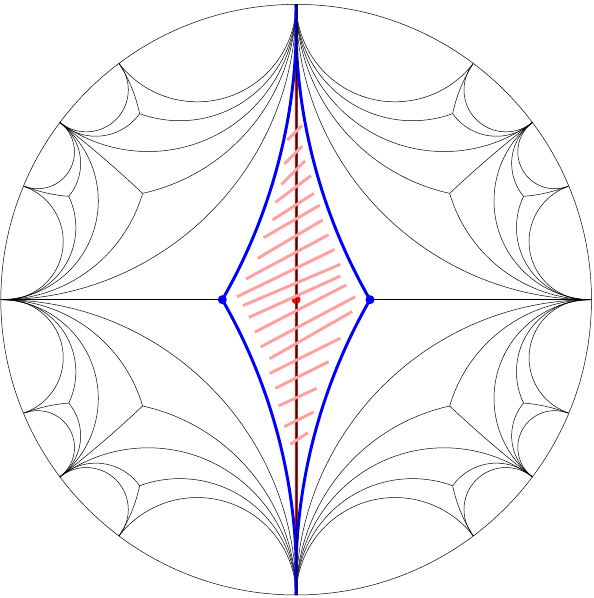}\hskip2.0cm\includegraphics{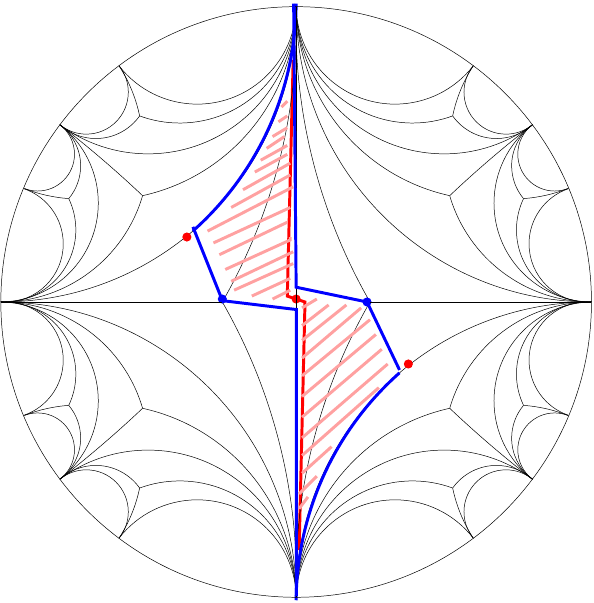}\hskip2.0cm\includegraphics{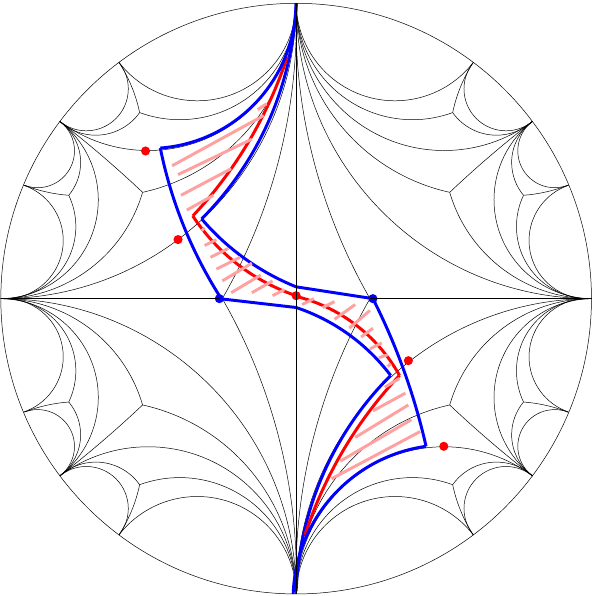}}

\vskip0.3cm

(i)\scalebox{0.25}{\includegraphics{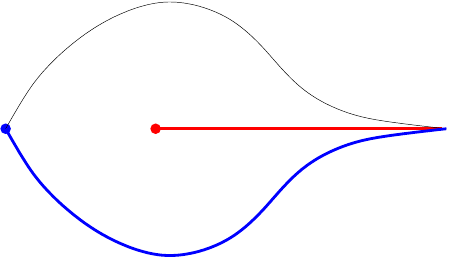}}\hskip1.2cm(ii)\scalebox{0.25}{\includegraphics{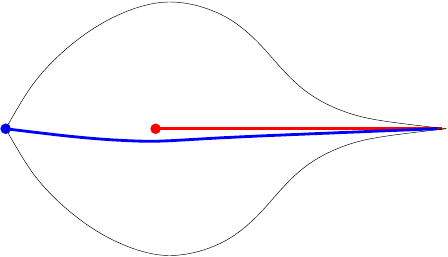}}\hskip1.2cm(iii)\scalebox{0.25}{\includegraphics{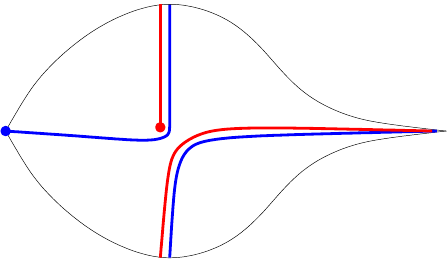}}

\vskip0.6cm

\vskip0.5cm

\scalebox{0.3}{\includegraphics{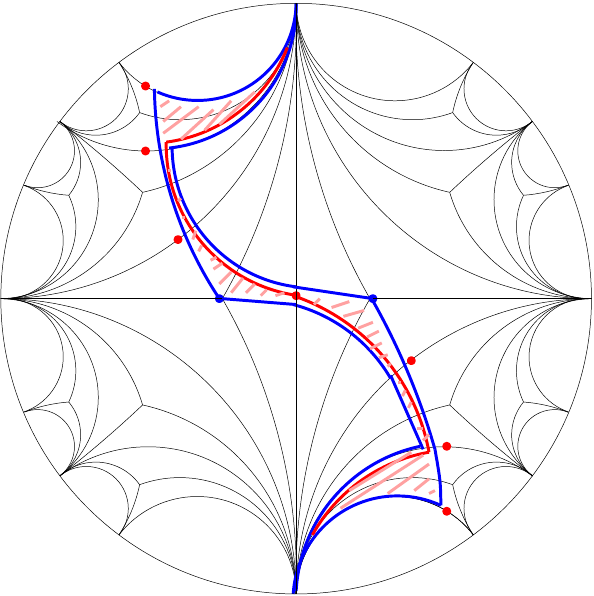}\hskip2.0cm\includegraphics{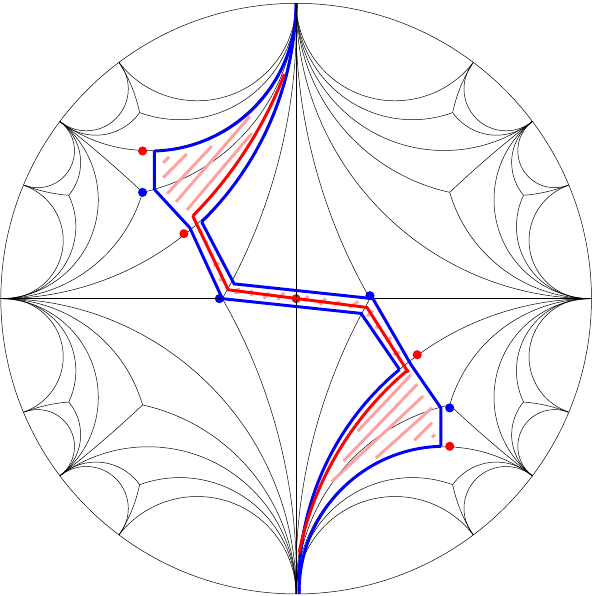}\hskip2.0cm\includegraphics{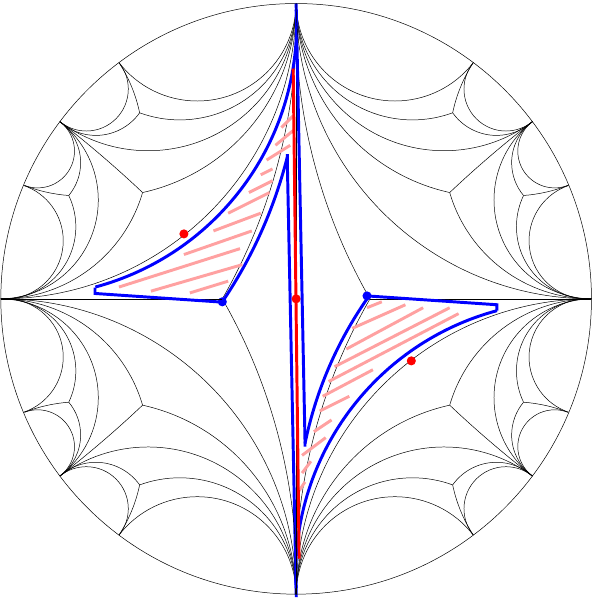}}

\vskip0.3cm

(iv)\scalebox{0.25}{\includegraphics{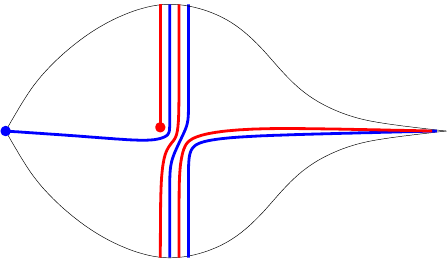}}\hskip1.2cm(v)\scalebox{0.25}{\includegraphics{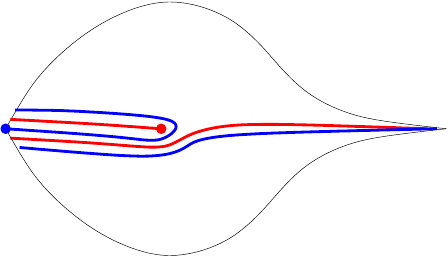}}\hskip1.2cm(vi)\scalebox{0.25}{\includegraphics{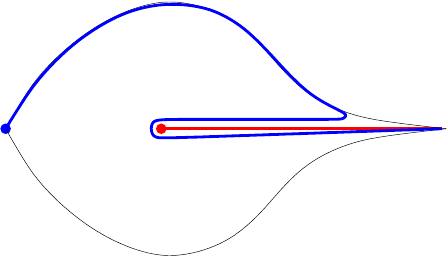}}

\caption{Proposition \ref{accessible}: based fundamental domains for $PSL(2,\Z)$ on $H$. In each example the fundamental domain $\Delta\subset\HH$ and its `twin' $\sigma\Delta$ are indicated in the upper picture by pink hatching; in the lower picture we display the corresponding pair of red and blue arcs on the standard representation of the modular surface $\HH/PSL(2,\Z)$ (a standard tile on $\HH$, with edges identified via $\sigma$ and $\rho$): see Remark \ref{based_remark} for more about these examples.}\label{isot}
\end{center}
\end{figure}

Our first observation is  that each based fundamental domain $\Delta\subset \HH$ corresponds to a choice of a pair of smooth non-intersecting arcs on the modular surface $\HH/PSL(2,\Z)$, between $[0]$ and $[i]$, and between $[0]$ and $[(-1+i\sqrt{3})/2]$ respectively (where $[z]$ denotes the projection of $z\in \HH$ to $\HH/PSL(2,\Z)$). Thus every based fundamental domain is based isotopic to the standard based fundamental domain. Figure \ref{isot} displays some examples of these pairs of arcs on the standard modular surface, and the corresponding based fundamental domains when the arcs on the modular surface are lifted to $\HH$.

In Section \ref{structure_K} we shall be interested in which torsion points in $\HH$ (points on the $PSL(2,\Z)$-orbit of $i$ or of $(-1+i\sqrt{3})/2$) are on the boundary of $\mathcal{W}_{acc}$. Let $\HH_+:=\{x+iy\in\HH:x>0\}$.

\begin{prop}\label{accessible}
The torsion points in $\overline{\mathcal{W}_{acc}}\cap\HH_+$ are $(\sigma\rho)^{n}(i)$, $(\sigma\rho^{-1})^{n}(i)$ ($n>0)$; and $(1+i\sqrt{3})/2$, $(\sigma\rho)^{n}((1+i\sqrt{3})/2)$,  $(\sigma\rho^{-1})^{n}((1+i\sqrt{3})/2)$ ($n>0$).
\end{prop}

\begin{proof}

The modular surface $\HH/PSL(2,\Z)$ has the complex structure of a sphere with three marked points.  We may place the puncture point at the north pole of this sphere, the $2\pi/3$-cone point at the south pole and the $\pi$-cone point somewhere on the equator. The images under the projection $\HH \to \HH/PSL(2,\Z)$ of the boundary of our standard fundamental domain become our two arcs on $\widehat{\C}$, emanating from the puncture point and ending at the respective marked points. This configuration of arcs is shown in the lower part  of Figure \ref{isot}(i): here the punctured sphere (the modular surface) is on its side, with the north pole (the puncture) on the right, the south pole (the $2\pi/3$-cone point) on the left, and the $\pi$-cone point in the centre. As is usual in our illustrations of the modular surface, the top and bottom edges are identified. The lower picture in Figure \ref{isot}(i) shows the pair of arcs on the modular surface  which lift to our standard tile boundaries for $PSL(2,\Z)$ on $\HH$ shown in the upper picture.

We first show that for each of the torsion points listed in the statement of the Proposition there exists an isotopy of the two arcs on the modular surface, which keeps the marked end points fixed, but brings the boundary of the corresponding fundamental domain arbitrarily close to the chosen torsion point. Figure \ref{isot}(ii) is the case of the torsion point $\sigma\rho(i)$:  we can reach this configuration from Figure \ref{isot}(i) by moving the blue arc upwards, to bring it close to the $\pi$-cone point, without moving the red arc. The new arcs lift to $\HH$ as shown in the upper part of the figure and we are done. Equally we can bring the fundamental domain boundary arbitrarily close to $\sigma\rho^{-1}(i)$ by moving the blue arc in Figure \ref{isot}(i) downwards. Our next move is to push both red and blue arcs around a line of latitude encircling the puncture point, so that they make a complete turn around the puncture (Figure \ref{isot}(iii)). In the lift to $\HH$ this brings a point on the based fundamental domain boundary coloured blue, close to $(\sigma\rho)^2(i)$.  Another push around the same circle, Figure \ref{isot}(iv), takes the boundary close to $(\sigma\rho)^3(i)$, so on. Pushing the curves the opposite way around the same circle deals with $(\sigma\rho^{-1})^n(i)$ in the obvious way. Finally we can make the boundary of the fundamental domain approach arbitrarily close to the torsion points $(\sigma\rho)^n((-1+i\sqrt{3})/2)$, and $(\sigma\rho^{-1})^{n}((-1+i\sqrt{3})/2)$ for integers $n>0$: in the case $n=1$ we apply an isotopy of the arcs in Figure \ref{isot}(iii) to the position in Figure \ref{isot}(v), and for the cases $n>1$ there are obvious generalizations.

It remains to show that these are the only torsion points in $\HH_+$ that can be approached arbitrarily closely by based fundamental domain boundaries. It will suffice to show that they are the only torsion points $z\in \HH_+$ with the property that for all $\varepsilon>0$ there exists a path $\ell$ in $\HH$ from $i$ to a point $\zeta$ in an $\varepsilon$-neighbourhood of $z$ such that $\ell$ projects injectively to the modular surface $\HH/PSL(2,\Z)$ (since every based fundamental domain $\Delta$ contains a path from the vertex $i$ to any given point $\zeta\in \Delta$). First consider $2$-torsion points. After a small perturbation of $\ell$ if necessary, we can assume $[\ell]$ does not pass through the $2\pi/3$-cone point, and thus it is a path on a cylinder (the twice-punctured sphere) which ends within distance $\varepsilon$ of its starting point [i]. Such a curve necessarily spirals $n$ times around the cylinder. Hence $z=(\sigma\rho)^{n}(i)$ or $(\sigma\rho^{-1})^{n}(i)$ (the other possibilities $(\rho\sigma)^n(i)$ and $(\rho^{-1}\sigma)^n(i)$ are in $\HH_-$). For $z$ a $3$-torsion point the argument is similar: now either the path $\ell$ is from $i$ to (within distance $\varepsilon$ of) $(1+i\sqrt{3})/2$ or its projection $[\ell]$ spirals around an arc between the images of these two points on $\HH/PSL(2,\Z)$.

\end{proof}

\begin{remark}\label{based_remark}
1. Let  $\mathcal{E}$ denote the closure of the set of all points which lie in based fundamental domains for $PSL(2,\Z)$ (Definition \ref{based_defi}). The proof of Proposition \ref{accessible} tells us $\mathcal{E}$ contains the `horodisc'  bounded by the images under $(\sigma\rho)^n$ (for all $n\in \Z$) of the geodesic segment from $(-1+i\sqrt{3})/2$ through $i$  to  $(+1+i\sqrt{3})/2$, and also the corresponding horodisc bounded by images $(\sigma\rho^{-1})^n$ of the same geodesic segment. But notice that the based fundamental domain in Figure \ref{isot}(vi) contains  points outside these two horodiscs. Determining $\mathcal{E}$ is an intriguing problem: we are not aware of any relevant work in the literature.

2. While Proposition \ref{accessible} tells us which torsion points cannot lie in $\psi(\K)$ or its boundary, it does not tell us which do, since we have no guarantee that any given based fundamental domain for $PSL(2,\Z)$ arises as $\varphi_a(\Delta_a)$ for some $\F_a$ with $a\in\K$, or, if it does, that $v_a$ can be any point in $\Delta_a$.  We believe (see Conjecture \ref{conj_an} in Section \ref{partialK}) that all $(\sigma\rho)^n(i)$ and $(\sigma\rho)^{-n}(i)$ are in this boundary of $\psi(\K)$ and that they correspond to discrete
correspondences $\F_a$ with $a\in\partial\K$. We also conjecture that the points $(\sigma\rho)^n((-1+i\sqrt{3})/2)$, and $(\sigma\rho^{-1})^n((-1+i\sqrt{3})/2)$ cannot correspond to discrete $\F_a$, but it remains an open question whether these torsion points lie on the boundary of the image of $\K$ (Section \ref{partialK}).
\end{remark}

Here we have considered the fundamental domains for the action of the group $PSL(2,\Z)$ on $\HH$ which are candidates to be $\varphi_a(\Delta_a)\subset\HH$. In the next section we will be concerned with the position of the point $\varphi_a(v_a)$ within such a fundamental domain, in particular we prove that $\varphi_a(v_a)\in\HH$ is analytic in $a$.


\section{Analyticity of canonical maps}\label{anal}

Contrary to the usual B\"{o}ttcher coordinates for rational maps, we do not know how the B\"{o}ttcher maps $\varphi_a$ for the family $\F_a$ behave with respect to the parameter, especially since we don't even necessarily have uniqueness of such maps. Unable to show that the maps $\varphi_a$ can be taken to depend analytically on $a$, we will essentially show that any mapping $a \mapsto \varphi_a(v_a)$ is analytic along $a$. In order to achieve that, we first need a holomorphic motion of fundamental pairs.

\begin{lemma}\label{lem.holomorphic_motion}

    Let $a_0 \in \K$ and $(\Delta_Q, \Delta_{J_{a_0}})$ be a Klein combination pair for $a_0$ satisfying the transversality condition. Then there exists a neighborhood $D$ of $a_0$ and a holomorphic motion
    \[ \xi_a: \Delta_{J_{a_0}} \to \widehat{\C}, \ a \in D \]
    of $\Delta_{J_{a_0}}$, along $D$ and centered at $a_0$, such that, for all $a \in D$, $\Delta_{J_a} := \xi_a(\Delta_{J_{a_0}})$ is a fundamental domain for $J_a$ and $(\Delta_Q, \Delta_{J_a})$ is a Klein combination pair satisfying the transversality condition.
    
\end{lemma}

\begin{proof}

    Since the correspondence $\F_{a_0}$ has $Z = 1$ as a parabolic fixed point for a well defined local branch, we may find pre-Fatou coordinates $\zeta$ on a neighborhood $U$ of $1$ for which the expression of $\F_{a_0}$ is a map $f_{a_0}(\zeta) = \zeta + 1 + \mathcal{O}(1/z)$. Notice that both $J_{a_0}$ and $Cov_0^Q$ act as involutions near $Z = 1$, and therefore their expressions are maps
    \[ j_{a_0}: \zeta \mapsto -\zeta + b_{a_0} + \mathcal{O}\left( \frac{1}{z} \right) \text{ and } C: \zeta \mapsto -\zeta + B + \mathcal{O}\left( \frac{1}{z} \right), \]
    respectively, for some constants $b_{a_0}, B \in \widehat{\C}$. Since $\F_{a_0} = J_{a_0}\circ Cov_0^Q$, we see that $b_{a_0} = B + 1$. The transversality condition then tells us that $U\cap \partial\Delta_Q$ and $U\cap \partial\Delta_{J_{a_0}}$ are mapped to pairs of curves, one above the real line and the other bellow, that go to infinity in directions transversal to the real axis.

    Let us denote the two components of $U\cap \partial\Delta_Q$ in the $\zeta$ coordinate as $\gamma_Q^+$ and $\gamma_Q^-$, and the two components of $U\cap \partial\Delta_{J_{a_0}}$ in the $\zeta$ coordinate as $\gamma_{a_0}^+$ and $\gamma_{a_0}^-$, with $\gamma_Q^+$ and $\gamma_{a_0}^+$ being above the real line, and $\gamma_Q^-$ and $\gamma_{a_0}^-$ being below. By restricting $U$, we can also assume that $\gamma_Q^- = C(\gamma_Q^+)$ and $\gamma_{a_0}^- = j_{a_0}(\gamma_{a_0}^+)$. Let us also assume that there exists some $\delta$ such that $B_\delta(\gamma_Q^+)\cap B_\delta(\gamma_{a_0}^+) = B_\delta(\gamma_Q^-)\cap B_\delta(\gamma_{a_0}^-) = \emptyset$ in $\C$. Then, for a small enough neighborhood $D$ of $a_0$, denoting by $j_a$ the expression of $J_a$ in the $\zeta$ coordinate, we have that $j_a(\gamma_{a_0}^+) \subset B_\delta(\gamma_{a_0}^-)$ and, in particular, does not intersect $\gamma_Q^-$. Then one can define a holomorphic motion $\xi_a$ of $\partial\Delta_{J_{a_0}}$ along $a \in D$ by setting $\xi_a$ as identity on the part corresponding to $\gamma_{a_0}^+$, an arbitrary motion (via Slodkowski \cite{Sl} for instance) of the part connecting $\gamma_{a_0}^+$ to the point $Z = a$, and the image $J_a$ of the two previous parts. By all we have considered, the image $\xi_a(\partial\Delta_{J_{a_0}})$ will be a $J_a$-invariant Jordan curve, tangentially intersecting $\Delta_Q$ at $Z = 1$, such that one of the two connected components of $\widehat{\C}\setminus \xi_a(\partial\Delta_{J_{a_0}})$ covers $\widehat{\C}\setminus \Delta_Q$. We can extend $\xi_a$ via Slodkowski to the whole of $\Delta_{J_{a_0}}$ and the image $\xi_a(\Delta_{J_{a_0}}) = \Delta_{J_a}$ has to be a fundamental domain for $J_a$ satisfying that $(\Delta_Q, \Delta_{J_a})$ is a Klein combination pair.
    
    Now, let us deal with the case where we do not have $B_\delta(\gamma_Q^+)\cap B_\delta(\gamma_{a_0}^+) = B_\delta(\gamma_Q^-)\cap B_\delta(\gamma_{a_0}^-) = \emptyset$ for any $\delta > 0$. Without loss of generality, we may assume $B_\delta(\gamma_Q^+)\cap B_\delta(\gamma_{a_0}^+) \neq \emptyset$ for all $\delta > 0$, meaning the curves $\gamma_Q^+$ and $\gamma_{a_0}^+$ get arbitrarily close to each other near $\infty$. We can then, maybe after reducing the neighborhood $U$, find $\delta' > 0$ such that $\gamma_{a_0}^+ \subset B_{\delta'}(\gamma_Q^+)$. For $a \in D$, we see that $j_a(\zeta) = -\zeta + b_a + \mathcal{O}(1/z)$, with $b_a$ varying analytically with $a$. In particular, $b_a = b_{a_0} + \varepsilon_a = B + 1 + \varepsilon_a$, with $\varepsilon_a$ close to $0$ if $D$ is made small. Thus $j_a(\zeta) - C(\zeta) = 1 + \varepsilon_a + \mathcal{O}(1/z)$, and we find that, if $D$ is small enough, $j_a(B_{\delta'}(\gamma_Q^+))\cap C(B_{\delta'}(\gamma_Q^+)) = \emptyset$, meaning $j_a(\gamma_{a_0}^-)$ does not intersect $\gamma_Q^-$ for any $a \in D$, and we can still carry out the previous argument.
    
\end{proof}

\begin{cor}

    The Klein combination locus $\K$ is open.
    
\end{cor}

The following Lemma establishes that we may always assume the boundaries of the fundamental domains $\Delta_Q$ and $\Delta_{J_a}$ to be smooth. This will be necessary for some arguments in the proof of analyticity.

\begin{lemma}\label{lem.smooth_boundaries}

    If $a \in \K$, we may find a Klein combination pair $(\Delta_Q, \Delta_{J_a})$ such that $\partial\Delta_Q$ and $\partial\Delta_{J_a}$ are piecewise $\mathcal{C}^1$ curves.
    
\end{lemma}

\begin{proof}

    Let $(\Delta_Q', \Delta_{J_a}')$ be a Klein combination pair for $a$. It is clear that we only need to construct $\Delta_Q$ and $\Delta_{J_a}$ near $Z = 1$ in order to guarantee the tangency, since the boundaries can be made piecewise smooth outside of a neighborhood of $Z = 1$ without incurring any other intersections. Again, let us use the coordinates $\zeta$ from Lemma \ref{lem.holomorphic_motion}, denote $C$ and $j_a$ the actions of $Cov_0^Q$ and $J_a$ in the $\zeta$ coordinates, and let $\gamma_Q^\pm$ and $\gamma_a^\pm$ be the pieces of the boundaries of $\Delta_Q'$ and $\Delta_{J_a}'$ in the $\zeta$ coordinates. As long as we can make either the curves above or the curves below the real line be $\mathcal{C}^1$, the result is proven by the fact that $\gamma_Q^- = C(\gamma_Q^+)$ and $\gamma_a^- = j_a(\gamma_a^+)$. If we have the property that $B_\delta(\gamma_Q^+)\cap B_\delta(\gamma_{a_0}^+) = B_\delta(\gamma_Q^-)\cap B_\delta(\gamma_{a_0}^-) = \emptyset$ for some $\delta > 0$, this is clearly possible. Without loss of generality, let us then assume that $B_\delta(\gamma_Q^+)\cap B_\delta(\gamma_{a_0}^+) \neq \emptyset$ for all $\delta > 0$. Then again we may find $\delta' > 0$ such that $\gamma_{a_0}^+ \subset B_{\delta'}(\gamma_Q^+)$. Since $\partial\Delta_Q'$ and $\partial\Delta_{J_a}'$ do not intersect other than at the point $Z = 1$, one of the curves $\gamma_Q^+$ and $\gamma_a^+$ is to the left of the other. Take any pair of smooth curves in $B_{\delta'}(\gamma_Q^+)$ that do not intersect, i.e. one is to the left of the other, landing at infinity. We can then make the boundaries of the new sets $\Delta_Q$ and $\Delta_{J_a}$ near $Z = 1$ be given by these curves and their images under $C$ and $j_a$ in the $\zeta$ coordinates. Again, we recall from the proof of Lemma \ref{lem.holomorphic_motion} that the images under $C$ and $j_a$ will not intersect.
    
\end{proof}

Before proving analyticity of $\psi$, we must isolate some special parameters: the \textit{critical orbit relation set} is defined as
\[ CR := \left\{ a \in \K\setminus \m \ | \ \pi_a(c_a) \text{ is a cone point of } \faktor{\Omega_a}{\left< \F_a \right>} \right\}. \]
Notice that $CR$ is discrete in $\K$. Thus, if we prove analyticity of $\psi$ outside of $CR$, boundedness of $\psi$ implies, via the Riemann Removable Singularity Theorem, that $\psi$ is analytic at the points of $CR$ as well --- we may look at $\psi$ as a function to $\D$ instead of $\HH$ for boundedness to be a more obvious property.

\begin{teor}\label{psi_analytic}

    Let $\psi$ be any canonical map defined in some open set $\mathcal{U} \subset \K\setminus \m$. Then $\psi$ is analytic on $\mathcal{U}$.
    
\end{teor}

\begin{proof}

    Let us fix a parameter $a_0 \in \mathcal{U}\setminus CR$ and a Klein combination pair $(\Delta_Q, \Delta_{J_{a_0}})$, that we assume to have piecewise $\mathcal{C}^1$ boundaries by Lemma \ref{lem.smooth_boundaries}. We can assume that $v_{a_0}$ has one representative of its grand orbit under $\F_{a_0}$ in the interior of $\Delta_{a_0} = \Delta_Q\cap \Delta_{J_{a_0}}$. Indeed, if the representatives were in the boundary of this fundamental domain, then one could make a small perturbation of the boundary of $\Delta_Q$ or $\Delta_{J_{a_0}}$ to make one of them get inside $\Delta_{a_0}$ and the other one outside $\overline{\Delta_{a_0}}$. Notice that this is only possible because $a_0 \notin CR$ --- otherwise its unique representative in $\overline{\Delta_{a_0}}$, being a cone point, would always have to lie in the boundary of any fundamental domain.

    Finally, let $\varphi_{a_0}$ be a partial B\"ottcher map such that $\psi(a_0) = \varphi_{a_0}(v_{a_0})$, and $\hat{\Delta}_{a_0}$ be the fundamental domain derived from $\Delta_{a_0}$ containing $v_{a_0}$. For a small enough neighborhood $D$ of $a_0$ in $\mathcal{U}\setminus CR$, the holomorphic motion $\xi_a$ of Lemma \ref{lem.holomorphic_motion} induces a holomorphic motion
    \[ \eta_a: \hat{\Delta}_{a_0} \to \widehat{\C}, \ a \in D \]
    such that $\hat{\Delta}_a := \eta_a(\hat{\Delta}_{a_0})$ is a fundamental domain for the action of $\F_a$ on $\Omega_a$, containing $v_a$. Since the point $v_a$ moves analytically with the parameter $a$, we may adapt this holomorphic motion in order to guarantee that $\eta_a(v_{a_0}) = v_a$ for all $a \in D$.

    Define now the family of Beltrami coefficients $\mu_a$ on $\hat{\Delta}_{a_0}$ by setting $\mu_a := \eta_a^\ast \mu_0$ for each $a \in D$, where $\mu_0$ is the trivial Beltrami coefficient. Notice that $\{\mu_a\}_{a\in D}$ is an analytic family of Beltrami coefficients, since $\{\eta_a\}_{a\in D}$ is a holomorphic motion. Let $\Pi: \HH \to \Sigma := \HH/PSL(2,\Z)$ be the projection map. Then the following diagram commutes:
    \[ \begin{tikzcd}
    \hat{\Delta}_{a_0} \arrow[d, "\varphi_{a_0}"'] \arrow[r, "\pi_{a_0}"] & \faktor{\Omega_{a_0}}{\left< \F_{a_0} \right>} \arrow[d, "\Theta_{a_0}"] \\
    \HH \arrow[r, "\Pi"']                           & {\Sigma}               
    \end{tikzcd} \]
    We can thus induce an analytic family of $PSL(2,\Z)$-invariant Beltrami forms on $\HH$ by setting
    \[ \nu_a := \Pi^\ast(\Theta_{a_0}^{-1})^\ast(\pi_{a_0}^{-1})^\ast\mu_a. \]
    The map $\pi_{a_0}^{-1}$ is not well defined on $\pi_{a_0}(\partial \hat{\Delta}_{a_0})$, but that is a 0-measure set, so the form $(\pi_{a_0}^{-1})^\ast\mu_a$ is well defined. Notice that $\nu_a$ can also be constructed by taking $\varphi_{a_0}^\ast\mu_a$ on $\varphi_{a_0}(\hat{\Delta}_{a_0}) =: T$ and spreading by the dynamics of $PSL(2,\Z)$, but the above construction makes evident that $\{\nu_a\}_{a\in D}$ is an analytic family of Beltrami coefficients on $\HH$.

    Let then $\psi_a: \HH \to \HH$ be the family of integrating maps of $\nu_a$ that fix the two points of $\partial T$ corresponding to the cone points of $\Sigma$ and one of the points in $\partial T\cap \mathbb{S}^1$, corresponding to the puncture of $\Sigma$. Since $\nu_a$ is $PSL(2,\Z)$-invariant, $\psi_a$ conjugates $PSL(2,\Z)$ with the action of another group. Since $PSL(2,\Z)$ is quasi-conformally rigid, we get that
    \[ \psi_a\circ PSL(2,\Z)\circ \psi_a^{-1} = PSL(2,\Z). \]
    In particular, the other point in $\partial T\cap \mathbb{S}^1$ is also fixed by $\psi_a$. Consider then the composition $\hat{\varphi}_a := \psi_a\circ \varphi_{a_0}\circ \eta_a^{-1}: \hat{\Delta}_a \to \psi_a(T)$:
    \[ \begin{tikzcd}
    \hat{\Delta}_{a_0} \arrow[r, "\eta_a"] \arrow[d, "\varphi_{a_0}"'] & \hat{\Delta}_{a_0} \arrow[d, "\hat{\varphi}_a"] \\
    T\subset \HH \arrow[r, "\psi_a"]                                                      & \psi_a(T) \subset \HH      
    \end{tikzcd} \]
    We shall prove that $\hat{\varphi}_a(v_a) = \psi(a)$. This implies analyticity of $\psi$ on $D$ since then
    \[ \psi(a) = \hat{\varphi}_a(v_a) = \psi_a(\varphi_{a_0}(\eta_a^{-1}(v_a))) = \psi_a(\varphi_{a_0}(v_{a_0})) = \psi_a(\psi(a_0)), \]
    and $\{ \psi_a \}_{a\in D}$ is an analytic family of quasiconformal maps. Since $a_0 \in \mathcal{U}\setminus CR$ was arbitrary, we get that $\psi$ is analytic on $\mathcal{U}\setminus CR$. Since $CR$ is discrete, Riemann's Removable Singularity Theorem then concludes that $\psi$ is analytic on all of $\mathcal{U}$.
    
    To show that $\hat{\varphi}_a(v_a) = \psi(a)$, we need to show that $[\hat{\varphi}_a(v_a)] = \Theta_a([v_a])$. This will mean that the map $a\mapsto \hat{\varphi}_a(v_a)$ is canonical, and so $\hat{\varphi}_a(v_a)$ differs from $\psi(a)$ by an element of $PSL(2, \Z)$. Since the modular group is discrete, these functions are continuous, and they coincide for $a = a_0$, we conclude that this element is identity, giving the desired equality. Let us then begin by noticing that $\eta_a$ descends to a quasiconformal map $H_a: \Omega_{a_0}/\left< \F_{a_0} \right> \to \Omega_a/\left< \F_a \right>$. Indeed, $\eta_a$ is well defined on $\overline{\hat{\Delta}_{a_0}}$, conjugating the dynamics of $\F_{a_0}$ and $\F_a$ on the boundary, and so it descends to a map $H_a$. Furthermore, $H_a$ is quasiconformal on $\pi_{a_0}(\hat{\Delta}_{a_0})$, and $\pi_{a_0}(\partial \hat{\Delta}_{a_0})$ is comprised of piecewise $\mathcal{C}^1$ curves, meaning it is a removable set and $H_a$ is quasiconformal on the whole quotient. We may thus define $\hat{\mu}_a := H_a^\ast\mu_0 = (\pi_{a_0}^{-1})^\ast\mu_a$. Similarly, the map $\psi_a$ descends to a quasiconformal map $P_a: \Sigma \to \Sigma$ since it is $PSL(2,\Z)$-invariant, and, if we set $\hat{\nu}_a := \Pi_\ast\nu_a = (\Theta_{a_0}^{-1})^\ast\hat{\mu}_a$, then $P_a^\ast\mu_0 = \hat{\nu}_a$. We thus find that
    \[ (P_a\circ \Theta_{a_0}\circ H_a^{-1})^\ast\mu_0 = (H_a^{-1})^\ast\Theta_{a_0}^\ast P_a^\ast\mu_0 = (H_a^{-1})^\ast\Theta_{a_0}^\ast \hat{\nu}_a = (H_a^{-1})^\ast \hat{\mu}_a = \mu_0, \]
    meaning $P_a\circ \Theta_{a_0}\circ H_a^{-1} = \Theta_a$, since $\Theta_a$ is the only conformal map of $\Omega_a/\left< \F_a \right>$ to $\Sigma$. These facts combined give us the following commutative cube:

    \[ \begin{tikzcd}
    \left( \faktor{\Omega_{a_0}}{\left< \F_{a_0} \right>}, \hat{\mu}_a \right) \arrow[ddd, "\Theta_{a_0}"] \arrow[rrr, "H_a"'] &                                                                                &                                                                                     & \left( \faktor{\Omega_a}{\left< \F_a \right>}, \mu_0 \right) \arrow[ddd, "\Theta_a"'] \\
                                                                   & (\hat{\Delta}_{a_0}, \mu_a) \arrow[lu, "\pi_{a_0}", hook] \arrow[r, "\eta_a"] \arrow[d, "\varphi_{a_0}"'] & (\hat{\Delta}_a,\mu_0) \arrow[ru, "\pi_a"', hook] \arrow[d, "\hat{\varphi}_a"] &                                                 \\
                                                                   & (\HH,\nu_a) \arrow[r, "\psi_a"'] \arrow[ld, "\Pi"', two heads]         & (\HH,\mu_0) \arrow[rd, "\Pi", two heads]                                             &                                                 \\
    (\Sigma,\hat{\nu}_a) \arrow[rrr, "P_a"]                                          &                                                                                &                                                                                     & (\Sigma,\mu_0)                                         
    \end{tikzcd} \]

    In particular, we have
    \[ \Theta_a([v_a]) = P_a(\Theta_{a_0}(H_a^{-1}([v_a]))) = P_a(\Theta_{a_0}([v_0])) = P_a([\varphi_{a_0}(v_0)]) = [\hat{\varphi}_a(v_a)], \]
    concluding the proof.

\end{proof}

The heuristic behind Theorem \ref{psi_analytic} is basically that the maps $\Theta_a$ depend analytically on $a$ --- or equivalently that the surfaces $\Omega_a/\left< \F_a \right>$ depend analytically on $a$. The definition of canonical map demands that $[\psi(a)] = \Theta_a[v_a]$, and thus, by taking an appropriate local branch of the inverse of $\Pi$, we see that $\psi(a) = \Pi^{-1}(\Theta_a[v_a])$.

Lemma \ref{lem.holomorphic_motion} can also be used to better understand the quasi-conformal conjugacy classes of elements in $\K$. In order to do that, we first state a Lemma that is a direct consequence of the results in \cite{BL1}.

\begin{lemma}

    Let $a \in \K\setminus\m$. Then $\Lambda_a$ has empty interior.
    
\end{lemma}

\begin{proof}

    According to Theorem B in \cite{BL1}, the action of $\F_a$ on $\Lambda_{a,-}$ is hybrid equivalent to the action of some map in the $Per_1(1)$ family on its filled Julia set. We have $a \in \m$ if and only if $\Lambda_{a,-}$ is connected. Therefore, $a \in \K\setminus\m$ implies $\Lambda_{a,-}$ is disconnected. Now, for maps in the $Per_1(1)$ family, the filled Julia set being disconnected implies it is a Cantor set, and in particular has empty interior. Thus $\Lambda_a = \Lambda_{a,-}\cup J_a(\Lambda_{a,-})$ also has empty interior. 
    
\end{proof}

\begin{teor}

    Let $a_0 \in \K\setminus (\m\cup CR)$. Then $\F_a$ is quasi-conformally conjugated to $\F_{a_0}$ whenever $a$ is in the same connected component of $\K\setminus (\m\cup CR)$ as $a_0$.
    
\end{teor}

\begin{proof}

    Notice that it is enough to show the result for every $a$ in a neighborhood of $a_0$, since quasi-conformal conjugacy is an equivalence relation. Let $(\Delta_Q, \Delta_{J_{a_0}})$ be a Klein combination pair for $a_0$. Lemma \ref{lem.holomorphic_motion} guarantees a neighborhood $D$ of $a_0$ and a holomorphic motion $\xi_a: \Delta_{J_{a_0}} \to \widehat{\C}$, $a \in D$, such that $\Delta_{J_a} := \xi_a(\Delta_{J_{a_0}})$ is a fundamental domain for $J_a$ and $(\Delta_Q, \Delta_{J_a})$ is a Klein combination pair for $a$. In particular, $\xi_a$ gives a holomorphic motion of $\Delta_{a_0} = \Delta_Q\cap \Delta_{J_{a_0}}$, and it also conjugates the actions of $\F_{a_0}$ on the boundary $\partial\Delta_{a_0}$ and of $\F_a$ on the image $\xi_a(\partial\Delta_{a_0}) = \partial\Delta_a$. That means we can dynamically extend $\xi_a$ to a holomorphic motion of $\Delta_{a_0}\cup J_{a_0}\Delta_{a_0}$, still conjugating $\F_{a_0}$ and $\F_a$.
    
    Recall from Lemma \ref{fund_lemma} that
    \[ v_{a_0} \in \bigcup_{n\geq 0}\F_{a_0}^{-n}(\Delta_{a_0}'\cup J_{a_0}\Delta_{a_0}') \ \text{ and } \ c_{a_0} \in \bigcup_{n\geq 1}\F_{a_0}^{-n}(\Delta_{a_0}'\cup J_{a_0}\Delta_{a_0}'), \]
    where $\Delta_{a_0}' = \overline{\Delta_{a_0}}\setminus \{P\}$, $P$ the parabolic fixed point. In particular, there exists some $k = k(a_0) \geq 0$ such that
    \[ v_{a_0} \in \F_{a_0}^{-k}(\Delta_{a_0}'\cup J_{a_0}\Delta_{a_0}') \ \text{ and } \ c_{a_0} \in \F_{a_0}^{-(k+1)}(\Delta_{a_0}'\cup J_{a_0}\Delta_{a_0}'). \]
    By perturbing the fundamental domain $\Delta_{J_{a_0}}$ slightly, we can assume that
    \[ v_{a_0} \in \F_{a_0}^{-k}(\Delta_{a_0}\cup J_{a_0}\Delta_{a_0}) \ \text{ and } \ c_{a_0} \in \F_{a_0}^{-(k+1)}(\Delta_{a_0}\cup J_{a_0}\Delta_{a_0}), \]
    that is $v_{a_0}$ and $c_{a_0}$ fall in the interior of some tile, not in the boundary. By making $D$ smaller if necessary, we get
    \[ v_a \in \F_a^{-k}(\Delta_a\cup J_a\Delta_a) \ \text{ and } \ c_a \in \F_a^{-(k+1)}(\Delta_a\cup J_a\Delta_a) \]
    for all $a \in D$. Thus, since there are no critical points in the first $k$ pre-images of $\Delta_a\cup J_a\Delta_a$, we can extend $\xi_a$ dynamically to a holomorphic motion
    \[ \xi_a: \bigcup_{0\leq n\leq k}\F_{a_0}^{-n}(\Delta_{a_0}'\cup J_{a_0}\Delta_{a_0}') \to \widehat{\C} \]
    with $\xi_a(\bigcup_{0\leq n\leq k}\F_{a_0}^{-n}(\Delta_{a_0}'\cup J_{a_0}\Delta_{a_0}')) = \bigcup_{0\leq n\leq k}\F_a^{-n}(\Delta_a'\cup J_a\Delta_a')$, still conjugating the actions of $\F_{a_0}$ and $\F_a$ in its domain of definition. By adapting the original $\xi_a$ on $\Delta_{a_0}$ if necessary, we can also assume that $\xi_a(v_{a_0}) = v_a$.
    
    Now, we can extend $\xi_a$ further to the tile in $\F_{a_0}^{-(k+1)}(\Delta_{a_0}\cup J_{a_0}\Delta_{a_0})$ containing the critical point, but there may be two possible extensions, since the local degree of $\F_a$ around $c_a$ is $2$. Noticing that only one of them will agree with the current extension of $\xi_a$ at the boundary tells us we can indeed further extend $\xi_a$ to $\bigcup_{0\leq n\leq k+1}\F_{a_0}^{-n}(\Delta_{a_0}'\cup J_{a_0}\Delta_{a_0}')$, still conjugating dynamics. Because we don't have any more critical points, we can continue to dynamically extend $\xi_a$ further to all of $\Omega_{a_0}$, obtaining a holomorphic motion $\xi_a: \Omega_{a_0} \to \widehat{\C}$, which conjugates the dynamics of $\F_{a_0}$ on $\Omega_{a_0}$ and $\F_a$ on $\xi_a(\Omega_{a_0}) = \Omega_a$. By the $\lambda$-Lemma of Mañé-Sad-Sullivan \cite{MSS}, $\xi_a$ extends to a holomorphic motion of $\overline{\Omega_{a_0}} = \widehat{\C}$, since $\Lambda_{a_0}$ has empty interior. This final extended holomorphic motion must still conjugate dynamics, now on the whole sphere, showing thus that $\F_a$ is quasi-conformally conjugated to $\F_{a_0}$ for all $a \in D$.
    
\end{proof}

Notice that this result cannot be true for parameters in $CR$, since the critical orbit relation must be preserved by a quasi-conformal conjugacy; indeed, notice that $\pi_a(c_a)$ is a cone point of $\Omega_a/\left< \F_a \right>$ if and only if the grand orbit of $c_a$ coincides with that of $a$ or $\infty$, which is a property that any topological conjugacy must preserve. Meanwhile, the critical relation cannot be persistent on a neighborhood of a point, or it would be persistent throughout the entire family $\F_a$ by analytic continuation. The quasi-conformal conjugacy classes of points in $\m$ are as well understood as the classes in $M_1$, since the homeomorphism between the two sets is dynamical. Thus, if we denote by $C$ the set of centers of interior components of $\m$, any connected component of $\K\setminus (CR\cup \partial\m\cup C)$ is contained in a single quasi-conformal conjugacy class, while every point in $CR\cup \partial\m\cup C$ is quasi-conformally rigid --- i.e. any other correspondence quasi-conformally conjugated to it is in fact conformally conjugated to it.


\section{Critical relations in $\K$ and in  $\widehat{\C}\setminus\K$}\label{crit-rel}

The definition of a {\it critical relation}, introduced in the previous Section for values of $a$ in $\K$, can be extended to $a\in \widehat{\C}$ in general, as follows.

\begin{defi}\label{crit_coincidence}
(i) A {\it critical relation parameter} is a value of $a$ such that the grand orbit of the critical point $Z=-1$, under $\F_a$ and $\F_a^{-1}$, contains the fixed point $Z=a$ of $J_a$ or the fixed point $Z=\infty$ of $Cov^Q$. We call the first type a {\it critical $2$-relation} and the second type a {\it critical $3$-relation}.

(ii) We describe a critical relation as `forwards' if $a$ or $\infty$ is in the forwards orbit of the critical value $v_a$ under $\F_a$. We call it `backwards' if the critical point $c_a$ ($Z=-1$) is in the forwards orbit of $a$ or $\infty$ under $\F_a$.
\end{defi}

A priori we might expect to find a further category of critical relations arising from `mixed' iteration of $\F_a$ and $\F_a^{-1}$. However the particular structure of the correspondence $\F_a=J_a\circ Cov_0^Q$ as a `reversible map of triples', \cite{BP2}, means that `mixed' critical relations can be reduced to `unidirectional' form:

\begin{lemma}
Every critical relation is a forwards critical relation (FCR) or a backwards critical relation (BCR).
\end{lemma}

\begin{proof}
For $x\in \widehat{\C}$ define $O_\pm(x)$ to be the union of the orbit of $x$ under $\F_a$ with the orbit of $x$ under $\F_a^{-1}$. By Theorem $6$ in Section 5 of \cite{BP2}, the grand orbit of $x$ under $\F_a$ and $\F_a^{-1}$ is $O_\pm(x) \cup O_\pm(J_a(x))$. The result can be deduced from this fundamental observation, or else directly as follows.

The points on the grand orbit of $z\in\widehat{\C}$ can all be expressed in the reduced form $W(z)$ where $W$ is a word with alternate `letters' $J_a$ and $Cov_0$. So every critical relation takes the form $W(-1)=a$ or $W(-1)=\infty$. It is now a straightforward exercise to verify that using the relations $J_a(a)=a$, $Cov_0(-1)\in\{-1,2\}$ and $Cov_0(\infty)=\infty$, every critical relation can be manipulated into the stated form. There are $8$ cases to consider, $4$ possibilities for the first and last `letters' of $W$, for $2$-relations and $3$-relations.
\end{proof}

We remark that our definitions permit critical relations to be both forwards and backwards critical.

\subsection{Critical relations in $\K\setminus\m$}

\begin{defi}
A torsion point of $\HH$ is a point $z\in\HH$ which is stabilised by a non-trivial subgroup of $PSL(2,\Z)$, necessarily cyclic of order $2$ or $3$, generated by a conjugate of $\sigma$ or $\rho$.
\end{defi}

\begin{prop}\label{CRinK}
(i) There are no backwards critical relations in $\K\setminus\m$;

(ii) $a\in \K\setminus\m$ is a critical relation if and only if $\psi(a)\in\HH$ is a torsion point for any canonical map $\psi$ defined at $a$.
\end{prop}

\begin{proof}

(i) When $a\in \K\setminus\m$ the forwards orbit of $a$ is contained in a transversal $\Delta_Q$ of $Q$, so cannot contain the critical point $-1$.

(ii) For $a\in \K\setminus\m$, Lemma \ref{fund_lemma} tells that there is a forwards image of $v_a$, say $Z_0$, in $\Delta'_a\cup J_a\Delta'_a$. The partial B\"ottcher map $\varphi_a$ was defined first on $\Delta'_a\cup J_a\Delta'_a$, then extended equivariantly backwards, up to and including the tile containing $v_a$. Thus our forwards $\F_a$-orbit from $v_a$ to $Z_0$ in $\Omega(\F_a)$ lifts equivariantly to a forwards $\{\sigma,\rho\}$-orbit from $\psi(a)=\varphi_a(v_a)$ to $\varphi_a(Z_0)$ in $\HH$. Since the property of being a torsion point of $\HH$ is preserved by $PSL(2,\Z)$, we deduce that $\psi(a)$ is a torsion point if and only if $\varphi_a(Z_0)$ is a torsion point. But this means that $Z_0$ projects to one of the two cone points of $\Omega(\F_a)/\langle \F_a \rangle$ and hence that $Z_0$ is the fixed point $Z=a$ of $J_a$ or the fixed point $Z=\infty$ of $Cov_0^Q$ (the point $Z=J_a(\infty)$ is in the forward orbit of $v_a$ if and only if $Z=\infty$ is too).
\end{proof}

\begin{figure}
\begin{center}
\scalebox{0.4}{\includegraphics{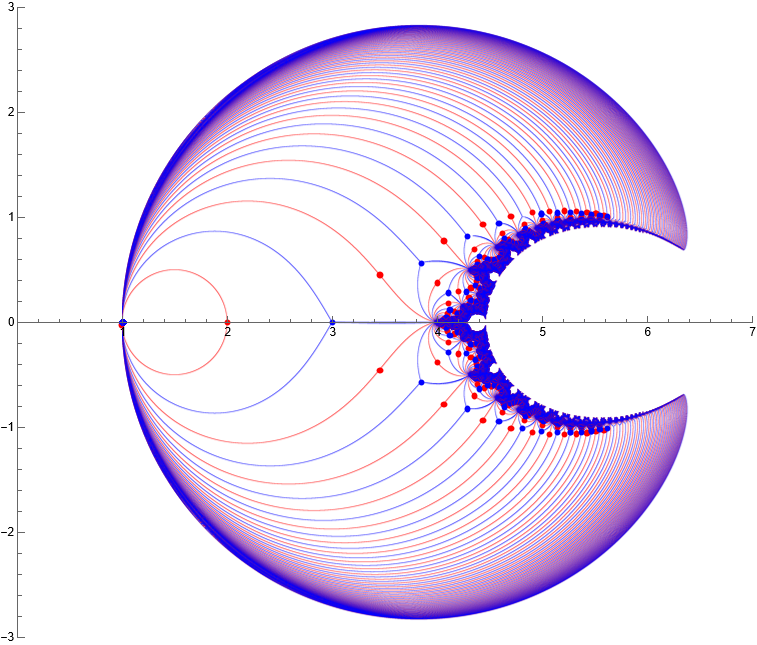}} 
\caption{The critical relation relation parameters $a\in\mathcal D \setminus \m'$ (marked with red and blue dots). All these are {\it forwards} critical relations and so each has image $\Psi(a)$ a {\it torsion point} of $\HH$ (Proposition \ref{CRinK}).}\label{coincidence_plot}
\end{center}
\end{figure}

Figure \ref{coincidence_plot} is a plot of the values of $a\in \mathcal D\setminus \m$ such that $\Psi(a)$ is a torsion point of $\HH$. The points marked by red dots are the values of $a$ where the fixed point $Z=a$ of $J_a$ is in the set $\F_a^n(v_a)$ for some $n\ge 0$, and those marked by blue dots are the values of $a$ where the fixed point $Z=\infty$ of $Cov_0^Q$ is in the set $\F^n(v_a)$ for some $n\ge 0$. Recall that $v_a=J_a(2)=2/(3-a)$ in the $Z$-coordinate. Thus, for example, when $a=2$, we have $v_a=2$ so $a\in \F_a^0(v_a)$, and when $a=3$ we have $v_a=\infty$, so $\infty\in \F_a^0(v_a)$. Recall also that $\Psi(2)=i\in\HH$, and $\Psi(\infty)=(-1+i\sqrt{3})/2\in\HH$.

\begin{remark}
Figure \ref{coincidence_plot} suggests that {\it all} the critical relations in $\K\setminus\m$ lie in $\mathcal D$, indeed that they lie in the subset $D_1\cap D_2\subset \mathcal D=\D(4,3)$, where
$D_1$ and $D_2$ are the round discs depicted in Figure \ref{D1andD2}. We warn the reader that, as in Figure \ref{mandelcorr}, the tile boundaries are not the true pull-backs via $\Psi^{-1}$ of the standard tile boundaries in $\HH$, but nevertheless the tile vertices are the true pull-backs of torsion points in $\HH$: see Section \ref{plots} for details.
\end{remark}

\subsection{Critical relations in $\widehat{\C}\setminus \K$}\label{CRoutside}

From Proposition \ref{CRinK}(i) we know that backwards critical relation points (BCR) are necessarily outside $\K$. In this subsection we exhibit a computer plot displaying a large number of BCR points $a$, and we prove that for each of these points the correspondence $\F_a$ has a finite parabolic cycle.

\begin{figure}
\begin{center}
\scalebox{.5}{\includegraphics{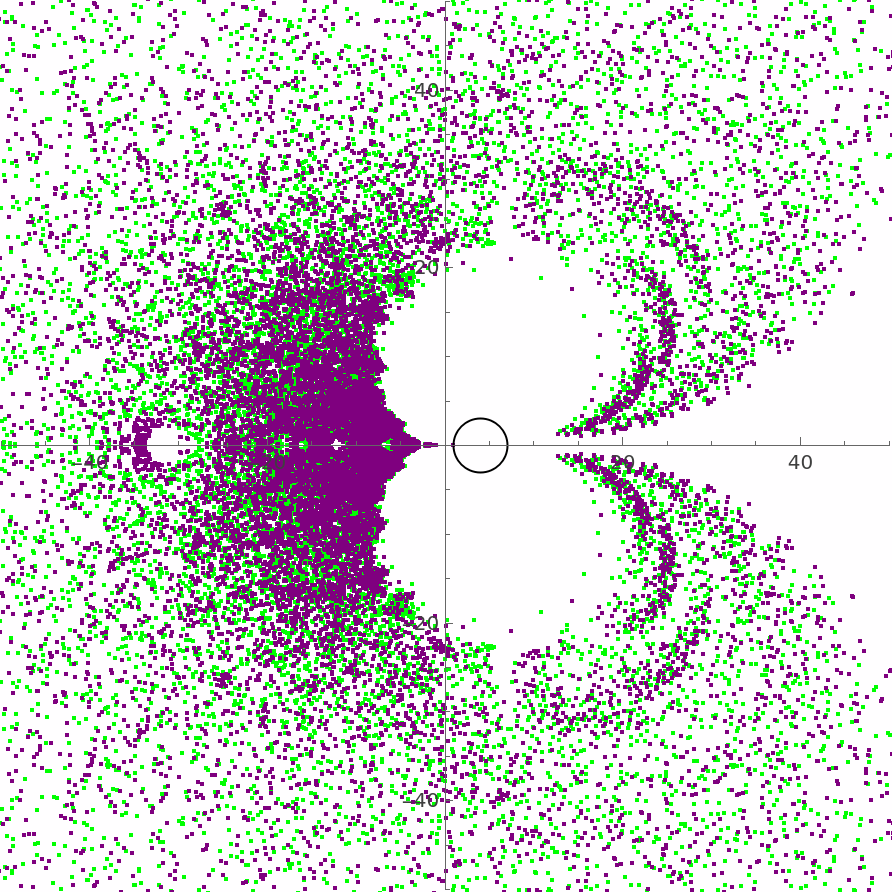}}
\caption{Values of the parameter $a$ where the critical value $v_a$ is in the forwards $\F_a$-orbit of $Z=a$ or $Z=\infty$ (see Section \ref{CRoutside}): these points are all outside $\K$, and are believed to be dense in the complement of $\K$. The small round disc plotted within $\K$ is $\mathcal D:=\D(4,3)$; just visible within $\mathcal D$ is $\m$.}\label{kleinC}
\end{center}
\end{figure}

Figure \ref{kleinC} is a plot of values of $a$ such that the critical value $v_a$ is in the forwards orbit under $\F_a$ of either the fixed point $Z=a$ of $J_a$ (in which case $a$ is marked in purple) or in the forwards orbit of the fixed point $Z=\infty$ of $Cov_0^Q$ (in which case $a$ is marked in green). These parameter values are necessarily outside $\K$, but we shall see that certain of them lie on the boundary $\partial\K$. The computer plot suggests the conjecture that the set of all green and purple points accumulates on $\partial\K$, and indeed everywhere outside $\partial\K$; it also provides numerical evidence supporting our conjecture that the Klein combination locus $\K$ is a topological punctured disc, with inward pointing cusps at $a=-1$ and $a=7$, and possibly at other critical relation parameters on $\partial\K$.

In the next Section we will investigate the question of which of these critical relations lie on the {\it boundary} $\partial\K$ of $\K$, and we will present  conjectures concerning the overall structure for $\K$. In preparation, we now prove a Proposition about backwards critical relations $a$ in general. If the relation has the form $-1 \in \F_a^n(a)$ or $-1\in \F_a^n(\infty)$ we call $n$ (supposed minimal) its {\it length}.

\begin{prop}\label{crit_cycles}
(i) Every backwards critical $2$-relation of length $n$ can be completed to a parabolic cyclic orbit of $\F_a$ of length $2n+1$.

(ii) Every backwards critical $3$-relation of length $n$ can be completed to a parabolic cyclic orbit of $\F_a$ of length $2n$.
\end{prop}

\begin{proof}
(i)  The two branches of $\F_a$ at $Z=-1$ are $-1\mapsto J_a(-1)$ and $-1 \mapsto J_a(2)(=v_a)$. Adding the first of these two links to our partial orbit $a \to \ldots \to -1$ and then following that by $J_a(-1) \to \ldots \to J_a(a)=a$ (recall that $J_a$ conjugates $\F_a$ to $\F_a^{-1}$) we obtain a cycle of length $2n+1$:
$$a \to \ldots \to -1\to J_a(-1) \to \ldots  \to J_a(a)=a.$$
This is parabolic, with derivative $\pm 1$, since it is the same cycle as its `time-reversed' $J$-conjugate.

(ii) A backwards $3$-relation gives us $\infty \to J_a(\infty)\to \ldots \to -1$ (at the first step we are using the fact that $Cov_0(\infty)=\infty$); again we add the link $-1 \mapsto J_a(-1)$ and follow it this time by the length ($n-1)$ composition $J_a(-1)\to\ldots \to J_a(J_a(\infty))=\infty$ to obtain a parabolic cycle of length $2n$:
$$\infty \to J_a(\infty)\to \ldots \to -1 \to J_a(-1)\to J_a(J_a(\infty))=\infty.$$
\end{proof}

Note that while these parabolic cycles pass through the critical {\it point} $Z=-1$, they do not pass through the critical {\it value} $J_a(2)$: they pass through the other image of $-1$ under $\F_a$, the point $J_a(-1)$.


\section{Conjectured structure of $\K$ and $\Psi(\K)$}\label{structure_K}

\begin{figure}
\begin{center}
\includegraphics[width=3cm]{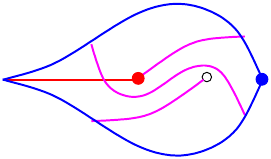}\hskip1.0cm\includegraphics[width=3.5cm]{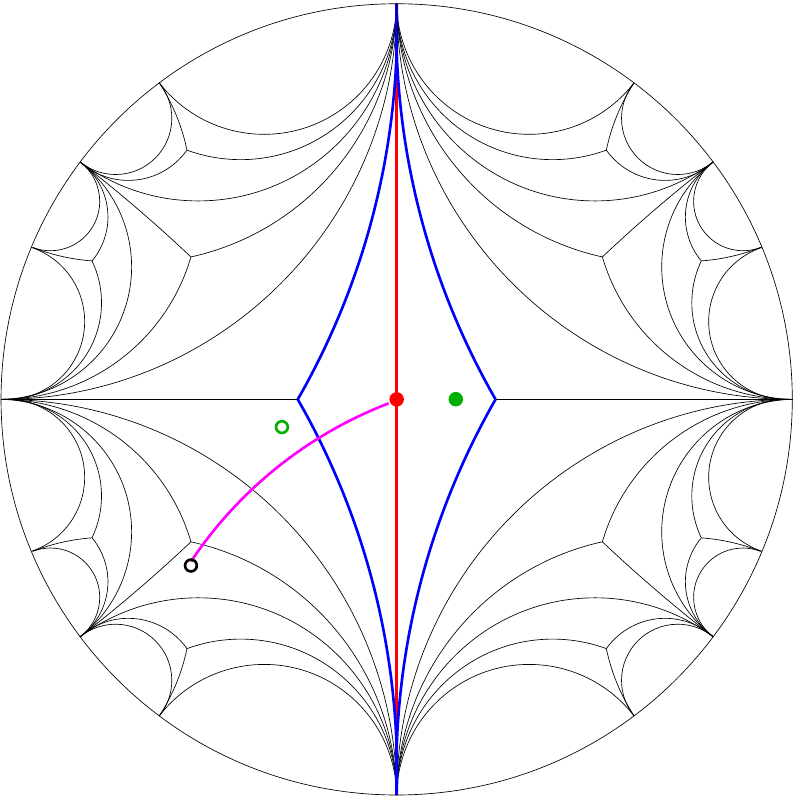}\hskip1.0cm\includegraphics[width=3.5cm]{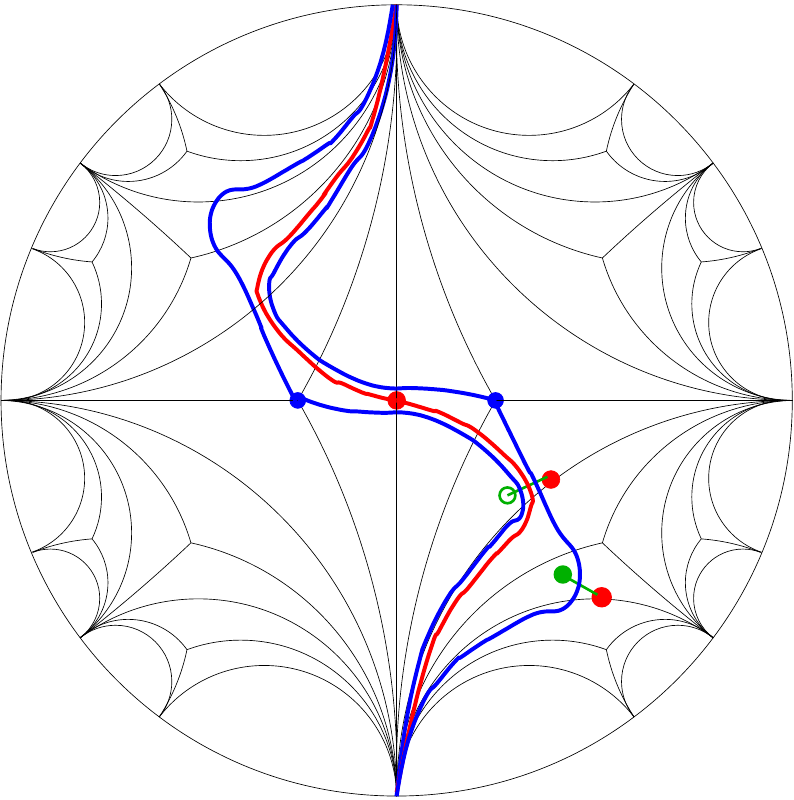}
\vskip0.5cm
\includegraphics[width=3cm]{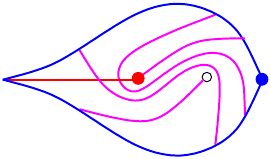}\hskip1.0cm\includegraphics[width=3.5cm]{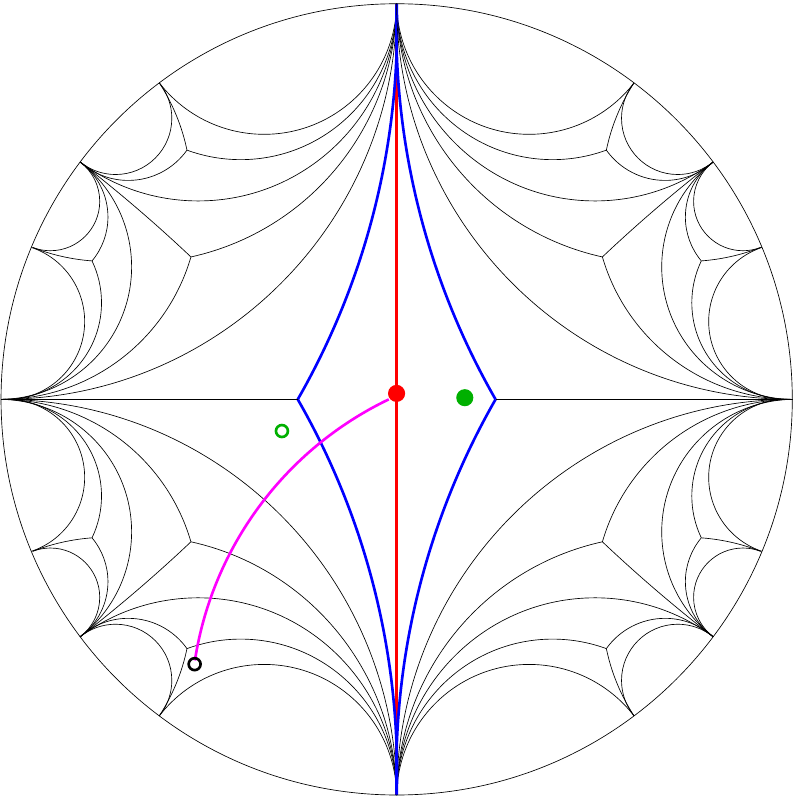}\hskip1.0cm\includegraphics[width=3.5cm]{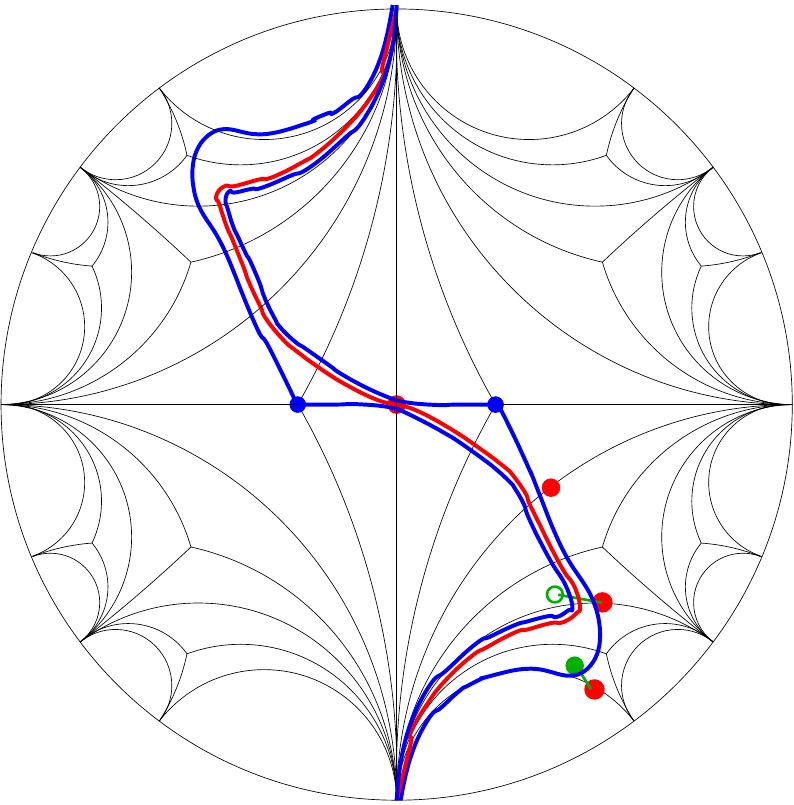}

\caption{Left: Paths $\ell_n^\mathcal{U}$, ($n=2,3$), from $a_0=1+\delta$ to $a_n$, on the dynamical space orbifold for $\F_{a_0}$ (compare Figure \ref{newfig5} for $n=1$). Middle: The images of these paths in $\HH$, under the conformal map $\varphi_{a_0}$ (the images of the critical point $c_{a_0}$ and value $v_{a_0}$ are marked by a green circle and green disc). Right: Conjectured $\varphi_a(\Delta_a\cup J_a(\Delta_a))\subset \HH$, when $a$ has moved along $\ell_n^\mathcal{U}$ to near $a_n$: the critical point and value tend to $(\sigma\rho)^{n-1}(i)$ and $(\sigma\rho)^n(i)$ when $a$ approaches $a_n$.}\label{an_path}
\end{center}
\end{figure}

\begin{figure}
\begin{center}
\scalebox{.35}{\includegraphics{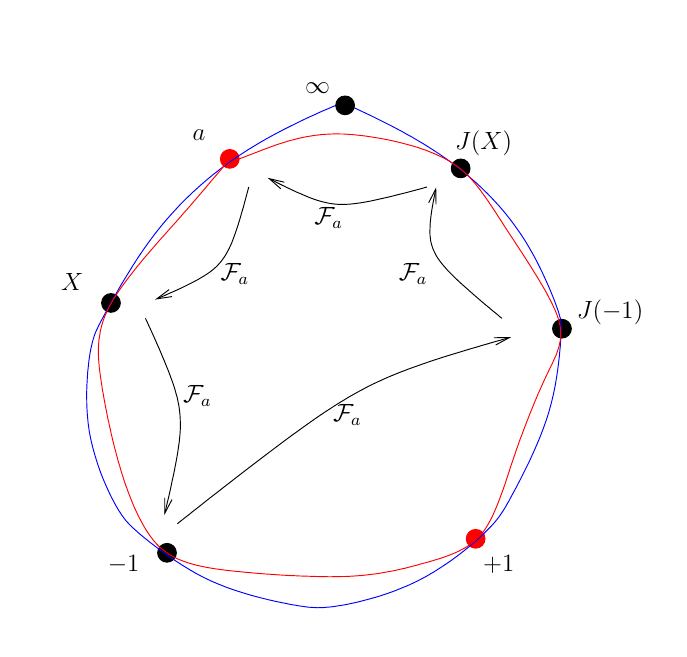}\includegraphics{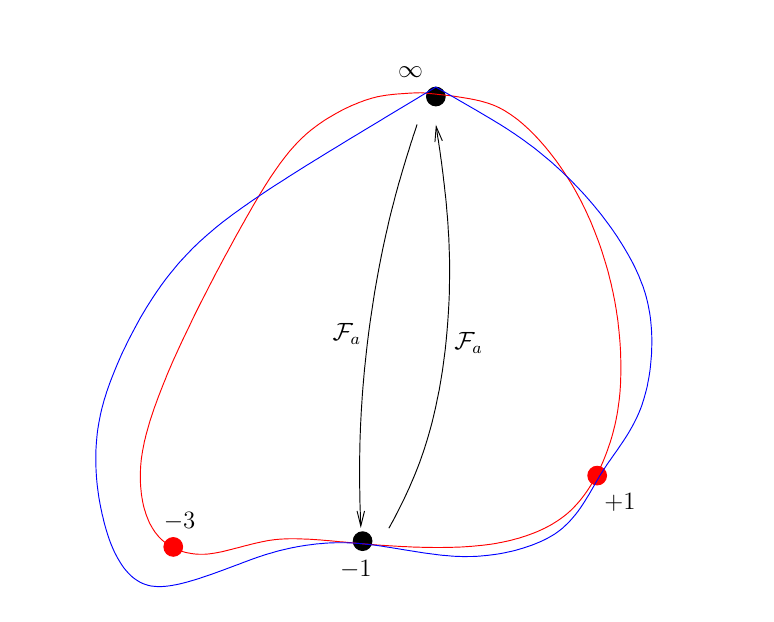}} 
\caption{Sketches of periodic orbits of $\F_a$ at the critical relation parameters $a=a_3$ and $a=b_1$. A pinched Klein combination pair of fundamental domains is illustrated schematically for $a_3$ (on the left): $\Delta_Q$ is the set bounded by the blue curve, and $\Delta_{J_a}$ is the exterior of the red curve. No such pair can exist for $b_1(=-3)$, as the angle of $2\pi/3$ at $\infty \in \partial\Delta_Q$ ensures that $\partial\Delta_{J_a}$ `overlaps' $\partial\Delta_Q$.}\label{cyclefigs}
\end{center}
\end{figure}

\subsection{Critical relations on $\partial\K$}\label{partialK}

\begin{conj}\label{conj_an}
For each $n \ge 1$ there exists a (unique) backwards critical $2$-relation parameter $a_n\in \partial\K\cap \mathcal U$ with the following properties, which also, with obvious changes, apply to the complex conjugate of $a_n$, the parameter $a'_n\in \partial\K\cap\mathcal L$.

\begin{enumerate}
    \item The critical point $-1 \in \F_{a_n}^{n-1}(a_n)$ (so the critical value $v_{a_n}\in \F_{a_n}^n(a_n)$).

    \item There exists a path $\ell_n^\mathcal{U}$ on $\Omega_{1+\delta}$ from $Z=v_{1+\delta}$ to $Z=a_n$ which projects injectively to a path on the orbifold $\Omega_{1+\delta}/\langle \F_{1+\delta}\rangle$. See Figure \ref{an_path}.

    \item There exists a pair of fundamental domains for $Cov^Q$ and  $J_{a_n}$ with boundaries touching only at the periodic cycle associated to the critical relation by Proposition \ref{crit_cycles}, and at the parabolic fixed point $Z=1$. The cycle is arranged in the same order as a rotation through $2\pi/(2n-1)$ (Figure \ref{cyclefigs}).

    \item The correspondence $\F_{a_n}$ is `discrete' (in the same sense as $\F_{-1}$ in Proposition \ref{F_-1}). It is conjugate to the correspondence obtained from $\F_{1+\delta}$ by contracting the path $\ell_n^\mathcal{U}$ on $\Omega_{1+\delta}$ to a point.

    \item As $a$ tends to $a_n$ along any path in $\K$ from $1+\delta$ to $a_n$ homotopic to $\ell_n^\mathcal{U}$ relative to its ends, $\lim_{a\to a_n}\Psi^\mathcal{U}(a)=(\sigma\rho)^n(i) \in \HH$.
\end{enumerate}
\end{conj}

\begin{remark}\label{boundary_remark}

1. Conjecture \ref{conj_an} generalizes the results we proved for  $a_1=a_1'=-1$ in Section \ref{second_extension}: indeed the properties could in principle be verified by similar methods, if we had explicit values for all the $a_n$. In \cite{BC} the parameter $a_2=(-9+i\sqrt{23})/2$ is shown to satisfy the properties listed for $n=2$: the regular set $\Omega_{a_2}$ is shown to be tessellated by ideal hexagons, and the limit set of $\F_{a_2}$ is plotted (it resembles the $1$-skeleton of a tetrahedron, decorated with dendrites). For general $n$, we note that Proposition \ref{accessible} guarantees the existence, for any $\epsilon >0$, of a based fundamental domain $\Delta\subset\HH$ for $PSL(2,Z)$ containing points $\epsilon$-close to $(\sigma\rho)^n(i)$ (in the hyperbolic metric). Hence if the inclusion $\psi(\K\setminus\m)\subseteq\mathcal{W}_{acc}$ of Section \ref{isotopy_description} is a surjection, Conjecture \ref{conj_an} will follow: in particular there will exist a pair $(\Delta_Q, \Delta_{J_{a_n}})$ for $\F_{a_n}$ such that the complement of their union consists of $2n$ points, namely $+1$ and the critical cycle, and it will follow that $\F_{a_n }$ is `discrete'.

2. It is an open question whether there are any backwards $3$-critical relations on $\partial\K$. In Figure \ref{kleinC} one can see parameter sequences $b_n$ and $b_{n'}$ of backwards $3$-critical correspondences close to $\partial\K$, and we know from Proposition \ref{accessible} that there are candidate based fundamental domains $\Delta\subset\HH$ for $PSL(2,\Z)$ with boundaries arbitrarily close to these torsion points. However, even if $b_n$ and $b_{n'}$ lie on $\partial\K$ it seems unlikely that the correspondences $\F_{b_n}$ and $\F_{b_{n'}}$ can be discrete. Consider $b_1$: here the (length $2$) critical cycle is
$$\infty \to \F_{b_1} (\infty) =J_{b_1} (\infty) = -1\to J_{b_1} (-1)=\infty$$
and the equality $J_{b_1}(-1)=\infty$ tells us that $b_1=-3$. It is not possible to find a pair $(\Delta_Q,\Delta_{J_{b_1}})$ whose boundaries meet only at the periodic cycle $\infty\to -1\to \infty$ and the parabolic fixed point $+1$ (see the right-hand sketch in Figure \ref{cyclefigs}).

\end{remark}

\subsection{The global structure of $\K$}\label{global}

Let $\K_0$ denote the component of $\K$ containing $\mathcal{D}=\D(4,3)$, let $I$ denote the real interval $(-1,+1)$, and let $\K_0^{cut}$ denote $\K_0\setminus I$.

\begin{conj}\label{conjK}
\begin{enumerate}
    \item $\K$ is connected, that is, $\K_0=\K$;

    \item $\K_0^{cut}\setminus \m$ is simply-connected;

    \item $\Psi$ extends uniquely from ${\mathcal D}\setminus\m$ to $\K_0^{cut}\setminus \m$, and 

    \item this extension is injective on $\K_0^{cut}\setminus \m$.
\end{enumerate}
\end{conj}

There are parallels between the structure of $\K$ and that of the space of conjugacy classes of faithful discrete representations in $PSL(2,\C)$ of the free product $C_2*C_3$ of cyclic groups of orders $2$ and $3$, the `discreteness locus' of $C_2*C_3$. The latter has the structure of a topological punctured disk that sits inside the moduli space of {\it all} representations of $C_2*C_3$ in $PSL(2,\C)$. Indeed this is more that an analogy: our family of correspondences $\F_a$ is a slice through a $2\C$-parameter family of correspondences $\F_{a,k}$, another slice through which is this moduli space of representations of $C_2*C_3$ (see \cite{BP1,BH}). Much more is known about this moduli space (an `elliptic cousin of the Riley slice') than about $\K$ (see \cite{KS,EMS}, and for a beautiful plot of the discreteness locus inside the moduli space see the front cover of the Notices of the AMS for December 2016: this plot, by Yasushi Yamashita, illustrates Gaven Martin's search for the smallest volume hyperbolic orbifold, though from our point of view the plot is inside out, as the complement of the discreteness locus is in the centre). 

\begin{figure}
\begin{center}
\scalebox{.23}{\includegraphics{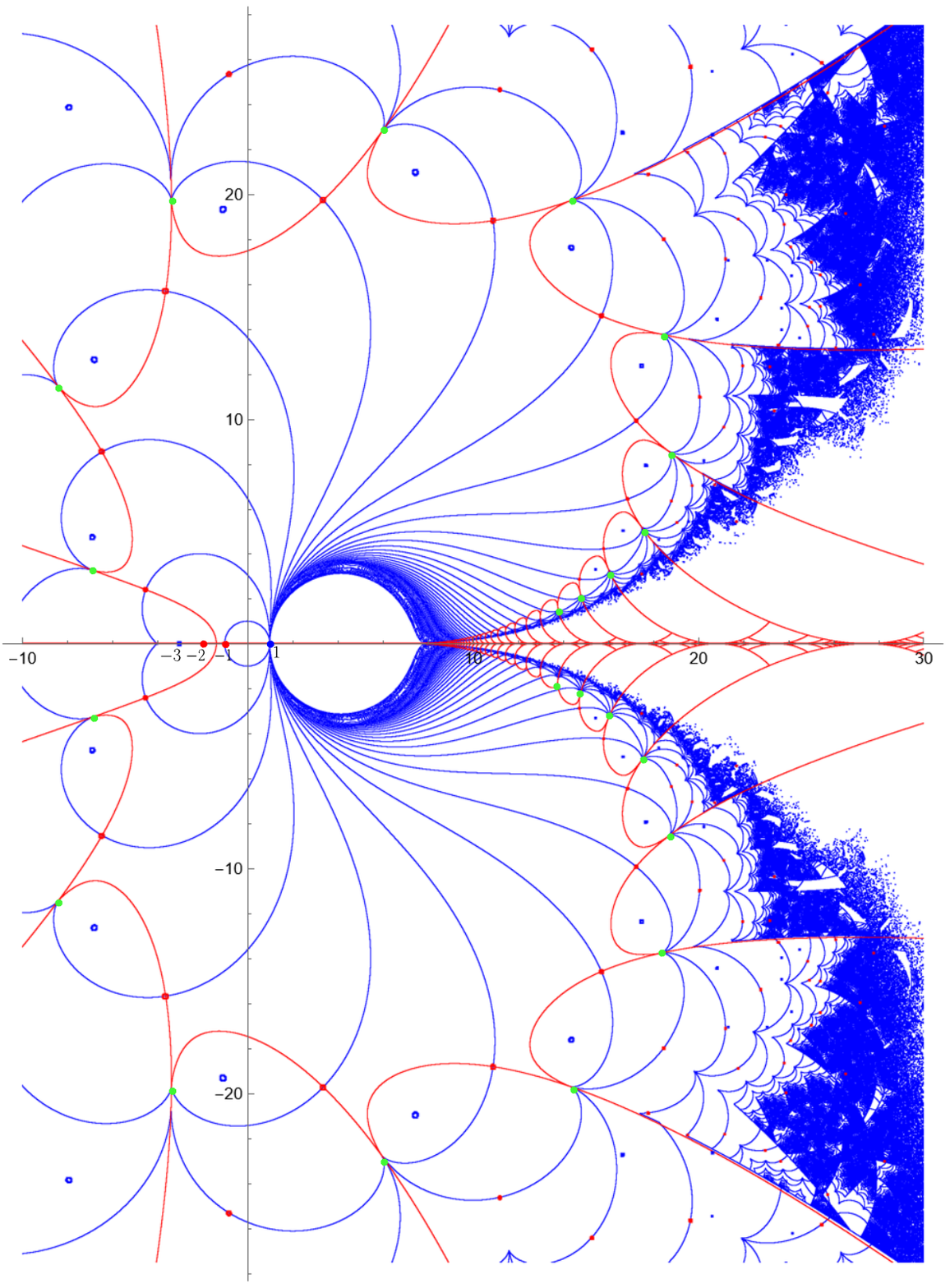}}
\end{center}
\caption{\small Initial points $z_0=a$ are outside the unit circle; the iteration proceeds at the $n$th step only if there is a value of $\F_a(z_n)$ outside the unit circle, and uses this as $z_{n+1}$. The red lines correspond to $\{a:\mathcal{F}_a^n(a) \in [-1, +1]\}$ for some $n$, and the blue lines correspond to $a$ such that $ \mathcal{F}_a^n(a) $ lies on the unit circle. The red points correspond to $\mathcal{F}_a^n(a) = -1 $, the green points to $ \mathcal{F}_a^n(a) = +1 $, and the blue points to $ \mathcal{F}_a^n(\infty) = -1 $.}\label{bdry}
\end{figure}

A Kleinian group action on $\widehat{\C}$ extends to an action by isometries on the hyperbolic $3$-space bounded by $\widehat{\C}$, providing extra tools with which to analyse the (conformal) action of the group on $\widehat{\C}$. In particular, for each rational $p/q$ with $0\le p/q<1$ there is a `pleating ray' in the discreteness locus for $C_2*C_3$, along which a pair of attracting and repelling $q$-cycles (of rotation number $p/q$) have real trace. These rays run from a puncture point at the centre of the locus, out to the boundary, where the two $q$-cycles collide, to become a parabolic cycle. We conjecture that there are analogous rays in $\K$, running out from the puncture point $a=+1$ to the critical relation parameters $a_n$ and $a_n'$ on $\partial\K$. However there is a major difference between the two situations: the pleating rays fill the discreteness locus for $C_2*C_3$, once one has added rays for irrational rotation numbers, but in $\K$ we have only a discrete set of rays, one for each rotation number of the form $1/(2n-1)$ (mod $1$). A feature the two situations have in common is the existence of `harmonics' on the rays extended beyond the parabolic points: in the group case these are a sequence of isolated discrete, but no longer faithful, representations of $C_2*C_3$ (giving the small volume orbifolds of Gaven Martin's search). In the correspondence case they are isolated `discrete' correspondences outside $\K$: see Appendix II (Section \ref{triangles}).

Figure \ref{bdry} is a computer plot displaying the role of the sequences $(a_n)$ and $(a_n')$ in the structure of $\K$ in a different way: the rays connecting these critical relation parameters to the unit circle in the $a$-plane are not the conjectured `pleating rays' referred to above.

\subsection{A possible scenario for $\Psi(\K^{cut})\subset\HH$}\label{scenario}

Figure \ref{schematic} is a schematic illustration of a scenario for $\Psi(\K^{cut})$ which is compatible with the structure of $\K$ proposed in Conjectures \ref{conj_an} and  \ref{conjK}: it is speculative, due to our limited knowledge both of the boundary of $\K$ and of the behaviour of the map $\Psi$ in a neighborhood of this boundary.

\begin{figure}
\begin{center}
\scalebox{0.6}{\includegraphics{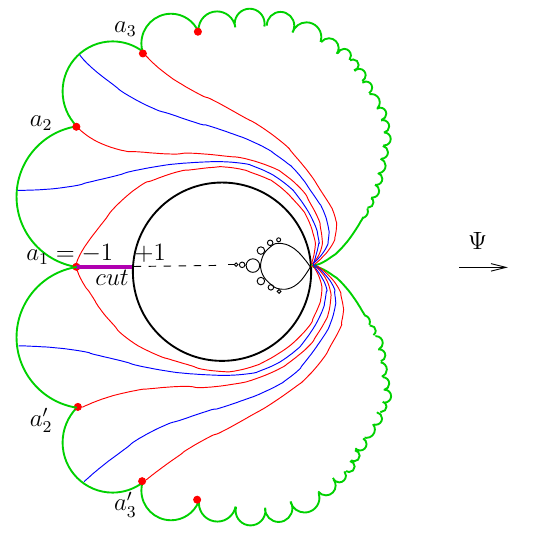}}\scalebox{0.4}{\includegraphics{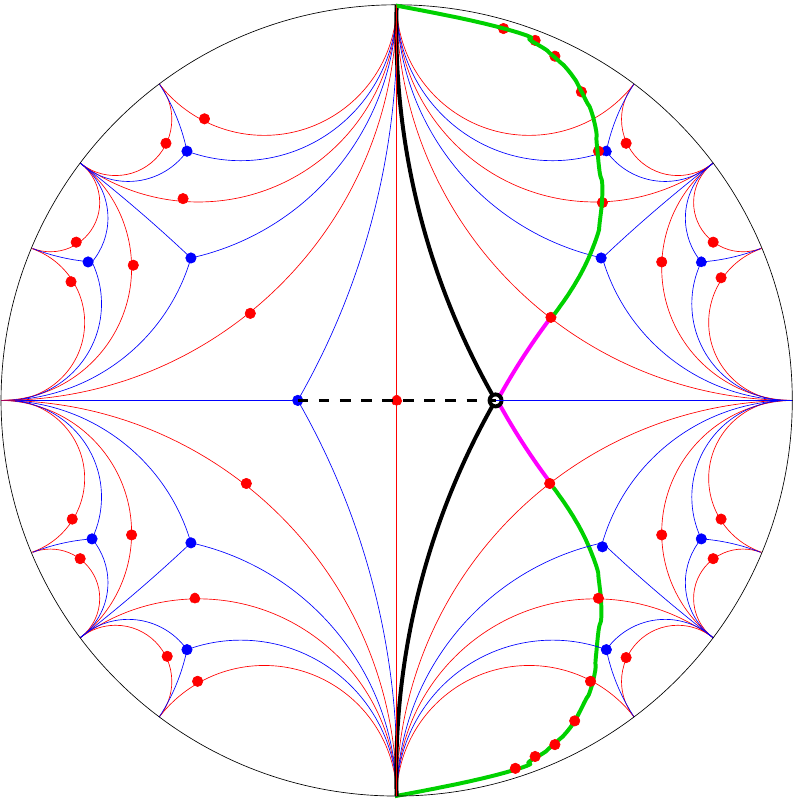}}

\scalebox{0.6}{\includegraphics{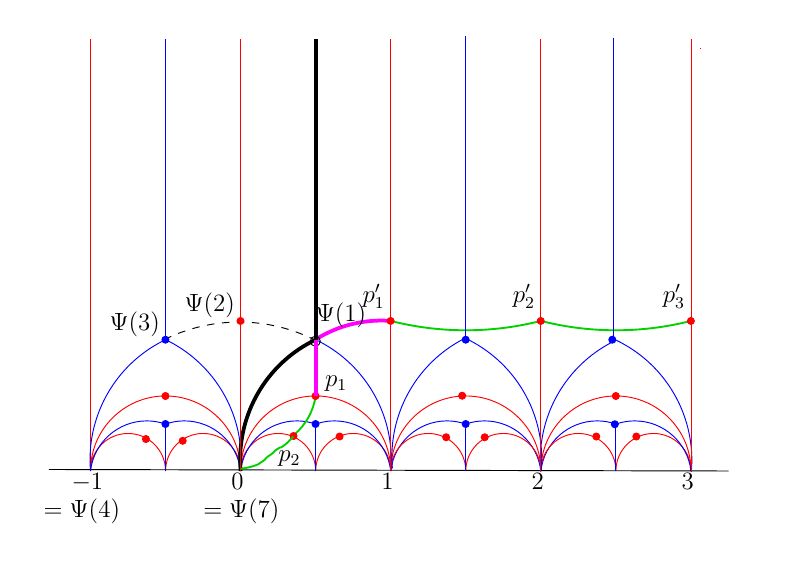}}
\caption{Schematic sketches of the (conjectured) set $\K$ and $\Psi(\K^{cut}\setminus\m)\subset \HH$. The set $\K$ (on the left) is not to scale. The image $\Psi(K^{cut}\setminus \m)\subset \HH$ is sketched on the right, where $\HH$ is represented as the Poincar\'e disc $\D$, and also sketched below on the half-plane $\HH$. In both, this image is the set above and to the left of the green and magenta curve: the image in $\D$ may be visualized as obtained from $\K$ by excising $\m$ (including its root point) from the right-hand side of $\D(4,3)\subset\K$, cutting along the interval $[-1,+1]$ on the left-hand side of $\K$ and opening the cut up to an angle of $2\pi/3$, then rotating the illustration of $\K$ through $\pi$, to move the cut to the right-hand side and the boundary of $\m$ to the left-hand side. The boundary of the modular Mandelbrot set $\m$ is mapped to the boundary of $\HH$ along the interval $[-\infty,0]\subset \widehat\R$. The image $\Psi^{cut}(\mathcal D)$ of $\mathcal D=\D(4,3)$ has boundary approximately the curve drawn in black. The set of all torsion points in $\HH$ accumulates on $\widehat \R=\R\cup\{\infty\}$ (only the first few layers are shown). The torsion points $p_n$ and $p_{n'}$ ($n\ge 1$) are the images under $\Psi$ of the critical relation parameters $a_n$ and $a_{n'}$. It is an open question whether the green curve passes through the nearby $3$-torsion points (blue dots) between adjacent $p_n$'s (resp. $p'_n$'s): see Remark \ref{boundary_remark}. The torsion points within $\Psi(\mathcal D)$ (i.e. to the left of the solid black curve) are all accounted for by Figure \ref{coincidence_plot}.}\label{schematic}
\end{center}
\end{figure}

Three key open questions:

\begin{enumerate}
    \item If we coordinatize the hyperbolic plane as the unit disc $\D$, with the fixed point of $\sigma\rho\sigma$ placed at the origin in $\D$, does the multi-valued function $\Psi: \K\setminus \m \to \D$ have the form $a \to -(f(a-1))^{2/3}$, with $f$ a (single-valued) real analytic function, and is the restriction of $\Psi$ to $\Psi^{cut}:\K^{cut} \to \HH$ a bijection from $\K^{cut}$ onto its image?

    \item What is the structure of $\partial\K$ between each $a_n$ and $a_{n+1}$, and what is the structure of the corresponding part of $\partial(\Psi(\K))$?  

    \item What structures are there in parameter space {\it outside} $\K$? We have seen members of the family $\F_a$ which act `discretely' but which have `fundamental domains' with more than one component, for example the correspondences $\F_{a_n}$ of Section \ref{structure_K}, and $\F_{c_q}$ of Section \ref{triangles}. Are these isolated examples or are there regions of stable behavior outside $\K$?
\end{enumerate}

These questions are related. The boundary $\partial\K$ of $\K$ should be the natural boundary of the analytic function $f$, and contain a dense set of singularities. Given the structure of related moduli spaces (\cite{KS,EMS}) one might anticipate a dense set of cusps on $\partial\K$, each corresponding to an $\F_a$ which has a parabolic orbit. However the points $a_n,a_n'$  ($n\ge 1$) only accumulate at the root point $a=7$.


\section{Appendix I: Plotting the tessellation}\label{plots}

\begin{figure}
\begin{center}
\scalebox{.42}{\includegraphics{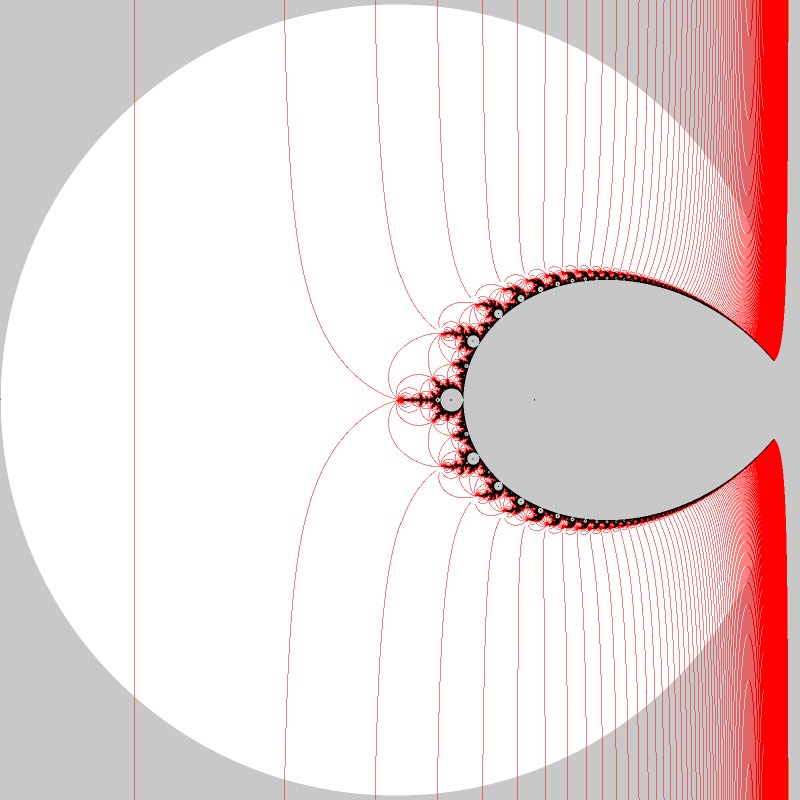}} 
\hfill
\scalebox{.42}{\includegraphics{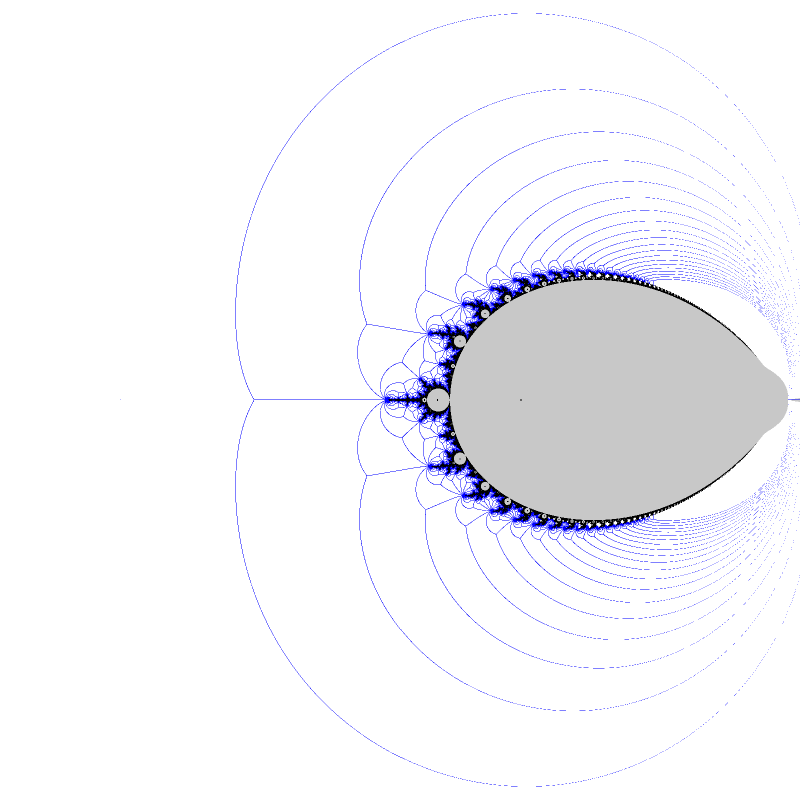}} 
\caption{Tesselations of a neighbourhood of $M_\Gamma$, plotted using different choices of fundamental domain boundary. Imperfections near $a=7$ are due to the slow speed of the algorithm near a parabolic fixed point.}\label{tessellations}
\end{center}
\end{figure}

We would be in a better position to make conjectures concerning the sets $\K$ and $\Psi(\K^{cut})$, their tilings, and their boundaries, if we had accurate computer plots of the tessellation to guide us. 
However we do not know of any algorithm to plot the pull-backs under $\Psi$ of the lines making up a standard tessellation of ${\mathbb H}$, the analogues of the Douady-Hubbard {\it parameter rays} for quadratic polynomials. The difficulty is that while the punctured sphere with two cone points $\HH/PSL(2,\Z)$ carries a canonical marking of lines joining the cone points to one another and to the puncture point, determined by its conformal geometry, we do do not know of any way to determine these lines in $\mathcal S=\Omega(\F_a)/\langle \F_a \rangle=\overline\Delta_a/{\rm (side-pairing)}$ despite our knowledge that $\mathcal S$ is conformally isomorphic to $\HH/PSL(2,\Z)$.

While the outer region (and much of the boundary) of $\mathcal K$ remains \textit{terra incognita}, the best we can hope to do at present is to plot approximate pictures of the tiling of pinched neighbourhoods of $M_\Gamma$ in $\mathcal K$, in particular the pinched neighbourhood $\D(4,3)$. In Figure \ref{mandelcorr} (in the Introduction) and Figure \ref{tessellations} (in this Section) we display three different tessellations of $\D(4,3)\setminus M_\Gamma$, each plotted using a condition that some image $\F_a^n(v_a)$ ($n\ge 0$) of the critical value $v_a$ lies on a fundamental domain boundary in the dynamical plane. In Figure \ref{mandelcorr} the boundaries of both the `standard' fundamental domains $\Delta_{Q}^{st}$ and $\Delta^{st}_{J_a}$ were used. In Figure \ref{tessellations} in the left-hand plot the fundamental domain $\Delta_{J_a}$ was taken to be the exterior of the circle in the $Z$-plane with diameter $[1,a]$. The fundamental domain $\Delta_{Q}$ can be suitably adjusted to form a Klein combination pair with this $\Delta_{J_a}$ but this is not straightforward to do algorithmically, so in the plot we only mark points $a$ such that $\F_a^n(v_a)$ lies in $\partial\Delta_{J_a}$ for some $n\ge 0$. In the right-hand plot we used a fundamental domain $\Delta_{Q,a}$ for $Cov^Q$ whose boundary moves holomorphically in the $Z$-plane as $a$ varies: to be precise, we took the images under $Z=(W+1/W)/2$ of the lines $W=e^{\pm \pi i/3}(1+te^{i\theta})$ where $\theta=\arg((a-1)/(7-a))$. We know that each such $\Delta_{Q,a}$ has a partner $\Delta_{J_a}$ making up Klein combination pair, but its boundary is not easy to compute algorithmically, so in this plot we have only marked the points $a$ such that some $\F_a^n(v_a)$ lies in $\partial\Delta_{Q,a}$. Observe that in all three plots the tile vertices are in the same positions: they correspond via $\Psi$ to the same torsion points in ${\mathbb H}$ (this is a consequence of Proposition \ref{CRinK}).

We believe that the right-hand plot in Figure \ref{tessellations} is the nearest of the three plots to the `true' tessellation of $\D(4,3)\setminus M_\Gamma$, because it fits best with the picture suggested outside $\D(4,3)$ by Figure \ref{schematic}.  


\section{Appendix II: Parameters outside $\partial\K$ at which $\F_a$ is `discrete'; analogy with triangle groups}\label{triangles}

As always, let $Q(Z)=Z^3-3Z$. The inverse image $Q^{-1}(-L)$ of the line segment $-L:=[2,\infty)\subset\C$ consists of $-L$ together with a curve $C_Q$ which crosses the real axis at $Z=-1$ and runs off towards infinity asymptotically to the directions of argument $\pm 2\pi/3$. The component of $\widehat\C\setminus C_Q$ to the left of $C_Q$ bounds a fundamental domain for $Cov^Q$ (an alternative to $\Delta_Q^{st}$ defined in Section \ref{standard_domains}), and $Cov_0^Q$ acts on $C_Q$ by complex conjugation. For $a<-1$ real, the two points where the circle $C_J$ which has diameter $[a,1]$ intersects $C_Q$ are each a fixed point of (a branch of) $\F_a$ since they are interchanged by both $Cov^Q_0$ and $J_a$. Moreover, these are neutral fixed points and the derivative of $\F_a$ at each is $e^{2\theta i}$ and $e^{-2\theta i}$ respectively, where $\theta$ is the angle between the curves (which act like mirrors). The angle $\theta$ decreases monotonically to $0$ as $a$ increases from $-\infty$ to $-1$. The parameter $c_q$ is defined to be the value of $a$ such that $\theta=\pi/q$.

\begin{figure}
\begin{center}
\includegraphics[width=5cm]{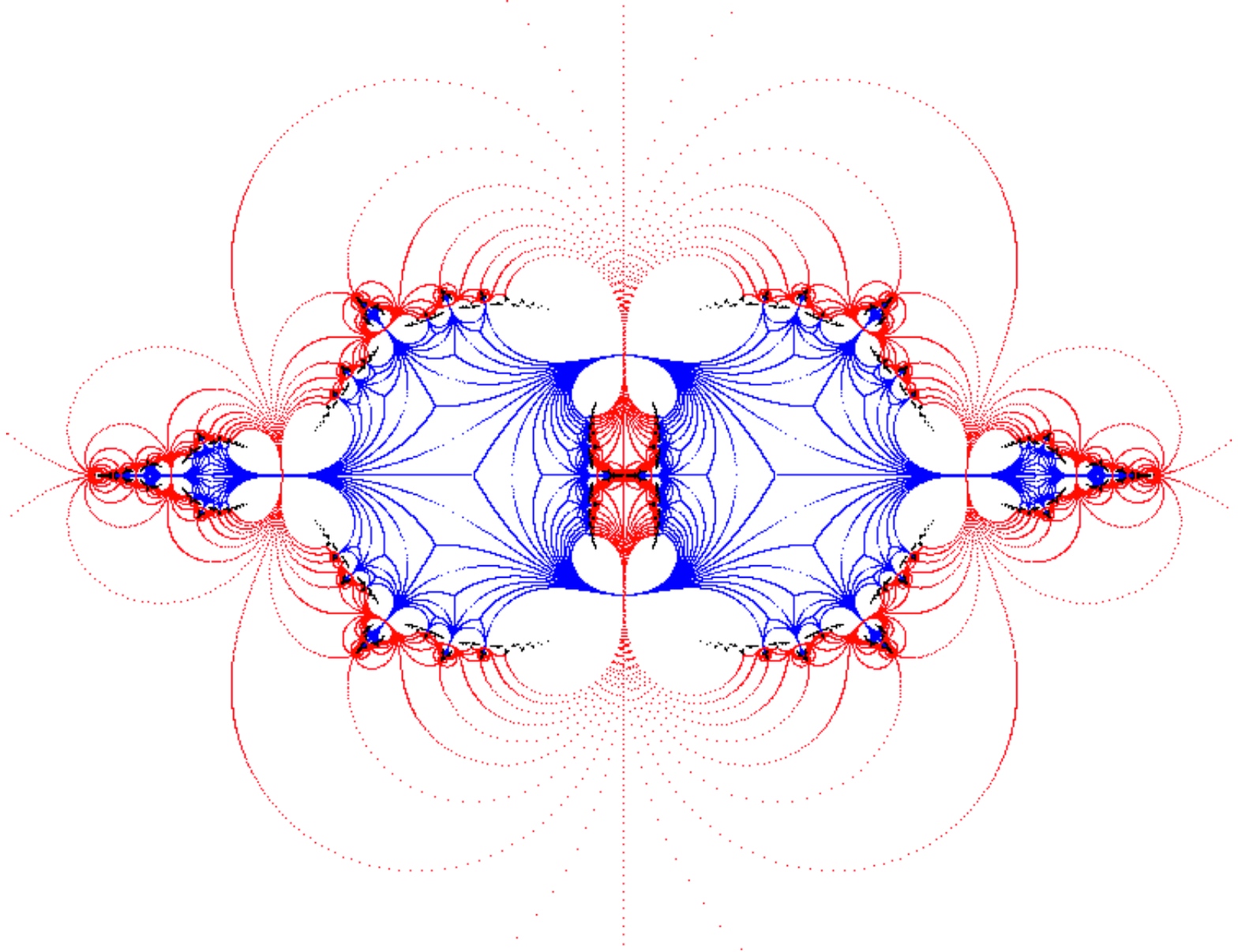}\hspace{0.1cm}
\includegraphics[width=5cm]{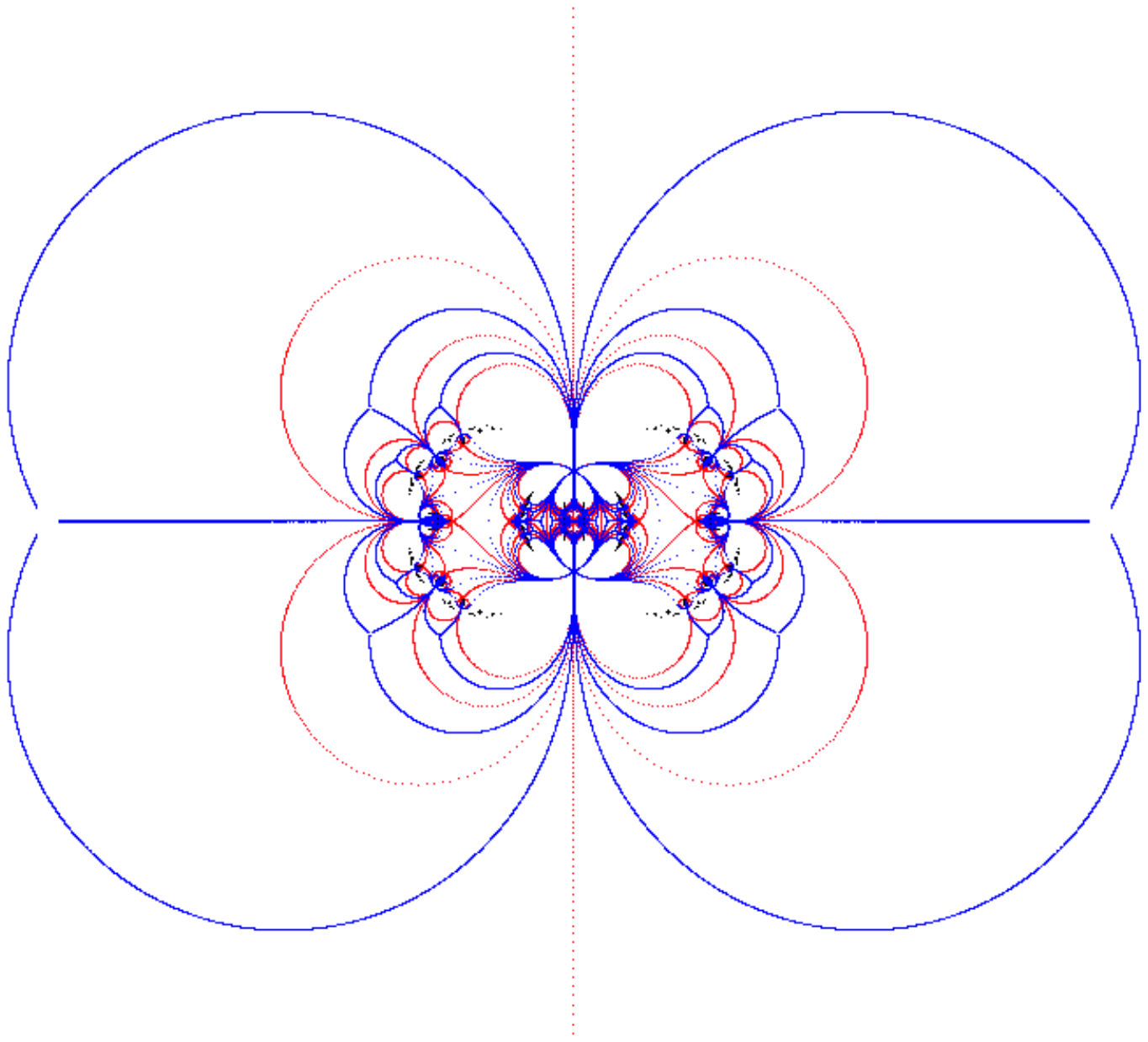} 

\includegraphics[width=4cm]{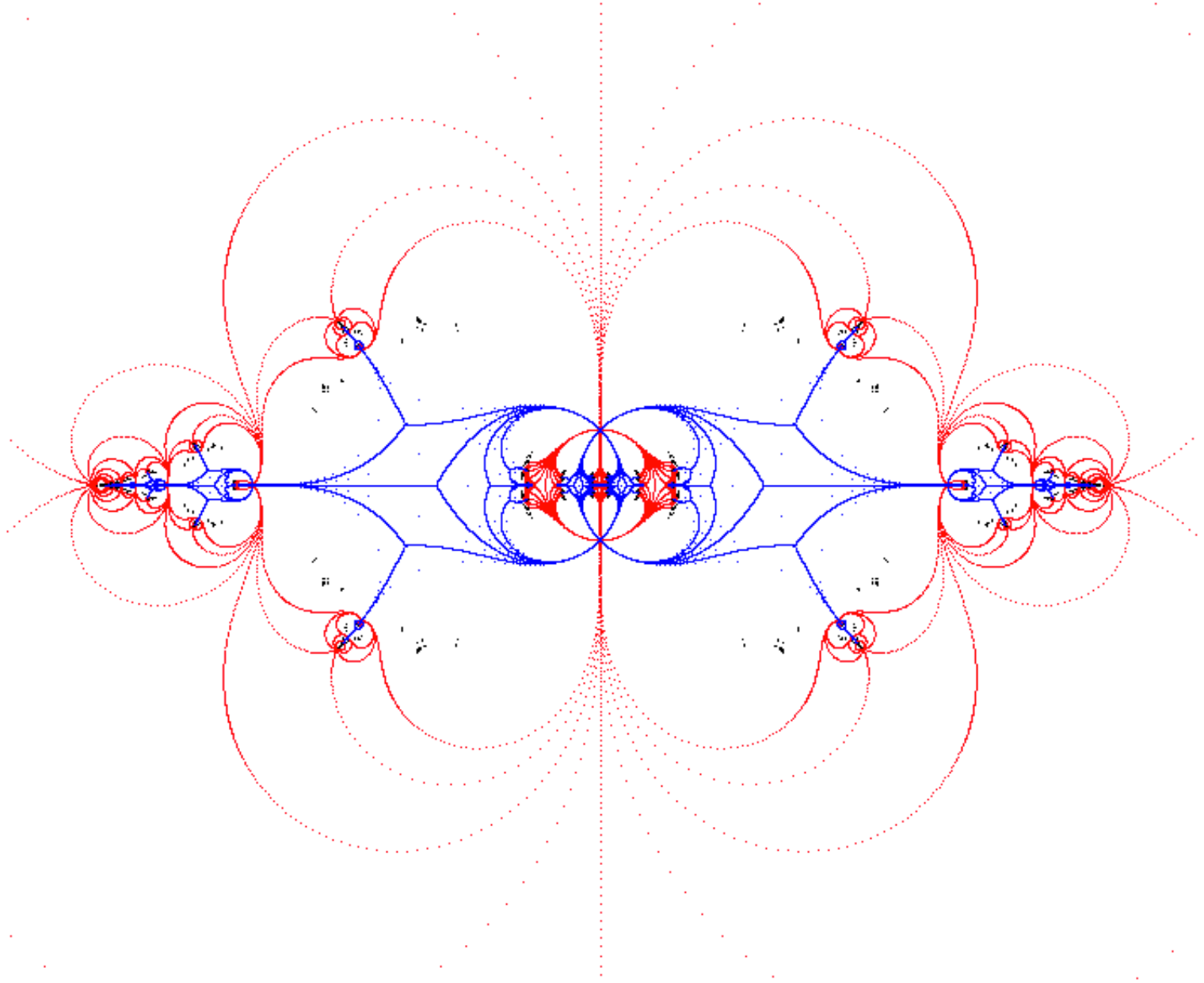} \hspace{1cm}
\includegraphics[width=4cm]{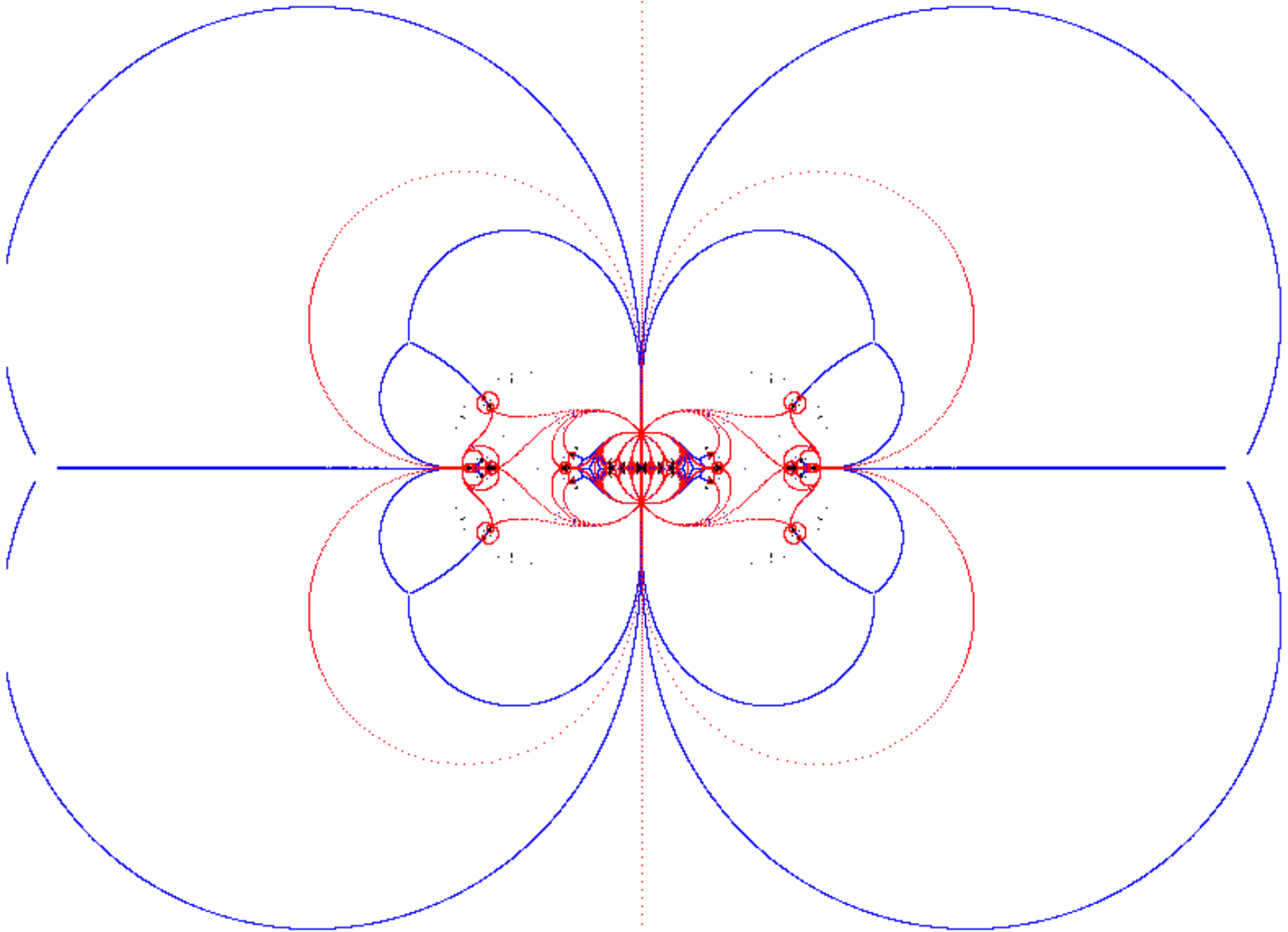}
\end{center} 
\caption{Isolated discrete $\F_a$ outside $\partial\K$ for rotation numbers $1/2,1/3,1/4$ and $1/5$ (top line: $a=-2.464$, $a=-1.492$; bottom line: $a=-1.255$ and $a=-1.157$; scales are adjusted to show the main features of the dynamics).}\label{isolated}
\end{figure}

Figure \ref{isolated} exhibits plots of the dynamics of $\F_{c_q}$ for $q=2,3,4$ and $5$. These are displayed in the $z$-coordinate (so $C_J$ is the imaginary axis, and $J$ is the involution $z\leftrightarrow -z$). The curves plotted are the grand orbits of $C_Q$ (in blue), and of the circle $C_J$ (in red), together with the grand orbit of a point on the limit set (in black). The iteration is only plotted to a limited depth, but the two fixed points in each plot are visibly the centres of Leau-Fatou parabolic flowers which have $2q$ petals, with the limit set approaching the centre of the flower along the mid-line of each petal. The plots suggest that dynamics of each $\F_{c_q}$ on the complement of its limit set is discrete, in the sense of the existence of fundamental domains. In Chapter 5 of his thesis \cite{C}, Andrew Curtis proves this in the case of $c_2$, by identifying a fundamental domain with two components, one a region bounded by segments of $C_J$ and $J(C_Q)$ between the two fixed points, and the other a region bounded by segments of $C_Q$ and $Cov_0^Q(C_J)$ between the fixed points. For $q>2$ he proposes fundamental domains for $\F_{c_q}$ with larger numbers of components bounded by curves between the fixed points, and he provides strong supporting evidence.

The mechanism behind the appearance of this sequence of parabolic examples is the fixed point bifurcation of $\F_a$ which occurs at $a=-1$ as $a$ travels along the real axis from right to left. For $a>-1$, $\F_a$ has a pair of attracting and repelling fixed points of on the real $Z$-axis: at $a=-1$ these collide to form a single (parabolic) point, and for $a<-1$ this then splits into a complex conjugate pair of neutral fixed points of increasing argument, which yield discrete correspondences when and only when the argument is $2\pi/q$ for an integer $q>2$. 

This is reminiscent of an analogous phenomenon in the moduli space of Fuchsian representations of the free product $C_2*C_3$ of cyclic groups of orders $2$ and $3$. Let $\sigma$ denote the rotation through $\pi$ of the Poincar\'e disc $\D$ about its centre $O$, and $\rho$ denote the elliptic isometry of $\D$ of order $3$ which fixes a point $P$ on the horizontal diameter of $\D$ somewhere to the right of $O$. Let $G<PSL(2,\R)$ be the group of isometries generated by $\sigma$ and $\rho$. When the distance $OP$ (in the hyperbolic metric) is large, $G$ is a faithful discrete representation of $C_2*C_3$, the element $\sigma\rho$ has attracting and repelling fixed points on the boundary circle $C$ of $\D$ and the limit set of $G$ is a Cantor set contained in $C$. As $d$ decreases, the gap between the attracting and repelling fixed points decreases and at a certain value of $d$ they collide to become a single parabolic fixed point and we see the familiar parabolic representation of $C_2*C_2$ (conjugate to $PSL(2,Z)$ in the half-plane model). When $d$ decreases further, the fixed point of $\sigma\rho$ splits into a pair of elliptic fixed points, which move off $C$ perpendicularly (we see both if we are viewing $G$ as a subgroup of $PSL(2,\C)$ but only the fixed point inside $\D$ when we work in $PSL(2,\R)$). The group $G$ is now discrete when and only when $\sigma\rho$ turns through an angle of form $2\pi/q$ about its fixed point; it is then the familiar $(2,3,q)$ triangle group, and though discrete, it is no longer a {\it faithful} representation of $C_2*C_3$, as $(\sigma\rho)^q$ is the identity.

The moduli space of representations of $C_2*C_3$ in $PSL(2,\C)$ has `faithful discreteness' locus a once-punctured topological disc $\mathcal D$, which has the boundary resembling that of a cauliflower if we parametrise the representations by the cross-ratio of the fixed points of $\sigma$ with those of $\rho$. The dense set of cusps in $\partial D$ are parametrised by rationals $p/q$ (mod $1$): each arises when a certain word in $\sigma,\rho,\rho^{-1}$ is parabolic, and we see the same phenomenon of a harmonic sequence of isolated non-faithful discrete representations outside $\mathcal D$ converging to each of these cusps.

Our correspondences $\F_a$, for the parameters $a_n$ and $a_n'$, $n>1$, on the boundary of the `discreteness locus' $\K$, are conjectured to have parabolic cyclic orbits (Conjecture \ref{conj_an}). We expect each of these cycles to give rise to a `harmonic' sequence of isolated `discrete' correspondences outside $\partial\K$. 


\end{document}